\documentclass[12pt,reqno]{amsart}
\usepackage{amsmath, amsthm, amssymb, mathrsfs}

\advance \topmargin by -\headheight
\advance \topmargin by -\headsep
     
\evensidemargin \oddsidemargin
\newtheorem{theorem}{Theorem}[section]
\newtheorem{lemma}[theorem]{Lemma}
\newtheorem{corollary}[theorem]{Corollary}

\theoremstyle{definition}

\theoremstyle{remark}

\numberwithin{equation}{section}

\newcommand{\mmod}[1]{\,\,(\text{mod}\,\,#1)}

\def\bfa{{\mathbf a}}
\def\bfb{{\mathbf b}}
\def\bfc{{\mathbf c}}
\def\bfd{{\mathbf d}} \def\bfe{{\mathbf e}} \def\bff{{\mathbf f}} \def\bfF{{\mathbf F}}
 \def\bfG{{\mathbf G}}
\def\bfh{{\mathbf h}} \def\bfH{{\mathbf H}}
\def\bfi{{\mathbf i}}
\def\bfj{{\mathbf j}}  \def\bfl{{\mathbf l}}
\def\bfm{{\mathbf m}}
\def\bfn{{\mathbf n}}

\def\bft{{\mathbf t}}
\def\bfu{{\mathbf u}} \def\vbfu{\overline{\bfu}}
\def\bfv{{\mathbf v}} \def\vbfv{\overline{\bfv}}
\def\bfw{{\mathbf w}} \def\vbfw{\overline{\bfw}}
\def\bfx{{\mathbf x}} \def\vbfx{\overline{\bfx}} \def\bfX{{\mathbf X}}
\def\bfy{{\mathbf y}} \def\vbfy{\overline{\bfy}}
\def\bfz{{\mathbf z}} \def\vbfz{\overline{\bfz}}

\def\calA{{\mathscr A}}  

\def\calB{{\mathscr B}} 
\def\calC{{\mathscr C}} 

\def\calD{{\mathscr D}}

\def\calF{{\mathscr F}}

\def\calH{{\mathscr H}}
\def\calI{{\mathscr I}}
\def\calJ{{\mathscr J}}
\def\calL{{\mathscr L}}
 
\def\calM{{\mathscr M}}
\def\calN{{\mathscr N}}

\def\calP{{\mathscr P}}

\def\calR{{\mathscr R}} 
\def\calS{{\mathscr S}}
\def\calT{{\mathscr T}}

\def\dbC{{\mathbb C}}\def\dbF{{\mathbb F}}\def\dbN{{\mathbb N}}
\def\dbR{{\mathbb R}}
\def\dbZ{{\mathbb Z}}\def\dbQ{{\mathbb Q}}

\def\grA{{\mathfrak A}}
\def\grB{{\mathfrak B}}
\def\grC{{\mathfrak C}}

\def\grf{{\mathfrak f}}\def\grF{{\mathfrak F}}
\def\grG{{\mathfrak G}}

\def\grJ{{\mathfrak J}}

\def\grm{{\mathfrak m}}\def\grM{{\mathfrak M}}
\def\grS{{\mathfrak S}}
\def\grB{{\mathfrak B}}\def\grC{{\mathfrak C}}

\def\alp{{\alpha}} \def\bfalp{{\boldsymbol \alpha}}
\def\bet{{\beta}}  \def\bfbet{{\boldsymbol \beta}}
\def\gam{{\gamma}} \def\bfgam{{\boldsymbol \gamma}} \def\vbfgam{{\overline \bfgam}} 
\def\del{{\delta}} \def\Del{{\Delta}}
\def\zet{{\zeta}} \def\bfzet{{\boldsymbol \zeta}} \def\vbfzet{{\overline {\boldsymbol \zeta}}}
\def\bfeta{{\boldsymbol \eta}}
\def\tet{{\theta}} \def\bftet{{\boldsymbol \theta}} \def\vbftet{\overline{\bftet}} 
\def\kap{{\kappa}}
\def\lam{{\lambda}}  
\def\bfmu{{\boldsymbol \mu}} \def\vbfmu{\overline {\bfmu}}
\def\bfnu{{\boldsymbol \nu}} \def\vbfnu{\overline {\bfnu}}
\def\bfxi{{\boldsymbol \xi}} \def\vbfxi{\overline {\bfxi}}

\def\sig{{\sigma}} \def\Sig{{\Sigma}} \def\bfsig{{\boldsymbol \sig}}
\def\bftau{{\boldsymbol \tau}}
\def\Ups{{\Upsilon}} 
\def\ome{{\omega}} \def\Ome{{\Omega}}
\def\d{{\partial}}
\def\eps{\varepsilon}

\def\le{\leqslant} \def\ge{\geqslant}

\def\d{{\,{\rm d}}}
\def\llbracket{\lbrack\;\!\!\lbrack} \def\rrbracket{\rbrack\;\!\!\rbrack}

\begin{document}
\title[Multidimensional Weyl sums]{Near-optimal mean value estimates\\ for multidimensional Weyl sums}
\author[S. T. Parsell]{S. T. Parsell}
\address{STP: Department of Mathematics, West Chester University, 25 University Ave., West Chester, PA 19383, U.S.A.}
\email{sparsell@wcupa.edu}
\author[S. M. Prendiville]{S. M. Prendiville}
\address{SMP:\,  School of Mathematics, University of Bristol, University Walk, Clifton, Bristol BS8 1TW, United Kingdom}
\email{sean.prendiville@gmail.com}
\author[T. D. Wooley]{T. D. Wooley$^*$}
\address{TDW: School of Mathematics, University of Bristol, University Walk, Clifton, Bristol BS8 1TW, United Kingdom}
\email{matdw@bristol.ac.uk}
\thanks{STP was supported by National Security Agency Grant H98230-11-1-0190, SMP by an EPSRC doctoral training grant through the University of Bristol, and TDW by a Royal Society Wolfson Research Merit Award.}
\date{}
\begin{abstract} We obtain sharp estimates for multidimensional generalisations of Vinogradov's mean value theorem for arbitrary translation-dilation invariant systems, achieving constraints on the number of variables approaching those conjectured to be the best possible. Several applications of our bounds are discussed.\end{abstract}
\maketitle

\section{Introduction}\label{sec:1} The investigation of Diophantine problems of large degree is in general fraught with difficulties only partially mollified by the presence of intrinsic diagonal structure. Indeed, such analyses as are made available via the Hardy-Littlewood (circle) method, when successful, involve complicated exponential sum estimates widely considered to be amongst the most challenging in the subject. The application of Weyl differencing by Davenport, Birch and Schmidt, on the one hand, involves a delicate interplay between the singular locus associated with the problem and the quality of ensuing exponential sum estimates (see \cite{Bir1961}, \cite{Dav1963}, \cite{Sch1985}). As an inescapable feature of such approaches, the number of variables required in a successful treatment grows exponentially with the degree of the problem at hand. The extension of Vinogradov's methods to exponential sums in many variables, on the other hand, is notoriously complicated. The work of Arkhipov, Karatsuba and Chubarikov \cite{AKC2004}, \cite{AKC1980}, for example, permits substantially sharper conclusions to be drawn when partial diagonal structure is present. However, the complexity of the underlying methods has deterred a consideration of all but the simplest model situations (see \cite{AKC1980} and \cite{Par2005}). In addition, the available conclusions fail to achieve their conjectured potential by a factor growing roughly like the logarithm of the total degree of the associated translation-invariant Diophantine system.\par

Our goal in this paper is to extend the efficient congruencing method introduced by the third author \cite{Woo2012a} so as to accommodate the generalised Vinogradov systems of Arkhipov, Karatsuba and Chubarikov (see \cite{AKC2004}, \cite{AKC1980}). It transpires that for systems of large degree, the bounds that we thereby derive miss those conjectured to hold by a factor of only $2$ or thereabouts, transforming the previous state of the art. Moreover, our methods are of such flexibility that they may be successfully applied to translation-invariant systems of wide generality, and in particular to systems closely related to those subject to recent investigations by quantitative arithmetic geometers studying the Manin-Peyre conjectures (see \cite[\S4.15]{Tsch2009}, \cite{KVV2011}). Since our methods yield estimates no less striking for such systems, we take the opportunity to derive rather general estimates again coming within a constant factor of those conjectured to hold. There are consequences of all of this work for exponential sum estimates of Weyl-type, for the solubility of systems of Diophantine equations and related problems, and for certain problems in additive combinatorics, and these we also explore herein.\par

Rather than encumber the reader at this point with the substantial notational prerequisites entailed by a discussion of our most general conclusions, we instead offer the more easily digestible corollaries particular to the model problem considered in earlier work \cite{Par2005} of the first author. Let $s$, $k$ and $d$ be natural numbers, and let $X$ be a positive number. We focus attention on the system of simultaneous Diophantine equations
\begin{equation}\label{1.1}
\sum_{j=1}^sx_{j1}^{i_1}x_{j2}^{i_2}\cdots x_{jd}^{i_d}=\sum_{j=1}^sy_{j1}^{i_1}y_{j2}^{i_2}\cdots y_{jd}^{i_d}\quad (1\le i_1+\ldots +i_d\le k).
\end{equation}
Here, the indices $i_m$ are non-negative integers, so that a modest computation reveals the total number of equations in the system (\ref{1.1}) to be $r=r_{d,k}$, where
\begin{equation}\label{1.2}
r_{d,k}=\binom{k+d}{d}-1.
\end{equation}
Meanwhile, the total degree of the system (\ref{1.1}), which is to say the sum of the degrees of the equations comprising the system, is equal to $K=K_{d,k}$, where
\begin{equation}\label{1.3}
K_{d,k}=\sum_{l=1}^kl\binom{l+d-1}{l}=\frac{d}{d+1}(r+1)k.
\end{equation}
In particular, in the familiar classical Vinogradov system with $d=1$, one has $r=k$ and $K=\frac{1}{2}k(k+1)$. Finally, we write $J_{s,k,d}(X)$ for the number of integral solutions of the system (\ref{1.1}) with $1\le x_{jm},y_{jm}\le X$ $(1\le j\le s,\,1\le m\le d)$.\par

In \S\ref{sec:2}, as a special case of Theorem \ref{theorem2.1}, we derive an estimate for $J_{s,k,d}(X)$ that is in many respects close to the best possible. Here and throughout, implicit constants in Vinogradov's notation $\ll$ and $\gg$ depend at most on $s$, $k$, $d$ and $\eps$, unless otherwise indicated. The letter $X$ should be interpreted as a positive number sufficiently large in terms of $s$, $k$, $d$ and $\eps$.

\begin{theorem}\label{theorem1.1} Suppose that $s$, $k$ and $d$ are natural numbers with $k\ge 2$ and $s\ge r(k+1)$. Then for each $\eps>0$, one has $J_{s,k,d}(X)\ll X^{2sd-K+\eps}$.
\end{theorem}

The special case of Theorem \ref{theorem1.1} with $d=1$ is equivalent to \cite[Theorem 1.1]{Woo2012a}, a conclusion which has very recently been sharpened in \cite[Theorem 1.1]{Woo2012b}, so that for $k\ge 3$ and $s\ge k^2-1$, one has
$$J_{s,k,1}(X)\ll X^{2s-\frac{1}{2}k(k+1)+\eps}.$$
When $d>1$, meanwhile, one may compare the conclusion of Theorem \ref{theorem1.1} with work of Arkhipov, Karatsuba and Chubarikov \cite{AKC1980}. Although set up slightly differently, it is clear that the methods of the latter authors have the potential to establish a bound of the shape
$$J_{s,k,d}(X)\ll X^{2sd-K+\Del_s},$$
where $\Del_s$ decays with $s$ roughly like $rke^{-s/(rk)}$. In \cite[Theorem 1.1]{Par2005} this conclusion was improved by the first author for sufficiently large values of $k$, so that on writing
$$s_0={\textstyle{\frac{1}{2}}}rk(\log k-2\log \log k),$$
one may take
$$\Del_s=\begin{cases}rke^{2-2s/(rk)},&\text{for $1\le s\le s_0$,}\\
r(\log k)^2e^{3-3(s-s_0)/(2rk)},&\text{for $s>s_0$.}
\end{cases}$$
The estimate supplied by Theorem \ref{theorem1.1} is substantially sharper. Thus, provided only that $s\ge r(k+1)$, one may take $\Del_s=\eps$ for any positive number $\eps$. The number of variables required in typical applications, as we discuss in due course, is thereby reduced by a factor of order $\log (rk)$.\par

In order to discern the strength of the estimate supplied by Theorem \ref{theorem1.1}, we must consider available lower bounds for the mean value $J_{s,k,d}(X)$, and thereby infer plausible conjectures for corresponding upper bounds. In \S\ref{sec:3} we establish the lower bound for $J_{s,k,d}(X)$ contained in the following theorem.

\begin{theorem}\label{theorem1.2} Suppose that $s$, $k$ and $d$ are natural numbers. Then one has
$$J_{s,k,d}(X)\gg X^{sd}+\sum_{j=1}^dX^{(2s-1)j+d-K_{j,k}}.$$
\end{theorem}

It seems reasonable to conjecture that whenever $\eps>0$, one has the allied upper bound
$$J_{s,k,d}(X)\ll X^\eps \Bigl(X^{sd}+\sum_{j=1}^dX^{(2s-1)j+d-K_{j,k}}\Bigr).$$
Thus, when $s$ is sufficiently large in terms of $k$ and $d$, one expects that
\begin{equation}\label{1.4}
J_{s,k,d}(X)\ll X^{2sd-K+\eps},
\end{equation}
as is confirmed by Theorem \ref{theorem1.1} for $s\ge r(k+1)$. We emphasise that here, and throughout the introduction, we abbreviate $r_{d,k}$ to $r$ and $K_{d,k}$ to $K$. As a consequence of Theorem \ref{theorem1.2}, we show in Theorem \ref{theorem3.2} that when $\del$ is a real number with $2d/k<\del<1$, and
$$s\le \frac{d}{2d+2}(1-\del) r(k+1),$$
then there is a positive number $\eta=\eta(d,k)$ with the property that
$$J_{s,k,d}(X)\gg X^{2sd-K+\eta}.$$
When $d\ge 2$, therefore, it follows that the conclusion of Theorem \ref{theorem1.1} comes within a factor $2+2/d+O(d/k)$ of the least value of $s$ for which the conjectured upper bound (\ref{1.4}) might conceivably hold. In such multidimensional Weyl sums, a near-optimal conclusion of this type, merely a constant factor away from the best possible, has hitherto been wholly beyond our grasp.\par

We next consider upper bounds for exponential sums of Weyl-type, the discussion of which is much facilitated by the introduction of additional notation. It is convenient to abbreviate a monomial of the shape $x_1^{i_1}x_2^{i_2}\cdots x_d^{i_d}$ to $\bfx^\bfi$, in which $\bfi=(i_1,i_2,\ldots ,i_d)$. Likewise, we may write $\bfx_m^\bfi$ for $x_{m1}^{i_1}x_{m2}^{i_2}\cdots x_{md}^{i_d}$. In such circumstances, we put
$$|\bfi|=i_1+\ldots +i_d.$$
Also, in place of the $s$-tuple $(\bfx_1,\ldots ,\bfx_s)$ we write $\vbfx$, and we adopt the convention that $a\le \bfv\le b$ is to mean that each coordinate $v_l$ of the vector $\bfv$ satisfies $a\le v_l\le b$. Equipped with these conventions, the Diophantine system (\ref{1.1}) assumes the compact shape
$$\bfx_1^\bfi+\ldots +\bfx_s^\bfi=\bfy_1^\bfi+\ldots+\bfy_s^\bfi\quad (1\le |\bfi|\le k),$$
and $J_{s,k,d}(X)$ counts the number of integral solutions of this system with $1\le \vbfx,\vbfy\le X$.\par

Define the exponential sum $f(\bfalp)=f_{d,k}(\bfalp;X)$ by
$$f_{d,k}(\bfalp;X)=\sum_{1\le \bfx\le X}e(\psi(\bfx;\bfalp)),$$
where
$$\psi(\bfx;\bfalp)=\sum_{1\le |\bfi|\le k}\alp_\bfi\bfx^\bfi,$$
and, as usual, we write $e(z)$ for $e^{2\pi iz}$. Here, the subscript $d$ that identifies $\bfx$ as the $d$-tuple $(x_1,\ldots ,x_d)$ may usually be omitted without leading to confusion. As we have already noted, the number of coefficients $\alp_\bfi$ is $r$. We adopt the convention that, when $G:[0,1)^n\rightarrow \dbC$ is measurable, then
$$\oint G(\bfbet)\d\bfbet =\int_{[0,1)^n}G(\bfbet)\d\bfbet .$$
It then follows from orthogonality that
\begin{equation}\label{1.5}
J_{s,k,d}(X)=\oint|f_{d,k}(\bfalp;X)|^{2s}\d\bfalp .
\end{equation}

\par Upper bounds for mean values of exponential sums such as $f(\bfalp)$ may be converted into Weyl-type estimates by means of variants of the large sieve inequality. Before announcing such an estimate, which is a consequence of the more general result recorded in Theorem \ref{theorem10.3}, we pause to record a further notational convention. When $\bfa\in \dbZ^n$, we write $(q,\bfa)$ for the greatest common divisor $(q,a_1,\ldots ,a_n)$.

\begin{theorem}\label{theorem1.3}
Suppose that $d$ and $k$ are natural numbers with $k\ge 2$. Let $\sig$ be any real number with $$\sig^{-1}\ge \left(2k\binom{k+d-1}{d}-2k+1\right)(d+1).$$
Then whenever $|f_{d,k}(\bfalp;X)|\ge X^{d-\sig +\eps}$, for some $\eps>0$, it follows that there exist $q\in \dbN$ and $a_\bfj\in \dbZ$ $(1\le |\bfj|\le k)$ satisfying
$$(q,\bfa)=1,\quad 1\le q\le X^{k\sig}\quad \text{and}\quad |q\alp_\bfj-a_\bfj|\le X^{k\sig -|\bfj|}\quad (1\le |\bfj|\le k).$$
\end{theorem}

The special case of Theorem \ref{theorem1.3} with $d=1$ is very slightly weaker than \cite[Theorem 1.6]{Woo2012a}, a conclusion which has recently been sharpened in \cite[Theorem 11.2]{Woo2012b}. Thus, when $d=1$, the conclusion of Theorem \ref{theorem1.3} holds whenever $k\ge 4$ and $\sig^{-1}\ge 4k(k-2)$. The work of Arkhipov, Karatsuba and Chubarikov \cite{AKC1980}, as interpreted and sharpened by the first author, yields a conclusion similar to Theorem \ref{theorem1.3}. Indeed, it follows from a corrected version\footnote{See the discussion following the proof of Theorem \ref{theorem10.2} below for an explanation of the need for a modest correction in \cite[Theorem 1.2]{Par2005}.} of \cite[Theorem 1.2]{Par2005} that a conclusion of similar form holds for sufficiently large values of $k$, though with the constraint on the exponent $\sig$ replaced by a condition of the shape $\sig^{-1}\ge \frac{4}{3}(d+1)rk\log (rk)$. On noting that $\binom{k+d-1}{d}-1=r_{d,k-1}\le r_{d,k}$, the superiority of our new bound is clear.\par

We next consider the application of our new estimates to Diophantine problems. When $s$, $k$ and $d$ are natural numbers, and $a_{\bfi j}$ is a non-zero integer for $1\le |\bfi|\le k$ and $1\le j\le s$, write
$$\phi_\bfi (\vbfx)=\sum_{j=1}^sa_{\bfi j}\bfx_j^\bfi\quad (1\le |\bfi|\le k).$$
In \S\ref{sec:11}, we consider the Diophantine system
\begin{equation}\label{1.6}
\phi_\bfi(\vbfx)=0\quad (1\le |\bfi|\le k),
\end{equation}
consisting of $r$ equations of total degree $K$. Let $N(B)$ denote the number of integral solutions of the system (\ref{1.6}) with $|\vbfx|\le B$. We follow Schmidt \cite{Sch1985} when defining the (formal) real and $p$-adic densities associated with the system (\ref{1.6}). When $L>0$, define
$$\lam_L(\eta)=\begin{cases} L(1-L|\eta|),&\text{when $|\eta|\le L^{-1}$,}\\
0,&\text{otherwise,}\end{cases}$$
and put
$$\mu_L=\int_{|\vbfxi|\le 1}\prod_{1\le |\bfi|\le k}\lam_L(\phi_\bfi (\vbfxi))\d\vbfxi .$$
The limit $\sig_\infty={\displaystyle{\lim_{L\rightarrow \infty}}}\mu_L$, when it exists, is called the {\it real density}. Meanwhile, given a natural number $q$, we write
$$M(q)=\text{card} \left\{ \vbfx\in (\dbZ/q\dbZ)^{sd}:\phi_\bfi(\vbfx)\equiv 0\mmod{q}\ (1\le |\bfi|\le k)\right\} .$$
For each prime number $p$, we then put
$$\sig_p=\lim_{h\rightarrow \infty}p^{h(r-sd)}M(p^h),$$
provided that this limit exists, and refer to $\sig_p$ as the {\it $p$-adic density}.\par

As a special case of Theorem \ref{theorem11.1}, we establish an asymptotic formula for $N(B)$ valid whenever $s\ge 2r(k+1)+1$.

\begin{theorem}\label{theorem1.4} Suppose that $s$, $k$ and $d$ are natural numbers with $k\ge 2$ and $s\ge 2r(k+1)+1$. In addition, let $a_{\bfi j}$ $(1\le |\bfi|\le k, 1\le j\le s)$ be non-zero integers. Then provided that the system of equations (\ref{1.6}) possesses non-singular real and $p$-adic solutions for each prime number $p$, one has
\begin{equation}\label{1.7}
N(B)\sim \sig_\infty \left( \prod_p \sig_p\right) B^{sd-K}.
\end{equation}
In particular, the system (\ref{1.6}) satisfies the Hasse Principle.
\end{theorem}

We note that \cite[Theorem 9.1]{Woo2012a} delivers the same conclusion as Theorem \ref{theorem1.4} in the special case $d=1$. As is apparent from the lower bound supplied by Theorem \ref{theorem1.2} and the ensuing discussion, there exist choices of coefficients $\bfa$ for which the asymptotic formula (\ref{1.7}) necessarily fails when $k$ is large and
\begin{equation}\label{1.8}
s\le \frac{d}{d+1}(1+O(d/k))r(k+1).
\end{equation}
Consequently, the bound on the number of variables in the hypotheses of Theorem \ref{theorem1.4} is within a factor $2+2/d+O(d/k)$ of the best possible bound for such systems. Indeed, the argument underlying the proof of Theorem \ref{theorem1.2} shows that such remains true in wider generality. The point here is that special subvarieties contain the bulk of the set of integral solutions whenever the bound (\ref{1.8}) holds on the number $s$ of blocks of $d$ variables. Conclusions available hitherto of the type presented in Theorem \ref{theorem1.4} impose bounds on the number of blocks of variables weaker than our own by a factor of order $\log (rk)$.\par

As a special case of Corollary \ref{corollary11.2}, we obtain an asymptotic formula for the mean value $J_{s,k,d}(X)$.

\begin{theorem}\label{theorem1.5} Let $k$ and $d$ be natural numbers with $k\ge 2$. Then whenever $s\ge r(k+1)+1$, there exist positive constants $\calC=\calC(s,k,d)$ and $\del=\del(k,d)$ such that
$$J_{s,k,d}(X)=\calC X^{2sd-K}+O(X^{2sd-K-\del}).$$
\end{theorem}

A conclusion analogous to that of Theorem \ref{theorem1.5} is obtained in \cite[Theorem 1.3]{Par2005}, subject to the condition that $k$ be sufficiently large and
$$s\ge rk({\textstyle{\frac{2}{3}}}\log r+{\textstyle{\frac{1}{2}}}\log k+\log \log k+2d+4).$$
Again, the conclusion of Theorem \ref{theorem1.5} is much superior. When $d=1$ and $k\ge 3$, meanwhile, the conclusion of Theorem \ref{theorem1.5} is a consequence of \cite[Theorem 1.2]{Woo2012a}.\par

As a penultimate application of the bounds supplied by Theorem \ref{theorem1.1}, in \S\ref{sec:11} we consider the rational linear spaces of projective dimension $d-1$ lying on the diagonal hypersurface
\begin{equation}\label{1.9}
c_1z_1^k+\ldots +c_sz_s^k=0,
\end{equation}
with $c_1,\ldots ,c_s$ fixed non-zero integers. Such a linear space may be written in the form
$$\calL(\bfx_1,\ldots ,\bfx_d)=\{ t_1\bfx_1+\ldots +t_d\bfx_d:t_1,\ldots ,t_d\in \dbQ\},$$
for suitable linearly independent vectors $\bfx_1,\ldots ,\bfx_d\in \dbZ^s$. As noted in \cite{Par2005}, by substituting into (\ref{1.9}) and using the multinomial theorem to collect together coefficients of $\bft^\bfi$, one finds that the linear space $\calL(\vbfx)$ corresponds to a solution $\bfy_1,\ldots ,\bfy_s\in \dbZ^d$ of the Diophantine system
\begin{equation}\label{1.10}
c_1\bfy_1^\bfi+\ldots +c_s\bfy_s^\bfi=0\quad (|\bfi|=k).
\end{equation}
This correspondence is made explicit by means of the simple relation
$$x_{ij}=y_{ji}\quad (1\le i\le d,\, 1\le j\le s).$$
Write $N_{s,k,d}(X)$ for the number of integral solutions of the system (\ref{1.10}) with $|\vbfy|\le X$, and put
$$L=k\binom{k+d-1}{k}.$$

\par In \S11 we indicate how to prove an asymptotic formula for $N_{s,k,d}(X)$ subject to the condition that $s\ge 2r(k+1)+1$. In this context, we say that the integral $s$-tuple $\bfc$ is a {\it non-singular choice of coefficients for $k$ and $d$} when the system of equations (\ref{1.10}) has non-singular real and $p$-adic solutions, for every prime number $p$.

\begin{theorem}\label{theorem1.6}
Suppose that $s$, $k$ and $d$ are natural numbers with $k\ge 2$ and $s\ge 2r(k+1)+1$. Suppose further that $\bfc\in (\dbZ\setminus\{ {\mathbf 0}\})^s$ is a non-singular choice of coefficients for $k$ and $d$. Then there exist positive constants $\calD=\calD(s,k,d;\bfc)$ and $\nu=\nu(k,d)$ such that
$$N_{s,k,d}(X)=\calD X^{sd-L}+O(X^{sd-L-\nu}).$$
\end{theorem}
 
In particular, one finds that whenever $s\ge 2r(k+1)+1$ and appropriate local solubility conditions are met, then the hypersurface defined by (\ref{1.9}) contains an abundance of rational linear spaces of projective dimension $d-1$. A perusal of \cite{Par2005} reveals that, for sufficiently large values of $k$, a similar conclusion is asserted by Theorem \ref{theorem1.4} of the latter source, subject instead to the more stringent condition
$$s\ge rk({\textstyle{\frac{4}{3}}}\log r+\log k+2\log \log k+4d+8).$$
We remark that the lower bound $s\ge 2r(k+1)+1$ in Theorem \ref{theorem1.6} should be susceptible to some small improvement by adapting the methods of \cite{Woo2012} to the present multidimensional setting. Moreover, when the degree $k$ is very small, an approach of the first author \cite{Par2012} motivated by a method of Hua proves superior in some situations.\par

Further applications within the orbit of our methods and bounds include the generalised Waring problem of representing a given polynomial
$$\Phi(t_1,\ldots ,t_d)\in \dbZ[\bft]$$
in the form
$$\Phi(\bft)=\sum_{j=1}^s(x_{1j}t_1+\ldots +x_{dj}t_d)^k,$$
and also results concerning the number of integral solutions of Diophantine inequalities modulo $1$. We refer the reader to \cite{AKC1980} for a discussion of some such problems, and leave to the reader the satisfaction of incorporating our new bounds into the established methods so as to make similarly striking improvements over the previous state of knowledge.\par

As a final application of our new bounds for multidimensional Weyl sums, we announce an application in additive combinatorics based on the second author's recent work \cite{Pre2011} on translation invariant systems of equations devoid of solutions in multidimensional sets. In Theorem \ref{theorem11.3} below we present a conclusion more general than the one we presently record in Theorem \ref{theorem1.7}. For the purpose at hand, we describe the integral $s$-tuple $\bfc$ as an {\it extended non-singular choice of coefficients for $k$ and $d$} when (i) one has $c_1+\ldots +c_s=0$, and (ii) the system of equations
\begin{equation}\label{1.11}
c_1\bfy_1^\bfi+\ldots +c_s\bfy_s^\bfi=0\quad (1\le |\bfi|\le k)
\end{equation}
has non-singular real and $p$-adic solutions, for every prime number $p$.\par

Certain solutions of the system (\ref{1.11}) are atypically simple to obtain, such as the trivial solutions lying on the diagonal $\bfy_1=\bfy_2=\ldots =\bfy_s$. We formalise this notion by distinguishing two types of special solutions $\vbfy$ of (\ref{1.11}). We describe $\vbfy$ as {\it projected} when there is a translate of a proper subspace of $\dbQ^d$ that contains all of $\bfy_1,\ldots ,\bfy_s$. The aforementioned diagonal solutions are therefore projected, since they lie in a translate of the trivial subspace $\{{\mathbf 0}\}$ of $\dbQ^d$. Also, we say that $\vbfy$ is a {\it subset-sum solution} when there exists a partition $\{1,2,\ldots ,s\}=\calJ_1\cup\ldots \cup \calJ_l$, into $l\ge 2$ disjoint non-empty sets $\calJ_v$, such that for $1\le v\le l$ one has
$$\sum_{u\in \calJ_v}c_u\bfy_u^\bfi=0\quad (1\le |\bfi|\le k).$$
In the special case in which
$$\sum_{u\in \calJ_v}c_u=0\quad (1\le v\le l),$$
one sees that there are trivial subset-sum solutions in which $\bfy_u=\bfy_w$ whenever $u\in \calJ_v$ and $w\in \calJ_v$ $(1\le v\le l)$. 

\begin{theorem}\label{theorem1.7} Suppose that $s$, $k$ and $d$ are natural numbers with $k\ge 2$ and $s\ge 2r(k+1)+1$. Suppose further that $\bfc\in (\dbZ\setminus \{{\mathbf 0}\})^s$ is an extended non-singular choice of coefficients for $k$ and $d$, so that $c_1+\ldots +c_s=0$. Let $\calA$ be a subset of $\dbZ^d\cap [1,N]^d$, and suppose that the only solutions of the system (\ref{1.11}) from $\calA$ are either projected or subset-sum solutions. Then one has
$${\rm card}(\calA)\ll N^d(\log \log N)^{-1/(s-1)}.$$
\end{theorem} 

This theorem is a higher dimensional cousin of \cite[Theorem 5.1]{Pre2011}, which supplies an analogous conclusion for a case involving binary forms. Theorem \ref{theorem1.7} shows that when $\text{card}(\calA)$ grows more rapidly than $N^d(\log \log N)^{-1/(s-1)}$, then the system (\ref{1.11}) contains solutions from $\calA$ besides such obvious ones as the diagonal solutions with $\bfy_1=\ldots =\bfy_s$. As we see in \S3, the extended system (\ref{1.11}) contains more general special subvarieties defined by means of a projection process, the simplest of which set one or more variables to be zero. In Theorem \ref{theorem11.3} we present a conclusion that refines Theorem \ref{theorem1.7} in which, under the same hypotheses concerning the cardinality of $\calA$, one finds that the system (\ref{1.11}) contains solutions from $\calA$ which avoid all of these special subvarieties. In this way, one may legitimately describe the solutions of (\ref{1.11}) thus shown to exist as honestly non-trivial. The interested reader will find the necessary ideas in earlier work \cite{Pre2011} of the second author.\par

It may be useful to provide an informal sketch hinting at the argument underlying the proof of Theorem \ref{theorem1.1} so that the reader is better prepared to draw parallels with previous approaches. A more comprehensively illuminated sketch of this argument in the case $d=1$ may be found in \cite[\S2]{Woo2012a}. In common with the previous approaches of \cite{AKC1980} and \cite{Par2005}, the basic tool employed in our proof of Theorem \ref{theorem1.1} is a (so-called) $p$-adic iteration mirroring the one devised by Linnik \cite{Lin1943} in the classical setting with $d=1$. Thus, we begin by artificially introducing a congruence condition, modulo a suitable prime $p$, amongst the bulk of the variables underlying the mean value (\ref{1.5}). An application of H\"older's inequality leads to a new mean value in which the latter variables lie in common congruence classes across blocks. At this point, the multiple translation invariance of the system (\ref{1.1}) may be utilised so as to pass to the zero congruence class, and thereby a congruence condition is forced on a subset of the variables of greater strength than that previously introduced. The approach of \cite{AKC1980} is to choose the prime $p$ in such a way that this strong congruence condition forces a diagonal condition amongst blocks of variables, and thereby one is able to bound a mean value involving $2(s+r)$ blocks of $d$ variables in terms of a corresponding mean value involving $2s$ blocks of $d$ variables. In \cite{Par2005} the strong congruence condition is interpreted as a differencing process analogous to, though more efficient than, that of Weyl. By appropriate use of the Cauchy-Schwarz inequalities, one is able to repeat this {\it efficient differencing} process, deferring the moment at which to force the diagonal condition. In the present paper, following \cite{Woo2012a}, we instead interpret the strong congruence condition as an efficient method of imposing a second artificial congruence condition amongst variables. By appropriate application of H\"older's inequality, one recovers a new mean value resembling that obtained in the first step, but now yielding a fresh congruence condition amongst variables significantly stronger than before. If one begins with a mean value significantly larger in size than anticipated, then repeated application of this {\it efficient congruencing} procedure yields a related mean value larger in size than that anticipated by an amount so large that even a trivial estimate demonstrates the presumed initial deviation from the expected size to be untenable. In this way, one shows that the mean value under consideration has size very close to that expected.\par

We finish by emphasising that the methods of this paper are robust to changes of the ambient ring. Thus the rational integers $\dbZ$ central to this paper may be replaced with the ring of integers from a number field, or the polynomial ring $\dbF_q[t]$, without diminishing the strength of the ensuing estimates. Such ideas have been explored very recently in the case $d=1$ in work emerging from the body of research exploiting the efficient congruencing method (see \cite{LW2012} and \cite{Woo2012c}).\par

In \S\ref{sec:2} we introduce the general translation-dilation invariant systems which constitute the central objects of attention in this paper. Then, in \S\ref{sec:3}, we discuss the lower bounds recorded in Theorem \ref{theorem1.2}. The notation and infrastructure required for our most general conclusions is discussed in \S\ref{sec:4}, and then in \S\ref{sec:5} we derive the basic mean value estimates which initiate our efficient congruencing argument. Next, in \S\ref{sec:6}, we provide estimates for the number of solutions of a system of basic congruences. Here, the singular locus of the system is of particular concern. The conditioning process, required to guarantee appropriate non-singularity conditions, is engineered in \S\ref{sec:7}, and in \S\ref{sec:8} we discuss the efficient congruencing process itself. In \S\ref{sec:9} we combine the output of \S\S\ref{sec:7} and \ref{sec:8} so as to deliver Theorem \ref{theorem1.1} via an iterative process. Consequences for Weyl-type estimates are discussed in \S\ref{sec:10}, yielding the conclusion of Theorem \ref{theorem1.3}. Finally, in \S\ref{sec:11} we sketch the arguments required to establish the Diophantine consequences recorded in Theorems \ref{theorem1.4}-\ref{theorem1.7}. 

\section{Translation-dilation invariant systems}\label{sec:2} In order to describe our most general conclusions, we must introduce some notation having flexibility sufficient for our needs. An overly prescriptive approach has the potential to shroud the details of our arguments in a thick blanket of impenetrable symbols. With this undesirable potential outcome in mind, we opt for a somewhat abstract approach, and only later do we spend time detailing the most interesting situations.\par

Let $r$, $s$ and $d$ be natural numbers, and consider a system of homogeneous polynomials $\bfF=(F_1,\ldots ,F_r)$, where $F_j(\bfz)\in \dbZ[z_1,\ldots ,z_d]$ $(1\le j\le r)$. We investigate the system of Diophantine equations
\begin{equation}\label{2.1}
\sum_{i=1}^s(\bfF(\bfx_i)-\bfF(\bfy_i))={\bf 0},
\end{equation}
in which $\bfx_i=(x_{i1},\ldots ,x_{id})$ and $\bfy_i=(y_{i1},\ldots ,y_{id})$ for $1\le i\le s$. Note that, in view of our conventions concerning vector notation, the system (\ref{2.1}) consists of $r$ simultaneous Diophantine equations. Write $\vbfx=(\bfx_1,\ldots ,\bfx_s)$ and $\vbfy=(\bfy_1,\ldots ,\bfy_s)$, and denote by $J_s(X;\bfF)$ the number of integral solutions of the system (\ref{2.1}) with $1\le \vbfx,\vbfy\le X$.\par

In this paper we are concerned with translation-dilation invariant systems of the shape (\ref{2.1}). With the discussion to come in mind, we take a pragmatic approach to defining such systems. We say that the system $\bfF=(F_1,\ldots ,F_r)$ is {\it translation-dilation invariant} if:\vskip.0cm
\noindent (i) the polynomials $F_1,\ldots ,F_r$ are each homogeneous of positive degree, and\vskip.0cm
\noindent (ii) there exist polynomials
$$c_{jl}\in \dbZ[\xi_1,\ldots ,\xi_d]\quad (1\le j\le r\ \text{and}\ 0\le l\le j),$$
with $c_{jj}=1$ for $1\le j\le r$, having the property that whenever $\bfxi\in \dbZ^d$, then
\begin{equation}\label{2.2}
F_j(\bfx+\bfxi)=c_{j0}(\bfxi)+\sum_{l=1}^jc_{jl}(\bfxi)F_l(\bfx)\quad (1\le j\le r).
\end{equation}
Extend the definition of the coefficients $c_{jl}$ by putting $c_{jl}(\bfxi)=0$ when $l>j$. Then on writing $\bfc_0(\bfxi)=(c_{j0}(\bfxi))_{1\le j\le r}$ and $C(\bfxi)$ for the matrix $(c_{jl}(\bfxi))_{1\le j,l\le r}$, we see that the relations (\ref{2.2}) are summarised by the formula
\begin{equation}\label{2.3}
\bfF(\bfx+\bfxi)=C(\bfxi)\bfF(\bfx)+\bfc_0(\bfxi).
\end{equation}
Notice that the matrix $C(\bfxi)$ is lower unitriangular, which is to say that it is a lower triangular matrix whose main diagonal entries are all $1$. Suppose that $s$ is a natural number, that $\lam$ is a non-zero rational number, and $\bfxi\in \dbZ^d$. Then we see from (\ref{2.2}) that the Diophantine system (\ref{2.1}) possesses an integral solution $\bfx,\bfy$ if and only if one has
\begin{equation}\label{2.4}
\sum_{i=1}^s(\bfF(\lam \bfx_i+\bfxi)-\bfF(\lam \bfy_i+\bfxi))={\bf 0}.
\end{equation}
This observation justifies the description of such systems of equations as trans-lation-dilation invariant. We should note that while this formal definition facilitates many of our arguments, it is clear that one may rearrange the ordering of the forms, and also consider independent linear combinations of the original forms, without altering the number of integral solutions of the system (\ref{2.1}) counted by $J_s(X;\bfF)$. Thus we may be expedient in most circumstances, and instead describe a system as translation-dilation invariant when it is equivalent in such a manner to some new system which is translation-dilation invariant in the strict sense.\par

We emphasise that translation-dilation invariant systems are easily generated. Given a collection of homogeneous polynomials
$$G_1,\ldots ,G_h\in \dbZ[z_1,\ldots ,z_d],$$
consider the set $\calF$ consisting of all the partial derivatives
\begin{equation}\label{2.5}
\frac{\partial^{l_1+\ldots +l_d}G_j(\bfz)}{\partial z_1^{l_1}\ldots \partial z_d^{l_d}}\quad (1\le j\le h),
\end{equation}
with $l_i\ge 0$ $(1\le i\le d)$. Plainly, when $l_1+\ldots +l_d$ exceeds the largest total degree of any of the polynomials $G_j$, this partial derivative vanishes. The set $\calF$ is consequently finite. Let $\calF_0$ denote the subset of $\calF$ consisting of all polynomials in $\calF$ having positive degree. We write $\calF_0=\{F_1,\ldots ,F_r\}$, labelling the elements in such a way that $\deg F_1\le \deg F_2\le \ldots \le \deg F_r$. An application of the multidimensional version of Taylor's theorem now shows that the relations (\ref{2.2}) hold for some choice of coefficients $c_{jl}(\bfxi)\in \dbZ[\xi_1,\ldots ,\xi_d]$ satisfying $c_{jj}(\bfxi)=1$ $(1\le j\le r)$. Since we may replace the set of forms $\calF_0$ by any subset whose span contains the polynomials $F_1,\ldots ,F_r$, there is no loss of generality in supposing the set $\{F_1,\ldots ,F_r\}$ to be linearly independent. Such a system of forms we call {\it reduced}.\par

Finally, by replacing the forms $F_1,\ldots ,F_r$ by appropriate linear combinations of the original forms, we find that there is no loss of generality also in supposing that the matrix $C(\bfxi)$ with entries $c_{jl}(\bfxi)$ is lower unitriangular. This new system $\bfF=(F_1,\ldots ,F_r)$, generated from the partial derivatives (\ref{2.5}), is a reduced translation-dilation invariant system.\par

We are now almost equipped to state our main theorem, but first pause to introduce some parameters associated with a translation-dilation invariant system of polynomials $\bfF$. When $\bfF=(F_1,\ldots ,F_r)$ consists of polynomials $F_j(\bfz)\in \dbZ[z_1,\ldots ,z_d]$, we refer to the number of variables $d=d(\bfF)$ in $\bfF$ as the {\it dimension} of the system. In addition, we describe the number of forms $r=r(\bfF)$ comprising $\bfF$ as the {\it rank} of the system. We write $k_j=k_j(\bfF)$ for the total degree of the polynomial $F_j$, and then define the {\it degree} $k=k(\bfF)$ of the system $\bfF$ by
$$k(\bfF)=\max_{1\le j\le r}k_j(\bfF),$$
and the {\it weight} $K=K(\bfF)$ by
$$K(\bfF)=\sum_{j=1}^rk_j(\bfF).$$

\par Our goal in \S\S\ref{sec:4}-\ref{sec:9} is the proof of the following mean value estimate, which represents the main theorem of this paper.

\begin{theorem}\label{theorem2.1} Let $\bfF$ be a reduced translation-dilation invariant system of polynomials having dimension $d$, rank $r$, degree $k$ and weight $K$. Suppose that $s$ is a natural number with $s\ge r(k+1)$. Then for each $\eps>0$, one has $J_s(X;\bfF)\ll X^{2sd-K+\eps}$.
\end{theorem}

So far as we are aware, no mean value estimate available in the literature has generality to compete with Theorem \ref{theorem2.1}. Moreover, when $d\ge 2$ the estimates available hitherto are considerably weaker, even in the special situations in which they are applicable. In order to illustrate the ease with which estimates may be extracted from Theorem \ref{theorem2.1}, we finish this section with a brief discussion of some simple cases, and in particular we show how to establish Theorem \ref{theorem1.1} as a consequence of Theorem \ref{theorem2.1}.\vskip.1cm

\noindent{\it (a) The classical system of Vinogradov \cite{Vin1935}, \cite{Vin1947}.} Consider the seed polynomial $z^k$ $(k\ge 1)$. By taking successive derivatives, we find that an associated reduced translation-dilation invariant system of polynomials is $\bfF=(z^k,z^{k-1},\ldots ,z)$. This system has dimension $1$, rank $k$, degree $k$ and weight
$$K=\sum_{j=1}^kj=\tfrac{1}{2}k(k+1).$$
Then it follows from Theorem \ref{theorem2.1} that when $s\ge k(k+1)$, one has
$$J_s(X;\bfF)\ll X^{2s-\frac{1}{2}k(k+1)+\eps}.$$
This estimate recovers the main conclusion of the third author's recent work introducing the efficient congruencing method to Vinogradov's mean value theorem (see \cite[Theorem 1.1]{Woo2012a}). We note that subsequent work of the third author leads to the improved constraint $s\ge k^2-1$ on the number of variables in this conclusion (see \cite[Theorem 1.1]{Woo2012b}).\vskip.1cm

\noindent{\it (b) The system of Parsell \cite{Par2005}.} Consider the situation with $d\ge 2$ and seed polynomials $z_1^{l_1}z_2^{l_2}\ldots z_d^{l_d}$ $(|\bfl|=k)$. By taking successive partial derivatives, we find that an associated reduced translation-dilation invariant system of polynomials is
$$\bfF=(z_1^{i_1}z_2^{i_2}\ldots z_d^{i_d}:1\le |\bfi|\le k).$$
This system has dimension $d$, rank
\begin{equation}\label{2.6}
r=\sum_{1\le |\bfi|\le k}1=\binom{k+d}{d}-1,
\end{equation}
degree $k$ and weight
\begin{equation}\label{2.7}
K=\sum_{l=1}^kl\sum_{|\bfi|=l}1=\sum_{l=1}^kl\binom{l+d-1}{l}=\frac{d}{d+1}(r+1)k.
\end{equation}
In this instance, it follows from Theorem \ref{theorem2.1} that when $s\ge r(k+1)$, one has $J_s(X;\bfF)\ll X^{2sd-K+\eps}$. In view of (\ref{1.2}) and (\ref{1.3}), this completes the proof of Theorem \ref{theorem1.1}.\vskip.1cm

\noindent{\it (c) The system of Arkhipov, Karatsuba and Chubarikov \cite{AKC1980}.} Consider the situation with $d\ge 2$ and $l\ge 1$ and the seed polynomial $z_1^lz_2^l\ldots z_d^l$. By taking successive partial derivatives, we find that an associated reduced translation-dilation invariant system of polynomials is
\begin{equation}\label{2.8}
\bfF=(z_1^{i_1}z_2^{i_2}\ldots z_d^{i_d}:0\le \bfi\le l,\ \bfi\ne {\mathbf 0}).
\end{equation}
This system has dimension $d$, rank
\begin{equation}\label{2.9}
r=\sum_{\substack{0\le \bfi\le l\\ \bfi\ne {\mathbf 0}}}1=\sum_{0\le i_1\le l}\ldots \sum_{0\le i_d\le l}1-1=(l+1)^d-1,
\end{equation}
degree $dl$, and weight
\begin{align}
K&=\sum_{0\le \bfi\le l}|\bfi|=\sum_{0\le i_1\le l}\ldots \sum_{0\le i_d\le l}(i_1+\ldots +i_d)\notag \\
&=d\sum_{0\le i_1\le l}\ldots \sum_{0\le i_d\le l}i_d\notag \\
&=d(l+1)^{d-1}\cdot \tfrac{1}{2}l(l+1)=\tfrac{1}{2}dl(l+1)^d.\label{2.10}
\end{align}
In this instance, Theorem \ref{theorem2.1} delivers a conclusion important enough to summarise as a corollary.

\begin{corollary}\label{corollary2.2}
Let $d$ and $l$ be natural numbers, and let $\bfF$ be the reduced translation-dilation invariant system given by (\ref{2.8}). Suppose that $s$ is a natural number with $s\ge (dl+1)((l+1)^d-1)$. Then for each $\eps>0$, one has
$$J_s(X;\bfF)\ll X^{2sd-K+\eps},$$
where $K=\frac{1}{2}dl(l+1)^d$.
\end{corollary}

A conclusion similar to that provided by Corollary \ref{corollary2.2}, but with the condition $s\ge (dl+1)((l+1)^d-1)$ replaced by
$$s\ge Cdl(l+1)^d\log (dl(l+1)^d),$$
for a suitable positive constant $C$, may be extracted from \cite[Theorem 1 of Chapter III.1]{AKC1980}. The superiority of our new bound is self-evident.\vskip.1cm

\noindent{\it (d) Simple binary systems.} A system of relevance to recent work in quantitative arithmetic geometry (see \cite[\S4.15]{Tsch2009},\cite{KVV2011}) deserves to be singled out for special attention. Consider the situation with $k_1\ge k_2\ge 1$ and the seed polynomial $z_1^{k_1}z_2^{k_2}$. By taking successive partial derivatives, we find that an associated reduced translation-dilation invariant system of polynomials is
\begin{equation}\label{2.11}
\bfF=(z_1^{i_1}z_2^{i_2}:\text{$0\le i_1\le k_1$, $0\le i_2\le k_2$ and $(i_1,i_2)\ne (0,0)$}).
\end{equation}
This system has dimension $2$, rank
$$r=\sum_{0\le i_1\le k_1}\sum_{0\le i_2\le k_2}1-1=(k_1+1)(k_2+1)-1,$$
degree $k_1+k_2$, and weight
\begin{align*}
K&=\sum_{0\le i_1\le k_1}\sum_{0\le i_2\le k_2}(i_1+i_2)\\
&=(k_1+1)\cdot \tfrac{1}{2}k_2(k_2+1)+(k_2+1)\cdot \tfrac{1}{2}k_1(k_1+1)\\
&=\tfrac{1}{2}(k_1+k_2)(k_1+1)(k_2+1).
\end{align*}
By applying Theorem \ref{theorem2.1}, we obtain the following corollary.

\begin{corollary}\label{corollary2.3}
Let $k_1,k_2\in \dbN$, and let $\bfF$ be the reduced translation-dilation invariant system given by (\ref{2.11}). Suppose that $s$ is a natural number with $s\ge (k_1k_2+k_1+k_2)(k_1+k_2+1)$. Then for each $\eps>0$, one has
$$J_s(X;\bfF)\ll X^{4s-K+\eps},$$
where $K=\frac{1}{2}(k_1+k_2)(k_1+1)(k_2+1)$.
\end{corollary}

\vskip.1cm

\noindent{\it (e) The binary systems of Prendiville \cite{Pre2011}.} Consider the situation with $k\ge 1$ and the seed polynomial given by the binary form $\Phi(z_1,z_2)\in \dbZ[z_1,z_2]$ of degree $k$. In this instance, we extract the partial derivatives
$$\frac{\partial^{i_1+i_2}\Phi(z_1,z_2)}{\partial z_1^{i_1}\partial z_2^{i_2}}\quad (\text{$i_1\ge 0$, $i_2\ge 0$}),$$
and restrict attention to any subset which spans the set of all partial derivatives of positive degree, yet is linearly independent over $\dbQ$. We take the polynomials in this spanning set to be our reduced translation-dilation invariant system $\bfF$. The number of partial derivatives with $i_1+i_2=l$ is plainly $l+1$, while the number of monomials $z_1^{j_1}z_2^{j_2}$ with $j_1+j_2=m$ is $m+1$. Thus we see that this system has dimension $d=2$, rank
$$r\le \sum_{l=0}^{[k/2]}(l+1)+\sum_{m=1}^{k-[k/2]-1}(m+1)\le \tfrac{1}{4}k(k+4),$$
degree $k$ and weight
$$K\le \sum_{l=0}^{[k/2]}(k-l)(l+1)+\sum_{m=1}^{k-[k/2]-1}m(m+1)\le \tfrac{1}{8}k(k+2)^2.$$
In typical situations, indeed, one has $K\sim \tfrac{1}{8}k^3$. By applying Theorem \ref{theorem2.1}, we deduce that when $s\ge \frac{1}{4}k(k+1)(k+4)$, one has $J_s(X;\bfF)\ll X^{4s-K+\eps}$. This conclusion may be compared with the mean value estimate underlying \cite[Theorem 1.3]{Pre2011}, which delivers a similar conclusion for $s\ge (\frac{3}{8}+o(1))k^3\log k$. The constraint on $s$ imposed in our present work is therefore stronger by a factor $(\frac{3}{2}+o(1))\log k$.

\section{Lower bounds}\label{sec:3} In order to put into perspective the upper bounds recorded in Theorem \ref{theorem2.1}, and such corollaries as Theorem \ref{theorem1.1}, we consider in this section the topic of lower bounds for the mean value $J_s(X;\bfF)$. Here one must consider integral solutions to the system of equations (\ref{2.1}) of two types. On the one hand, there are typical solutions whose contribution to $J_s(X;\bfF)$ we expect to be given by a product of local densities. On the other hand, there are integral solutions lying on special subvarieties, the most obvious of which are diagonal linear spaces such as that given by $\bfx_i=\bfy_i$ $(1\le i\le s)$. It transpires that when $d>1$, there are special subvarieties not of the latter type which potentially make the dominant contribution to $J_s(X;\bfF)$. In order to describe the latter subvarieties, we must introduce some further notation.\par

Let $r$ and $d$ be natural numbers, and consider a system of translation-dilation invariant polynomials $\bfF=(F_1,\ldots ,F_r)$, where $F_j(\bfz)\in \dbZ[z_1,\ldots ,z_d]$ $(1\le j\le r)$. Let $\del$ be a natural number with $1\le \del \le d-1$, and consider indices $i_l$ $(1\le l\le \del)$ satisfying
\begin{equation}\label{3.1}
1\le i_1<i_2<\ldots <i_\del \le d.
\end{equation}
We say that the system of polynomials $\bfG=(G_1,\ldots ,G_r)$, where $G_j(\bfw)\in \dbZ[w_1,\ldots ,w_\del]$ $(1\le j\le r)$, is the {\it orthogonal projection of $\bfF$ determined by $\bfi$} when
$$G_j(\bfw)=F_j(\bfzet)\quad (1\le j\le r),$$
in which $\zet_m=w_l$ when $m=i_l$ for some index $l$ with $1\le l\le \del$, and $\zet_m=0$ when $m\not\in \{i_1,\ldots ,i_\del\}$. The system $\bfG$ remains translation-dilation invariant, and may be replaced by an equivalent reduced system $\bfG^\prime $. We describe $\bfG^\prime $ as a {\it reduced orthogonal projection of $\bfF$ determined by $\bfi$}. Finally, write $\pi_\del(\bfF)$ for the set of all reduced orthogonal projections of $\bfF$ determined by sets of indices $\{i_1,\ldots ,i_\del\}$ satisfying (\ref{3.1}). We remark that these orthogonal projections are in fact a special case of the more general projections introduced in the preamble to Theorem \ref{theorem1.7}. In this section we consider only the former projections, since they are simpler to analyse and in any case deliver all of the salient features of importance for our discussion of lower bounds.\par

In order to facilitate our subsequent discussion, we define the polynomial $\psi(x;\bfalp)=\psi(x;\bfalp;\bfF)$ by putting
\begin{equation}\label{3.2}
\psi(\bfx;\bfalp;\bfF)=\sum_{i=1}^r\alp_iF_i(\bfx),
\end{equation}
and then define the associated exponential sum $f(\bfalp)=f(\bfalp;X;\bfF)$ by
\begin{equation}\label{3.3}
f(\bfalp;X;\bfF)=\sum_{1\le \bfx\le X}e(\psi(\bfx;\bfalp;\bfF)).
\end{equation}
By orthogonality, we then have
\begin{equation}\label{3.4}
J_s(X;\bfF)=\oint |f(\bfalp;X;\bfF)|^{2s}\d\bfalp .
\end{equation}

\par We are now equipped to describe our most general lower bound for the mean value $J_s(X;\bfF)$.

\begin{theorem}\label{theorem3.1}
Let $\bfF$ be a reduced translation-dilation invariant system of polynomials having dimension $d$ and weight $K$. Then for each natural number $s$, one has
$$J_s(X;\bfF)\gg X^{sd}+X^{2sd-K}+\sum_{\del=1}^{d-1}X^{d-\del}\max_{\bfG\in \pi_\del(\bfF)}J_s(X;\bfG).$$
\end{theorem}

\begin{proof} We consider first typical solutions of the system (\ref{2.1}) not constrained to lie on special subvarieties. Suppose that $\bfF$ has rank $r$, and write $k_j$ for the degree of $F_j$ for $1\le j\le r$. There exists a positive number $A$, depending at most on $d$, $k$ and the coefficients of $\bfF$, such that whenever $1\le \bfx\le X$, one has
$$|F_j(\bfx)|\le AX^{k_j}\quad (1\le j\le r).$$
Consequently, when $1\le \vbfx,\vbfy\le X$, one sees that for $1\le j\le r$ the integer
$$\sum_{l=1}^s(F_j(\bfx_l)-F_j(\bfy_l))$$
lies in the interval $[-2sAX^{k_j},2sAX^{k_j}]$. We therefore deduce by means of orthogonality in combination with the triangle inequality and (\ref{3.4}) that
\begin{align*}
[X]^{2sd}&=\sum_{\substack{|h_j|\le 2sAX^{k_j}\\ (1\le j\le r)}}\oint |f(\bfalp;X;\bfF)|^{2s}e(-\alp_1h_1-\ldots -\alp_rh_r)\d\bfalp\\
&\ll \Bigl( \prod_{1\le j\le r}X^{k_j}\Bigr) \oint |f(\bfalp;X;\bfF)|^{2s}\d\bfalp =X^KJ_s(X;\bfF).
\end{align*}
Thus we conclude that
\begin{equation}\label{3.5}
J_s(X;\bfF)\gg X^{2sd-K}.
\end{equation}

\par Next, by considering the diagonal solutions of (\ref{2.1}) with $1\le \vbfx,\vbfy\le X$ and $\bfx_j=\bfy_j$ $(1\le j\le s)$, we obtain the lower bound
\begin{equation}\label{3.6}
J_s(X;\bfF)\gg X^{sd}.
\end{equation}

\par We now come to consider the solutions of the system (\ref{2.1}) lying on certain additional special subvarieties. We assert that when $\bfG$ is a reduced translation-dilation invariant system with dimension $e\ge 2$, then for every system $\bfH$ lying in $\pi_{e-1}(\bfG)$, one has
\begin{equation}\label{3.7}
J_s(X;\bfG)\gg XJ_s(X;\bfH).
\end{equation}
In view of (\ref{3.5}) and (\ref{3.6}), the lower bound claimed in the statement of Theorem \ref{theorem3.1} then follows by induction.\par

Let $\bfG=(G_1,\ldots ,G_u)$ be a reduced translation-dilation invariant system of dimension $e>1$, and consider a system $\bfH\in \pi_{e-1}(\bfG)$. The system $\bfH$ is the orthogonal projection of $\bfG$ determined by some $(e-1)$-tuple $\bfi=(i_1,\ldots ,i_{e-1})$. By relabelling variables, if necessary, we may suppose that $\bfi=(1,2,\ldots ,e-1)$. Consider now the system of equations
\begin{equation}\label{3.8}
\sum_{i=1}^s(\bfG(\bfx_i)-\bfG(\bfy_i))={\bf 0}.
\end{equation}
Let $a$ be an integer with $1\le a\le X$, set $\bfa=(0,\ldots ,0,a)$, and consider the effect of the translation $(\bfx_i,\bfy_i)\longmapsto (\bfx_i-\bfa,\bfy_i-\bfa)$. In view of the translation invariance of the system (\ref{3.8}) that is a consequence of the discussion leading to (\ref{2.4}), one finds that whenever the system of equations
\begin{equation}\label{3.9}
\sum_{i=1}^s(\bfG(\bfx_i-\bfa)-\bfG(\bfy_i-\bfa))={\bf 0}
\end{equation}
is satisfied, then so too is the system (\ref{3.8}). If we substitute $x_{ie}=y_{ie}=a$ $(1\le i\le s)$, then we find that $\bfG(\bfx_i-\bfa)=\bfH(\bfx_i)$, where $\bfH$ is the aforementioned orthogonal projection of $\bfG$ determined by $(1,2,\ldots,e-1)$. Write $\bfH^\prime$ for any reduced system equivalent to $\bfH$. Then we conclude that whenever $\vbfz,\vbfw$ is a solution of the system
$$\sum_{i=1}^s(\bfH^\prime(\bfz_i)-\bfH^\prime(\bfw_i))={\bf 0},$$
then the system (\ref{3.9}) has the solution
$$\bfx_i=(\bfz_i,a)\quad \text{and}\quad \bfy_i=(\bfw_i,a)\quad (1\le i\le s).$$
The latter is also a solution of (\ref{3.8}), and hence
$$J_s(X;\bfG)\ge \sum_{1\le a\le X}J_s(X;\bfH^\prime)=\sum_{1\le a\le X}J_s(X;\bfH)=[X]J_s(X;\bfH).$$
This confirms the lower bound (\ref{3.7}), and in view of our earlier discussion the proof of the theorem is complete.
\end{proof}

We turn now to discuss lower bounds for $J_s(X;\bfF)$ for the most basic examples of reduced translation-dilation invariant systems $\bfF$.\vskip.1cm

\noindent{\it (a) The classical system of Vinogradov.} We recall the system
$$\bfF=(z^k,z^{k-1},\ldots ,z)$$
of dimension $1$, rank $k$, degree $k$ and weight $\tfrac{1}{2}k(k+1)$. In this situation, the conclusion of Theorem \ref{theorem3.1} delivers the familiar lower bound 
$$J_s(X;\bfF)\gg X^s+X^{2s-\frac{1}{2}k(k+1)}.$$
\vskip.1cm

\noindent{\it (b) The system of Parsell.} We next return to the situation with $d\ge 2$ and the system $\bfF_d=(z_1^{i_1}z_2^{i_2}\ldots z_d^{i_d}:1\le |\bfi|\le k)$. On writing
\begin{equation}\label{3.10}
K_\del =\sum_{l=1}^kl\binom{l+\del-1}{l},
\end{equation}
we see that this system has weight $K_d$, and it follows from Theorem \ref{theorem3.1} that
\begin{align*}
J_s(X;\bfF_d)&\gg X^{sd}+X^{2sd-K_d}+\sum_{\del=1}^{d-1}X^{d-\del}J_s(X;\bfF_\del)\\
&\gg X^{sd}+\sum_{\del=1}^dX^{d-\del}(X^{2s\del -K_\del}+X^{s\del}).
\end{align*}
We therefore conclude that
$$J_s(X;\bfF_d)\gg X^{sd}+\sum_{j=1}^dX^{(2s-1)j+d-K_j},$$
and this establishes Theorem \ref{theorem1.2}.\par

Let us consider the strength of the upper bound presented in Theorem \ref{theorem1.1} in the light of the lower bound just established. Define $r$ and $K$ as in (\ref{2.6}) and (\ref{2.7}). When $s$ is large enough, we expect that $J_s(X;\bfF_d)\ll X^{2sd-K}$, and indeed this estimate is a consequence of Theorem \ref{theorem1.5} when $s\ge r(k+1)+1$. We show here that this lower bound on $s$ cannot be relaxed substantially when $k$ is large in terms of $d$.

\begin{theorem}\label{theorem3.2} Suppose that $s$, $k$ and $d$ are natural numbers, and that $\nu$ is a real number with $2d/k<\nu<1$. Then there is a positive number $\eta=\eta(d,k)$ such that, whenever
\begin{equation}\label{3.11}
s\le \frac{d}{2d+2}(1-\nu)r(k+1),
\end{equation}
one has $J_{s,k,d}(X)\gg X^{2sd-K+\eta}$.
\end{theorem}

\begin{proof} Recall the notation introduced in (\ref{3.10}), and consider a natural number $s$ satisfying (\ref{3.11}). We observe that as a consequence of Theorem \ref{theorem1.2}, one has
\begin{equation}\label{3.12}
\frac{J_{s,k,d}(X)}{X^{2sd-K_d}}\gg \frac{X^{2s(d-1)+1-K_{d-1}}}{X^{2sd-K_d}}=X^{K_d-K_{d-1}+1-2s}.
\end{equation}
From equations (\ref{1.2}) and (\ref{1.3}), one has
$$K_d=\frac{dk}{d+1}\binom{k+d}{d}=\frac{d}{d+1}(r+1)k$$
and
$$K_{d-1}=\frac{(d-1)k}{d}\binom{k+d-1}{d-1}=\frac{d-1}{k+d}(r+1)k.$$
Consequently, one finds that
\begin{align*}
K_d-K_{d-1}+1-2s\ge &\, \frac{d}{d+1}(r+1)k-\frac{d-1}{k+d}(r+1)k+1\\
&\, -\frac{d}{d+1}(1-\nu)r(k+1)\\
\ge &\, \frac{d}{d+1}(r+1)k\Bigl( 1-\frac{d^2-1}{d(k+d)}-(1-\nu)(1+1/k)\Bigr) .
\end{align*}
Then provided that $\nu>2d/k$, one may infer that
$$K_d-K_{d-1}+1-2s>\frac{d}{d+1}(r+1)k\Bigl( \frac{2d-1}{k}-\frac{d}{k+d}+\frac{1}{d(k+d)}\Bigr) >0.$$
One concludes therefore that there exists a positive number $\eta$ such that
$$J_{s,k,d}(X)\gg X^{2sd-K_d+\eta}.$$
This completes the proof of the theorem.\end{proof}

Recall that Theorem \ref{theorem1.1} asserts that $J_{s,k,d}(X)\ll X^{2sd-K_d+\eps}$ whenever $s\ge r(k+1)$. Consequently, for any positive number $\nu<1$, it follows from Theorem \ref{theorem3.2} that when $k$ is large enough in terms of $\nu$ and $d$, such a conclusion is impossible if the constraint on $s$ is replaced by
$$s\ge \frac{d}{2d+2}(1-\nu)r(k+1).$$
Thus, the conclusion of Theorem \ref{theorem1.1} is at worst a factor of essentially $2+2/d$ away from the best possible conclusion of its type.\par

When $d$ is large and $k$ is small the situation changes. Here, whenever
$$s\le \frac{k}{2(k+d)}\Bigl( \frac{d}{d+1}\Bigr) (r+1)k,$$
we find that
$$K_d-K_{d-1}+1-2s>\frac{d}{d+1}(r+1)k\Bigl( 1-\frac{d^2-1}{d(k+d)}-\frac{k}{k+d}\Bigr)>0.$$
When $d$ is large enough in terms of $k$, one finds that
$$\frac{k}{2(k+d)}\Bigl( \frac{d}{d+1}\Bigr) (r+1)k=r(k+1)\Bigl( \frac{k^2}{2(k+1)d}\Bigr) (1+O(k/d)).$$
Thus, when $\nu>0$ is small enough in terms of $d$ and $k$, the lower bound $J_{s,k,d}(X)\gg X^{2sd-K_d+\eta}$ holds for a positive number $\eta$ provided that
$$s\le \frac{(1-\nu)k^2}{2(k+1)d}r(k+1).$$
In this situation, the conclusion of Theorem \ref{theorem1.1} is at worst a factor of essentially $2(1+1/k)d/k$ away from the best possible conclusion of its type.\vskip.1cm

\noindent{\it (c) The system of Arkhipov, Karatsuba and Chubarikov.} Here we consider the situation with $d\ge 2$ and $l\ge 1$ in which $\bfF$ is given by (\ref{2.8}). On recalling (\ref{2.10}), the weight of such a system with dimension $d$ is $K_d=\frac{1}{2}dl(l+1)^d$, whilst the corresponding weight of such a system with dimension $d-1$ is $K_{d-1}=\frac{1}{2}(d-1)l(l+1)^{d-1}$. As in the argument leading to (\ref{3.12}), an application of Theorem \ref{theorem3.1} in this instance delivers the lower bound
$$\frac{J_s(X;\bfF)}{X^{2sd-K_d}}\gg X^{K_d-K_{d-1}+1-2s}.$$
From (\ref{2.10}) one finds when
$$s\le \tfrac{1}{4}dl(l+1)^d(1-(1-1/d)(l+1)^{-1})$$
one has
$$K_d-K_{d-1}+1-2s\ge 1+\tfrac{1}{2}dl(l+1)^d(1-(1-1/d)(l+1)^{-1})-2s>0.$$
Observe that
$$\tfrac{1}{4}dl(l+1)^d(1-(1-1/d)(l+1)^{-1})=\frac{1+O(1/(dl))}{4(1+1/l)}(dl+1)((l+1)^d-1).$$
Thus, when $\nu>0$ is small enough in terms of $d$ and $l$, one finds that the lower bound $J_s(X;\bfF)\gg X^{2sd-K+\eta}$ holds for a positive number $\eta$ provided that
$$s\le \frac{(1-\nu)}{4(1+1/l)}(dl+1)((l+1)^d-1).$$
In this situation, the conclusion of Corollary \ref{2.2} is at worst a factor of essentially $4(1+1/l)$ away from the best possible conclusion of its type. By way of comparison, the work of Arkhipov, Karatsuba and Chubarikov \cite{AKC1980} would miss the best possible conclusion by a factor of order $d\log l$.

\section{Preliminary manoeuvres}\label{sec:4} Our purpose in this section is to describe further notation and establish such preliminary estimates as are required to initiate the efficient congruencing procedure, with the ultimate objective of proving Theorem \ref{theorem2.1}. With this aim in mind, let $\bfF$ be a reduced translation-dilation invariant system of polynomials having dimension $d$, rank $r$, degree $k$ and weight $K$. Since the conclusion of Theorem \ref{theorem2.1} follows from linear algebra when $k=1$, there is no loss of generality in supposing that $k\ge 2$. We consider the system $\bfF$ to be fixed throughout, and consequently suppress mention of $\bfF$ by abbreviating $J_s(X;\bfF)$ to $J_s(X)$, with similar conventions in other notation as appropriate. We recall the notation introduced in (\ref{3.2}) and (\ref{3.3}), and note the consequence of orthogonality recorded in (\ref{3.4}).\par

Our argument involves the investigation of systems of congruences, the singular solutions of which must be isolated for special treatment. We pause at this point to introduce a special Jacobian determinant that facilitates this treatment. First, given a polynomial $G(\bfz)\in \dbZ[z_1,\ldots ,z_d]$, we write $\partial_iG$ to denote the partial derivative of $G$ with respect to the $i$th variable, so that
$$\partial_iG(\bfz)=\frac{\partial G}{\partial z_i}(z_1,\ldots ,z_d).$$
We now consider a function $\sig:\{1,\ldots ,r\}\rightarrow \{1,\ldots ,d\}$, and define the associated Jacobian determinant $\Del_r(\vbfx;\sig)=\Del(\bfx_1,\ldots ,\bfx_r;\sig)$ by
\begin{equation}\label{4.1}
\Del(\bfx_1,\ldots ,\bfx_r;\sig)=\det \left( \partial_{\sig(i)}F_j(\bfx_i)\right)_{1\le i,j\le r}.
\end{equation}

\par We seek a choice for $\sig$ having the property that $\Del_r(\vbfx;\sig)$ is not identically zero as a polynomial in $\vbfx$. With this goal in mind, we consider the monomials occurring in $\bfF$ and $\Del_r(\vbfx;\sig)$, and introduce an ordering on the exponents associated with these monomials in order to ease discussion. When $t$ is a natural number, and $\bfa,\bfb\in (\dbN\cup \{0\})^t$, we say that {\it $\bfa$ is less than $\bfb$ in colex order} when there exists an index $i$ with $1\le i\le t$ such that $a_i<b_i$, and further $a_j=b_j$ for
 $j>i$. In such a situation, we write $\bfa\prec \bfb$, and we write $\bfa\preccurlyeq \bfb$ when $\bfa=\bfb$ or $\bfa\prec \bfb$. Next, given $\bfa\in (\dbN\cup \{0\})^t$, we write $\bfx^\bfa$ for the monomial $x_1^{a_1}x_2^{a_2}\ldots x_t^{a_t}$. The monomials $\bfx^\bfa$ may now be ordered according to the colexicographical order of their indices. Notice that when $\bfx_i=(x_{i1},\ldots ,x_{id})$, then the monomial $\bfx_1^{\bfc_1}\ldots \bfx_r^{\bfc_r}$ has smaller degree than $\bfx_1^{\bfd_1}\ldots \bfx_r^{\bfd_r}$ in colex if and only if there exists an index $i$ for which $\bfc_i\prec \bfd_i$, and further $\bfc_j=\bfd_j$ for $j>i$.

\begin{lemma}\label{lemma4.1}
There exists a function $\sig:\{1,\ldots ,r\}\rightarrow \{1,\ldots ,d\}$ such that $\Del_r(\vbfx;\sig)$ is non-zero as a polynomial in $\bfx_1,\ldots ,\bfx_r$.
\end{lemma}

\begin{proof} We begin by interpreting the polynomials $F_j(\bfz)$ in terms of the colex ordering of monomials. Recall that $k={\displaystyle{\max_{1\le j\le r}}}\text{deg}(F_j)$. Write
$$\calA=\{ \bfa\in (\dbN\cup \{0\})^d:1\le a_1+\ldots +a_d\le k\},$$
so that the polynomials $F_j(\bfz)$ are necessarily linear combinations of the monomials $\bfz^\bfa$ with $\bfa\in \calA$. We put $A=\text{card}(\calA)$, and label indices in such a manner that
$$\calA=\{ \bfa_1,\ldots ,\bfa_A\},$$
with $\bfa_1\prec \bfa_2\prec \ldots \prec \bfa_A$ the colex ordering of the elements of $\calA$. Thus, for $1\le j\le r$, there exists an integral $A$-tuple $\bfc_j$, which we consider as a row vector, having the property that
$$F_j(\bfz)=\bfc_j\cdot (\bfz^{\bfa_1},\ldots ,\bfz^{\bfa_A})^T.$$

\par Since $F_1,\ldots ,F_r$ are linearly independent, the matrix
$$\calC=\left( \bfc_j\right)_{1\le j\le r}$$
must have full rank, and hence there exists an invertible $r\times r$ matrix $\calM$ with rational coefficients having the property that the product $\calR=\calM\calC$ is a full rank matrix in inverted reduced row-echelon form. By the latter we mean that if $\calR$ is the matrix
$$(\rho_{l,m})_{\substack{1\le l\le r\\ 1\le m\le A}},$$
then the corresponding matrix
$$(\rho_{r+1-l,A+1-m})_{\substack{1\le l\le r\\ 1\le m\le A}}$$
is in conventional reduced row echelon form. We define the $r$-tuple of polynomials $(G_1,\ldots ,G_r)$ by putting
\begin{align}
(G_1,\ldots ,G_r)^T&=\calM(F_1,\ldots ,F_r)^T=\calM \calC (\bfz^{\bfa_1},\ldots ,\bfz^{\bfa_A})^T\notag \\
&=\calR(\bfz^{\bfa_1},\ldots ,\bfz^{\bfa_A})^T.\label{4.2}
\end{align}
Let $\bfz^{\bfb_j}$ denote the leading monomial of $G_j(\bfz)$ in colex, for $1\le j\le r$. Since $\calR$ is in inverted reduced row echelon form, we have
\begin{equation}\label{4.3}
\bfb_1\prec \bfb_2\prec \ldots \prec \bfb_r.
\end{equation}
We now define the function $\sig:\{1,\ldots ,r\}\rightarrow \{1,\ldots ,d\}$ as follows. When $1\le i\le r$, we take $\sig(i)$ to be the smallest index $h$ with the property that $b_{ih}>0$.\par

It remains to verify that with the choice of the function $\sig$ just made, the polynomial $\Del_r(\vbfx;\sig)$ is non-zero. Let $\bfe_i$ denote the $d$-dimensional vector whose $i$th coordinate is equal to $1$, all of whose remaining coordinates are $0$. When $1\le n\le r$, define the Jacobian determinant $D_n(\vbfx;\sig)=D(\bfx_1,\ldots ,\bfx_n;\sig)$ by
$$D(\bfx_1,\ldots ,\bfx_n;\sig)=\text{det}(\partial_{\sig(i)}G_j(\bfx_i))_{1\le i,j\le n}.$$
We proceed by induction to show that for $1\le n\le r$, the determinant $D_n(\vbfx;\sig)$ may be expanded in the shape
\begin{equation}\label{4.4}
D_n(\vbfx;\sig)=D_n^{(1)}(\vbfx;\sig)+D_n^{(2)}(\vbfx;\sig),
\end{equation}
where
\begin{equation}\label{4.5}
D_n^{(1)}(\vbfx;\sig)=\prod_{j=1}^nb_{j,\sig(j)}\bfx_j^{\bfb_j-\bfe_{\sig(j)}},
\end{equation}
and $D_n^{(2)}(\vbfx;\sig)$ is of smaller degree in colex than $D_n^{(1)}(\vbfx;\sig)$. Notice that the definition of $\sig$ implies that $b_{1,\sig(1)}\ldots b_{n,\sig(n)}\ne 0$. Then having established the inductive hypothesis for $n=r$, it follows that $D_r(\vbfx;\sig)$ contains the monomial
$$\prod_{j=1}^r\bfx_j^{\bfb_j-\bfe_{\sig(j)}}$$
with a non-zero coefficient, and hence $D_r(\vbfx;\sig)$ must be a non-zero polynomial. In this way, the proof of the inductive hypothesis will facilitate the proof of the lemma.\par

When $n=1$, we have $D_n(\vbfx;\sig)=\partial_{\sig(1)}G_1(\bfx_1)$, and hence it follows at once that (\ref{4.4}) holds with
$$D_1^{(1)}(\vbfx;\sig)=b_{1,\sig(1)}\bfx_1^{\bfb_1-\bfe_{\sig(1)}},$$
for some polynomial $D_1^{(2)}(\vbfx;\sig)$ of degree smaller in colex than $\bfx_1^{\bfb_1-\bfe_{\sig(1)}}$. Thus the inductive hypothesis holds with $n=1$.\par

Suppose next that the inductive hypothesis has been established already for $1\le n<u$. When $1\le i,j\le u$, define the polynomials $\tet_{ij}$ by putting
$$\tet_{ij}=\begin{cases} b_{u,\sig(u)}\bfx_u^{\bfb_u-\bfe_{\sig(u)}},&\text{when $i=j=u$},\\
0,&\text{when $(i,j)\ne (u,u)$.}\end{cases}$$
Then the determinant $D_u(\vbfx;\sig)$ has the expansion
$$D_u(\vbfx;\sig)=\tet_{uu}D_{u-1}(\vbfx;\sig)+\text{det}(\partial_{\sig(i)}G_j(\bfx_i)-\tet_{ij})_{1\le i,j\le u}.$$
On making use of the inductive hypothesis with $n=u-1$ in order to expand $D_{u-1}(\vbfx;\sig)$, we deduce that
\begin{align}
D_u(\vbfx;\sig)&=\tet_{uu}(D^{(1)}_{u-1}(\vbfx;\sig)+D^{(2)}_{u-1}(\vbfx;\sig))+\text{det}(\partial_{\sig(i)}G_j(\bfx_i)-\tet_{ij})_{1\le i,j\le u}\notag \\
&=D_u^{(1)}(\vbfx;\sig)+D_u^{(0)}(\vbfx;\sig),\label{4.6}
\end{align}
where
\begin{equation}\label{4.7a}
D^{(0)}_u(\vbfx;\sig)=b_{u,\sig(u)}\bfx_u^{\bfb_u-\bfe_{\sig(u)}}D^{(2)}_{u-1}(\vbfx;\sig)+\text{det}(\partial_{\sig(i)}G_j(\bfx_i)-\tet_{ij})_{1\le i,j\le u}.
\end{equation}
Every term on the right hand side of (\ref{4.7a}) may be expanded as a linear combination of terms of the shape
\begin{equation}\label{4.8a}
\bfx_1^{\bfh_1-\bfe_{\sig(1)}}\ldots \bfx_u^{\bfh_u-\bfe_{\sig(u)}},
\end{equation}
in which for some permutation $\pi:\{1,\ldots ,u\}\rightarrow \{1,\ldots ,u\}$, there exist $d$-tuples $\bfb_1^\prime,\ldots ,\bfb_u^\prime$ with $(\bfb_1^\prime,\ldots ,\bfb_u^\prime)\preccurlyeq (\bfb_1,\ldots ,\bfb_u)$ and $\bfb_l^\prime \preccurlyeq \bfb_l$ $(1\le l\le u)$, and having the property that
$$(\bfh_1,\ldots ,\bfh_u)=(\bfb_{\pi(1)}^\prime,\ldots ,\bfb_{\pi(u)}^\prime).$$
Moreover, in the event that $\pi$ is the identity permutation, then one has $(\bfb_1^\prime,\ldots ,\bfb_u^\prime)\prec (\bfb_1,\ldots ,\bfb_u)$.\par

If $\pi$ is the identity permutation, then
$$(\bfh_1,\ldots ,\bfh_u)=(\bfb_1^\prime,\ldots ,\bfb_u^\prime)\prec (\bfb_1,\ldots ,\bfb_u).$$
Since the (strict) colex ordering between tuples is not reversed on component-wise addition, a modicum of thought confirms that 
\begin{equation}\label{4.9a}
(\bfh_1-\bfe_{\sig(1)},\ldots ,\bfh_u-\bfe_{\sig(u)})\prec (\bfb_1-\bfe_{\sig(1)},\ldots ,\bfb_u-\bfe_{\sig(u)}).
\end{equation}
Unfortunately, our notation is somewhat opaque, and so it may be worthwhile to spell out the details underlying this deduction. Since $(\bfh_1,\ldots ,\bfh_u)\prec (\bfb_1,\ldots ,\bfb_u)$, there exists an index $l$ with the property that $\bfh_l\prec \bfb_l$, and further $\bfh_j=\bfb_j$ for $j>l$. But then there exists an index $m=m(l)$ with the property that $h_{lm}<b_{lm}$, and further $h_{lj}=b_{lj}$ for $j>m$. On recalling that the definition of the function $\sig$ implies that $b_{lj}=0$ for $j<\sig(l)$, we deduce that $m(l)\ge \sig(l)$. Thus we find that $\bfh_l-\bfe_{\sig(l)}\prec \bfb_l-\bfe_{\sig(l)}$, and further that $\bfh_j-\bfe_{\sig(j)}=\bfb_j-\bfe_{\sig(j)}$ for $j>l$. We therefore conclude that (\ref{4.9a}) holds, as we had previously asserted.\par

Suppose next that $\pi$ is not the identity permutation, and let $i$ be maximal with $\pi(i)\ne i$. Then for $j>i$ one has $j=\pi(j)$, and hence
$$\bfh_j=\bfb_{\pi(j)}^\prime =\bfb_j^\prime \preccurlyeq \bfb_j.$$
Since $j=\pi(j)$ for $j>i$, it follows that $i>\pi(i)$, and hence from (\ref{4.3}) we see that
$$\bfh_i=\bfb_{\pi(i)}^\prime \preccurlyeq \bfb_{\pi(i)}\prec \bfb_i.$$
It follows that $(\bfh_1,\ldots ,\bfh_u)\prec (\bfb_1,\ldots ,\bfb_u)$, so that as above one finds that the relation (\ref{4.9a}) holds also in the situation that $\pi$ is not the identity permutation. In view of (\ref{4.8a}), it follows that $D^{(0)}_u(\vbfx;\sig)$ is of smaller degree in colex than the polynomial $D^{(1)}_u(\vbfx;\sig)$ defined by means of (\ref{4.5}). We therefore deduce from (\ref{4.6}) that the relation (\ref{4.4}) holds with $n=u$, and with $D^{(2)}_u(\vbfx;\sig)$ of smaller degree in colex than $D^{(1)}_u(\vbfx;\sig)$. We have consequently established the inductive hypothesis with $n=u$, and hence the inductive hypothesis holds for $1\le n\le r$. In particular, the determinant $D_r(\vbfx;\sig)$ is a non-zero polynomial.\par

We now reverse course in order to relate the non-vanishing of $D_r(\vbfx;\sig)$ to the non-vanishing of $\Del_r(\vbfx;\sig)$. For each $l$ with $1\le l\le d$, it follows from (\ref{4.2}) that one has the identity
$$(\partial_lG_1,\ldots ,\partial_lG_r)^T=\calM(\partial_lF_1,\ldots ,\partial_lF_r)^T,$$
whence
$$(\partial_{\sig(i)}G_j(\bfx_i))^T_{1\le i,j\le r}=\calM (\partial_{\sig(i)}F_j(\bfx_i))^T_{1\le i,j\le r}.$$
Consequently, one has
$$D_r(\vbfx;\sig)=(\text{det}\,\calM)\Del_r(\vbfx;\sig).$$
But the matrix $\calM$ is invertible, so that $\text{det}\,\calM\ne 0$. Our conclusion that $D_r(\vbfx;\sig)$ is a non-zero polynomial therefore forces us to conclude that $\Del_r(\vbfx;\sig)$ is also a non-zero polynomial. This completes the proof of the lemma.
\end{proof}

Henceforth, we fix our choice of the function $\sig:\{1,\ldots ,r\}\rightarrow \{1,\ldots ,d\}$ so that $\Del_r(\vbfx;\sig)$ is a non-zero polynomial, as permitted by the conclusion of Lemma \ref{lemma4.1}, and we write $\Del_r(\vbfx)$ for $\Del_r(\vbfx;\sig)$.\par

We next turn to the task of introducing notation with which to describe the mean values central to our methods, as well as exponents with which to bound these mean values. We refer to the exponent $\lam_s=\lam_s(\bfF)$ as {\it permissible} when, for each positive number $\eps$, and for any real number $X$ sufficiently large in terms of $s$, $\bfF$ and $\eps$, one has $J_s(X;\bfF)\ll X^{\lam_s+\eps}$. Define $\lam_s^*$ to be the infimum of the set of exponents $\lam_s$ permissible for $s$ and $\bfF$, and then put $\eta_s=\lam_s^*-2sd+K$. Thus, whenever $X$ is sufficiently large in terms of $s$, $\bfF$ and $\eps$, one has
\begin{equation}\label{4.7}
J_s(X)\ll X^{\lam_s^*+\eps},
\end{equation}
where
\begin{equation}\label{4.8}
\lam_s^*=2sd-K+\eta_s.
\end{equation}
Note that, in view of the lower bound supplied by Theorem \ref{theorem3.1} and the trivial estimate $J_s(X)\le X^{2sd}$, one has $0\le \eta_s\le K$ for $s\in \dbN$.\par

We take $\del$ to be a small positive number to be fixed shortly. Let $u$ be a natural number with $u\ge k$, put $s_0=ur$, and fix a natural number $s$ with $s\ge s_0$. Our goal is to show that $\lam_{s+r}^*=2(s+r)d-K$, whence $\eta_{s+r}=0$. From this and the definition of $\lam_{s+r}^*$, it follows that there exists a sequence of natural numbers $(X_n)_{n=1}^\infty$, tending to infinity, with the property that
\begin{equation}\label{4.9}
J_{s+r}(X_n)>X_n^{\lam_{s+r}^*-\del}\quad (n\in \dbN).
\end{equation}
In view of (\ref{4.7}), when $X_n$ is sufficiently large and $X_n^{\del^2}<Y\le X_n$, we also have the corresponding upper bound
\begin{equation}\label{4.10}
J_{s+r}(Y)<Y^{\lam_{s+r}^*+\del}.
\end{equation}
Notice that since $s\ge s_0$, the trivial inequality $|f(\bfalp;X)|\le X^d$ leads to the upper bound
$$J_{s+r}(X)\le X^{2(s-s_0)d}\oint |f(\bfalp;X)|^{2s_0+2r}\d\bfalp =X^{2(s-s_0)d}J_{s_0+r}(X).$$
It follows that one has $\eta_{s+r}\le \eta_{s_0+r}$, and so we are free to restrict attention to the special case $s=s_0$. Finally, we take $N$ to be a natural number sufficiently large in terms of $s$ and $\bfF$. We then put
\begin{equation}\label{4.11}
\tet=N^{-1/2}(r/s)^{N+2},
\end{equation}
and fix $\del$ to be a positive number with $\del<(Ns)^{-3N}$, so that $\del$ is small compared to $\tet$. We now consider a fixed element $X=X_n$ of the sequence $(X_n)$, which we may assume to be sufficiently large in terms of $s$, $\bfF$, $N$ and $\del$, and we put $M=X^\tet$. Thus, in particular, one has $X^\del<M^{1/N}$.\par

Let $p$ be a fixed prime number with $M<p\le 2M$ to be chosen in due course. That such a prime exists is a consequence of the Prime Number Theorem. Recall the notation introduced in (\ref{3.2}). When $c$ is a non-negative integer, and $\bfxi\in \dbZ^d$ and $\bfalp\in [0,1)^r$, define
\begin{equation}\label{4.12}
\grf_c(\bfalp;\bfxi)=\sum_{\substack{1\le \bfx\le X\\ \bfx\equiv \bfxi\mmod{p^c}}}e(\psi(\bfx;\bfalp)).
\end{equation}
We need to equip ourselves with exponential sums of a type related to that defined in (\ref{4.12}), but possessing inherently non-singular structure. We introduce the notion of {\it well-conditioned} $r$-tuples $(\bfx_1,\ldots ,\bfx_r)$ with $\bfx_i\in \dbZ^d$ $(1\le i\le r)$. We denote by $\Xi_c(\bfxi)$ the set of $r$-tuples $\vbfxi=(\bfxi_1,\ldots ,\bfxi_r)\in (\dbZ^d)^r$, with
$$1\le \bfxi_i\le p^{c+1}\quad \text{and}\quad \bfxi_i\equiv \bfxi\mmod{p^c}\quad (1\le i\le r),$$
and satisfying the property that $\Del(\bfxi_1,\ldots ,\bfxi_r)\not \equiv 0\mmod{p^{(K-r)c+1}}$. In addition, write $\Sig_r=\{1,-1\}^r$, and consider an element $\bfsig$ of $\Sig_r$. Recalling the definition (\ref{4.12}), we then put
\begin{equation}\label{4.13}
\grF_c^\bfsig (\bfalp;\bfxi)=\sum_{\vbfxi\in \Xi_c(\bfxi)}\prod_{i=1}^r\grf_{c+1}(\sig_i\bfalp;\bfxi_i).
\end{equation}

\par Next we introduce the two mean values that underly our arguments. When $a$ and $b$ are non-negative integers, and $\bfsig,\bftau\in \Sig_r$, we define
\begin{equation}\label{4.14}
I_{a,b}^\bfsig(X;\bfxi,\bfeta)=\oint |\grF_a^\bfsig (\bfalp;\bfxi)^2\grf_b(\bfalp;\bfeta)^{2s}|\d\bfalp ,
\end{equation}
and
\begin{equation}\label{4.15}
K_{a,b}^{\bfsig,\bftau}(X;\bfxi,\bfeta)=\oint |\grF_a^\bfsig (\bfalp;\bfxi)^2\grF_b^\bftau (\bfalp;\bfeta)^{2u}|\d\bfalp .
\end{equation}
We then define
\begin{equation}\label{4.16}
I_{a,b}(X)=\max_{1\le \bfxi\le p^a}\max_{1\le \bfeta\le p^b}\max_{\bfsig\in \Sig_r}I_{a,b}^\bfsig(X;\bfxi,\bfeta),
\end{equation}
and
\begin{equation}\label{4.17}
K_{a,b}(X)=\max_{1\le \bfxi\le p^a}\max_{1\le \bfeta\le p^b}\max_{\bfsig,\bftau\in \Sig_r}K_{a,b}^{\bfsig,\bftau}(X;\bfxi,\bfeta).
\end{equation}
We stress that, although these mean values depend on our choice of $p$, this dependence will shortly be rendered irrelevant when we fix our choice of $p$ once and for all. Consequently, we suppress mention of $p$ in our notation.\par

Finally, following the simplifying notational device of \cite[\S3]{Woo2012a}, we define the normalised magnitude of the mean values $J_{s+r}(X)$ and $K_{a,b}(X)$ as follows. We define $\llbracket J_{s+r}(X)\rrbracket$ by means of the relation
\begin{equation}\label{4.18}
J_{s+r}(X)=X^{2(s+r)d-K}\llbracket J_{s+r}(X)\rrbracket ,
\end{equation}
and when $0\le a<b$, we define $\llbracket K_{a,b}(X)\rrbracket$ by means of the relation
\begin{equation}\label{4.19}
K_{a,b}(X)=(X/M^a)^{2rd-K}(X/M^b)^{2sd}\llbracket K_{a,b}(X)\rrbracket .
\end{equation}
Note that in view of (\ref{4.8}), the lower bound (\ref{4.9}) implies that
\begin{equation}\label{4.20}
\llbracket J_{s+r}(X)\rrbracket >X^{\eta_{s+r}-\del},
\end{equation}
while the upper bound (\ref{4.10}) ensures that whenever $X^{\del^2}<Y\le X$, then
\begin{equation}\label{4.21}
\llbracket J_{s+r}(Y)\rrbracket <Y^{\eta_{s+r}+\del}.
\end{equation}

\section{Auxiliary mean values}\label{sec:5} We collect together in this section several mean value estimates that facilitate our subsequent analysis. We begin by exploiting the translation-dilation invariance of the system $\bfF$ so as to bound an analogue of $J_s(X)$ in which variables are restricted to an arithmetic progression. This argument will be familiar to those who work on such problems.

\begin{lemma}\label{lemma5.1}
Suppose that $c$ is a non-negative integer with $c\tet\le 1$. Then for each natural number $t$, one has
\begin{equation}\label{5.1}
\max_{1\le \bfxi \le p^c}\oint|\grf_c(\bfalp;\bfxi)|^{2t}\d\bfalp \ll_t J_t(X/M^c).
\end{equation}
\end{lemma}

\begin{proof} Let $\bfxi$ be an integral $d$-tuple with $1\le \bfxi\le p^c$. By orthogonality, it follows from (\ref{4.12}) that the integral on the left hand side of (\ref{5.1}) is bounded above by the number of integral solutions of the system
\begin{equation}\label{5.2}
\sum_{i=1}^t(\bfF(p^c\bfy_i+\bfxi)-\bfF(p^c\bfz_i+\bfxi))={\bf 0},
\end{equation}
with $0\le \vbfy,\vbfz\le X/p^c$. The translation-dilation invariance of the system $\bfF$ discussed in the context of (\ref{2.1}) and (\ref{2.4}) shows that $\vbfy,\vbfz$ is an integral solution of (\ref{5.2}) if and only if
\begin{equation}\label{5.3}
\sum_{i=1}^t(\bfF(\bfy_i)-\bfF(\bfz_i))={\bf 0}.
\end{equation}
But on recalling (\ref{3.3}) and employing orthogonality again, we perceive that the number of integral solutions of (\ref{5.3}) with $0\le \vbfy,\vbfz\le X/p^c$ is equal to
$$\oint |1+f(\bfalp;X/p^c)|^{2t}\d\bfalp \ll_t 1+\oint |f(\bfalp;X/p^c)|^{2t}\d\bfalp .$$
Thus we obtain the upper bound
$$\oint |\grf_c(\bfalp;\bfxi)|^{2t}\d\bfalp \ll 1+J_t(X/p^c).$$
Since the condition $c\tet\le 1$ ensures that $X/M^c\ge 1$, a consideration of diagonal solutions ensures that $J_t(X/M^c)\ge 1$, and the conclusion of the lemma follows on noting that $J_t(X/p^c)\le J_t(X/M^c)$.
\end{proof}

Singular solutions associated with the vanishing of $\Del_r(\vbfx)$ are difficult to control, and so we prepare a lemma to bound their number. We first introduce some additional notation. Let $\dbF$ be a field. We restrict attention to either the field of rational numbers $\dbQ$, or else the finite field of $p$ elements $\dbF_p$. The coefficients of $\Del_r(\vbfx)$ embed into both of these fields. Suppose that $\calA\subseteq \dbF$ is finite, and write $A=\text{card}(\calA)$. Let $t$ be a natural number. We denote by $\calS_t(\calA;\dbF)$ the set of $t$-tuples $(\bfx_1,\ldots ,\bfx_t)\in (\calA^d)^t$ having the property that the determinants $\Del_r(\bfx_{j_1},\ldots ,\bfx_{j_r})$ vanish for all $r$-tuples $(j_1,\ldots ,j_r)$ with $1\le \bfj\le t$. Before announcing our main estimate for $\text{card}(\calS_t(\calA;\dbF))$, we first recall a familiar lemma bounding the number of zeros of polynomials in many variables (see \cite[Lemma 2]{Woo1993}, for example).

\begin{lemma}\label{lemma5.2} Let $\Ups\in \dbF[y_1,\ldots ,y_u]$ be a non-trivial polynomial of total degree $\kap$. Then the number of solutions of the equation $\Ups(y_1,\ldots ,y_u)=0$ with $\bfy\in \calA^u$ is at most $\kap A^{u-1}$.
\end{lemma}

\begin{proof} We proceed inductively. When $u=1$, the desired conclusion is a consequence of Lagrange's theorem. Suppose then that the conclusion of the lemma has been established for $1\le u<v$, and let $\Psi\in \dbF[y_1,\ldots ,y_v]$ be a non-trivial polynomial of total degree $\kap$. By relabelling variables if necessary, we may suppose that $\Psi$ is explicit in $y_v$. Let the degree of $\Psi$ with respect to $y_v$ be $\ome$, and let the coefficient of $y_v^\ome$ be $\Phi(y_1,\ldots ,y_{v-1})$. Then $\Phi$ is a non-trivial polynomial in $v-1$ variables of degree at most $\kap-\ome$. By the inductive hypothesis, the number of solutions of $\Phi(y_1,\ldots ,y_{v-1})=0$ with $(y_1,\ldots ,y_{v-1})\in \calA^{v-1}$ is at most $(\kap-\ome)A^{v-2}$. Then the number of $v$-tuples $(y_1,\ldots ,y_v)\in \calA^v$ satisfying $\Phi(y_1,\ldots ,y_{v-1})=0$ is at most $(\kap-\ome)A^{v-1}$. Meanwhile, when $\Phi(y_1,\ldots ,y_{v-1})\ne 0$ and $\Psi(y_1,\ldots ,y_v)=0$, then $y_v$ satisfies a non-trivial polynomial of degree $\ome$ determined by $y_1,\ldots ,y_{v-1}$. So there are at most $\ome A^{v-1}$ solutions of $\Psi(\bfy)=0$ with $\bfy\in \calA^v$ and $\Phi(y_1,\ldots ,y_{v-1})\ne 0$. We conclude that the total number of solutions of $\Psi(\bfy)=0$ with $\bfy\in \calA^v$ is at most $(\kap-\ome)A^{v-1}+\ome A^{v-1}=\kap A^{v-1}$, and hence the inductive hypothesis follows for $u=v$. The desired conclusion therefore follows by induction.
\end{proof}

\begin{lemma}\label{lemma5.3}
Suppose that $\Del_r(\vbfx)$ is not identically zero as a polynomial in $\dbF$. Then
$${\rm card}(\calS_t(\calA;\dbF))\ll A^{t(d-1)+r-1}.$$
\end{lemma}

\begin{proof} The conclusion of the lemma is trivial when $t<r$, so we may assume that $t\ge r$. We define a sequence of non-zero polynomials $\calD_i(\vbfx)=\calD_i(\bfx_1,\ldots ,\bfx_i)$ $(0\le i\le r)$ as follows. We begin by setting $\calD_r(\vbfx)=\Del_r(\vbfx)$. Suppose then that for some $l\ge 1$ we have constructed the polynomials $\calD_i(\vbfx)$ for $l\le i\le r$. Amongst the monomials $\bfx_1^{\bfh_1}\ldots \bfx_l^{\bfh_l}$ occurring in $\calD_l(\vbfx)$, let $\bfb_l$ denote the largest of the $d$-tuples $\bfh_l$ in colex. It follows that there exist polynomials $\calD_{l-1}(\vbfx)$ and $\calR_l(\vbfx)$ having the property that
\begin{equation}\label{5.4}
\calD_l(\bfx_1,\ldots ,\bfx_l)=\calD_{l-1}(\bfx_1,\ldots ,\bfx_{l-1})\bfx_l^{\bfb_l}+\calR_l(\bfx_1,\ldots ,\bfx_l).
\end{equation}
We may suppose that $\calD_{l-1}(\vbfx)$ is non-zero, and that every monomial $\bfx_1^{\bfh_1}\ldots \bfx_l^{\bfh_l}$ occurring in $\calR_l(\vbfx)$ satisfies $\bfh_l\prec \bfb_l$. In this way, we have defined polynomials $\calD_i(\vbfx)$ for $0\le i\le r$. Notice here that $\calD_0(\vbfx)$ is a non-zero element of $\dbF$.\par 

Consider an integer $j$ with $1\le j\le r$, and denote by $\calB_j$ the set of all $j$-element subsets of $\{1,2,\ldots ,t\}$. We define $\calT_j$ to be the set of $t$-tuples $(\bfx_1,\ldots ,\bfx_t)\in (\calA^d)^t$ satisfying the property that (a) for each subset $\{l_1,\ldots ,l_j\}$ in $\calB_j$, one has $\calD_j(\bfx_{l_1},\ldots ,\bfx_{l_j})=0$, and (b) whenever $i<j$, there exists a subset $\{m_1,\ldots ,m_i\}$ in $\calB_i$ such that $\calD_i(\bfx_{m_1},\ldots ,\bfx_{m_i})\ne 0$. A moment of reflection reveals that $\calS_t(\calA;\dbF)$ is the union of $\calT_1,\calT_2,\ldots ,\calT_r$.\par

We next seek to bound ${\rm card}(\calT_j)$ for each integer $j$ with $1\le j\le r$. Consider a $t$-tuple $(\bfx_1,\ldots ,\bfx_t)\in \calT_j$. There exists a subset $\{m_1,\ldots ,m_{j-1}\}$ in $\calB_{j-1}$ with $\calD_{j-1}(\bfx_{m_1},\ldots ,\bfx_{m_{j-1}})\ne 0$ having the property that for each index $m$ with $1\le m\le t$ for which $m\not \in \{m_1,\ldots ,m_{j-1}\}$, one has $\calD_j(\bfx_{m_1},\ldots ,\bfx_{m_{j-1}},\bfx_m)=0$. In view of the relation (5.4), the latter equation implies that the $d$-tuple $\bfx_m$ satisfies
\begin{equation}\label{5.5a}
\calD_{j-1}(\bfx_{m_1},\ldots ,\bfx_{m_{j-1}})\bfx_m^{\bfb_j}+\calR_j(\bfx_{m_1},\ldots ,\bfx_{m_{j-1}},\bfx_m)=0,
\end{equation}
in which the second term on the left hand side is of smaller degree in $\bfx_m$ in colex than the first term. Notice that since $\calD_{j-1}(\bfx_{m_1},\ldots ,\bfx_{m_{j-1}})\ne 0$, then in particular the equation (\ref{5.5a}) is non-trivial as a polynomial equation in $\bfx_m$. In this way, we deduce from Lemma \ref{lemma5.2} that for each index $m$ with $1\le m\le t$ for which $m\not \in \{m_1,\ldots ,m_{j-1}\}$, the number of $d$-tuples $\bfx_m\in \calA^d$ satisfying (\ref{5.5a}) is at most
$$(b_{j1}+\ldots +b_{jd})A^{d-1}\le (k-1)A^{d-1}.$$
The total number of choices of $\bfx_m\in \calA^d$ for all such indices $m$ is therefore at most $((k-1)A^{d-1})^{t-(j-1)}$. The number of choices for $(\bfx_{m_1},\ldots ,\bfx_{m_{j-1}})\in (\calA^d)^{j-1}$, meanwhile, is at most $A^{d(j-1)}$. Since a trivial estimate confirms that $\text{card}(\calB_{j-1})\le t^{j-1}$, we deduce that
$$\text{card}(\calT_j)\le t^{j-1}((k-1)A^{d-1})^{t-j+1}(A^d)^{j-1}\ll A^{t(d-1)+j-1}.$$
Combining the contributions of $\calT_1,\ldots ,\calT_r$, therefore, we conclude that
$$\text{card}(\calS_t(\calA;\dbF))=\sum_{j=1}^r\text{card}(\calT_j)\ll \sum_{j=1}^rA^{t(d-1)+j-1}\ll A^{t(d-1)+r-1}.$$
This completes the proof of the lemma.
\end{proof}

We are now equipped to initiate the iterative procedure. It is at this point that we fix our choice for the prime number $p$.

\begin{lemma}\label{lemma5.4} There exists a prime number $p$ with $M<p\le 2M$ for which $J_{s+r}(X)\ll M^{2sd}I_{0,1}(X)$.
\end{lemma}

\begin{proof} The mean value $J_{s+r}(X)$ counts the number of integral solutions of the system
\begin{equation}\label{5.5}
\sum_{i=1}^{2(s+r)}(-1)^i\bfF(\bfx_i)={\bf 0},
\end{equation}
with $1\le \vbfx\le X$. Let $T_0$ denote the number of such solutions in which
\begin{equation}\label{5.6}
\Del(\bfx_{i_1},\ldots ,\bfx_{i_r})=0
\end{equation}
for all indices $i_l$ $(1\le l\le r)$ satisfying
\begin{equation}\label{5.7}
1\le \bfi\le 2(s+r).
\end{equation}
Also, let $T_1$ denote the corresponding number of solutions in which (\ref{5.6}) fails to hold for some index $\bfi$ satisfying (\ref{5.7}). Then $J_{s+r}(X)=T_0+T_1$.\par

We first consider $T_0$. Put $\calA=\{1,2,\ldots,[X]\}$ and $\calS=\calS_{2s+2r}(\calA;\dbQ)$. Then on recalling the discussion in the preamble to the statement of Lemma \ref{lemma5.2}, we see that whenever $\bfx$ is counted by $T_0$, one has $\bfx\in \calS$. It therefore follows from Lemma \ref{lemma5.3} that
$$T_0\le \text{card}(\calS_{2s+2r}(\calA;\dbQ))\ll X^{2(s+r)(d-1)+r-1}.$$
Note that our hypotheses ensure that $s\ge rk$. Then since
$$K=\sum_{j=1}^rk_j\le rk,$$
we find that $s\ge K$. In view of the lower bound on $J_{s+r}(X)$ available from Theorem \ref{theorem3.1}, we deduce that
\begin{equation}\label{5.8}
T_0\ll X^{2(s+r)d-2s-1}\ll X^{2(s+r)d-K-1}\ll X^{-1}J_{s+r}(X).
\end{equation}

\par We next turn to the solutions counted by $T_1$. We begin by examining the determinant $\Del_r(\vbfz)$. On recalling (\ref{4.1}), we see that the $(j,i)$-th entry of the matrix associated with the determinant $\Del_r(\vbfz)$ has degree at most $k_j-1$. As a polynomial in $\vbfz$, therefore, we see that the degree of $\Del_r(\vbfz)$ satisfies
$$\text{deg}\, \Del_r(\vbfz)\le \sum_{j=1}^r(k_j-1)=K-r.$$
The coefficients of the monomial entries of $\Del_r(\vbfz)$ depend at most on $\bfF$, and so for sufficiently large values of $X$, the sum of the absolute values of these coefficients is bounded above by $X$. Thus we see that
$$\max_{1\le \vbfz\le X}|\Del_r(\vbfz)|\le X^K.$$
Let $\calP$ denote any set of $[K/\tet]+1$ distinct prime numbers with $M<p\le 2M$. Such a set exists by the Prime Number Theorem, since we are at liberty to assume $X$ to be large enough in terms of $K$ and $\tet$. It follows that
$$\prod_{p\in \calP}p>M^{K/\tet}=X^K\ge \max_{1\le \vbfz\le X}|\Del_r(\vbfz)|.$$
Consequently, whenever $1\le \vbfz\le X$ and $\Del_r(\vbfz)\ne 0$, then there exists a prime $p\in \calP$ for which $p\nmid \Del_r(\vbfz)$. In particular, for each solution $\vbfx$ of (\ref{5.5}) counted by $T_1$, there exists an index $\bfi$ satisfying (\ref{5.7}) and a prime $p\in \calP$ for which $\Del_r(\bfx_{i_1},\ldots ,\bfx_{i_r})\not\equiv 0\pmod{p}$.\par

Let $\calI$ denote the set of all indices $\bfi$ satisfying (\ref{5.7}), and define $\bfsig=\bfsig(\bfi)$ by putting $\bfsig(\bfi)=((-1)^{i_1},\ldots ,(-1)^{i_r})$. Also, put
$$l(\bfi)=(-1)^{i_1}+\ldots +(-1)^{i_r},\quad m(\bfi)=\tfrac{1}{2}(r+l(\bfi)), \quad n(\bfi)=\tfrac{1}{2}(r-l(\bfi)).$$
Then on recalling the definition (\ref{4.13}) and considering the underlying Diophantine equations, we see that
$$T_1\ll \sum_{p\in \calP}\sum_{\bfi\in \calI}\oint \grF_0^{\bfsig(\bfi)}(\bfalp;{\mathbf 0})f(\bfalp;X)^{s+r-m}f(-\bfalp;X)^{s+r-n}\d\bfalp .$$
An application of Schwarz's inequality therefore reveals that
$$T_1\ll \max_{p\in \calP}\Bigl( \max_{\bfsig \in \Sig_r}\oint |\grF_0^\bfsig (\bfalp;{\mathbf 0})^2f(\bfalp;X)^{2s}|\d\bfalp \Bigr)^{1/2}\Bigl( \oint |f(\bfalp;X)|^{2s+2r}\d\bfalp \Bigr)^{1/2}.$$
Recalling now (\ref{3.4}) and (\ref{4.16}), we deduce that there exists a prime number $p$ with $M<p\le 2M$ for which
$$T_1\ll (I_{0,0}(X))^{1/2}(J_{s+r}(X))^{1/2}.$$
By reference to (\ref{5.8}), we therefore arrive at the upper bound
$$J_{s+r}(X)=T_0+T_1\ll X^{-1}J_{s+r}(X)+(I_{0,0}(X))^{1/2}(J_{s+r}(X))^{1/2},$$
whence
\begin{equation}\label{5.9}
J_{s+r}(X)\ll 1+I_{0,0}(X)\ll I_{0,0}(X).
\end{equation}

\par Our final step is to split the summation in the definition (\ref{4.12}) of $\grf_0(\bfalp;{\mathbf 0})$ into arithmetic progressions modulo $p$. Thus we obtain
$$\grf_0(\bfalp;{\mathbf 0})=\sum_{1\le \bfxi\le p}\grf_1(\bfalp;\bfxi),$$
and so it follows from H\"older's inequality that
$$|\grf_0(\bfalp;{\mathbf 0})|^{2s}\le (p^d)^{2s-1}\sum_{1\le \bfxi \le p}|\grf_1(\bfalp;\bfxi)|^{2s}.$$
In this way we deduce from (\ref{4.14}) and (\ref{4.16}) that
$$I_{0,0}(X)\ll (M^d)^{2s}\max_{1\le \bfxi\le p}\max_{\bfsig \in \Sig_r}\oint |\grF_0^\bfsig (\bfalp;{\mathbf 0})^2\grf_1(\bfalp;\bfxi)^{2s}|\d\bfalp \ll M^{2sd}I_{0,1}(X).$$
The proof of the lemma is made complete by substituting this last estimate into (\ref{5.9}).
\end{proof}

This is the point at which we fix the prime number $p$, once and for all, in such a way that the estimate $J_{s+r}(X)\ll M^{2sd}I_{0,1}(X)$ holds. That such a choice is possible is guaranteed by the conclusion of Lemma \ref{lemma5.4}.

\section{Auxiliary congruences}\label{sec:6} The main thrust of our argument begins with a discussion of the congruences that play a critical role in what follows. We first introduce some additional notation. When $\bfsig\in \Sig_r$, denote by $\calB_{a,b}^\bfsig(\bfm;\bfxi,\bfeta)$ the set of solutions of the system of congruences
\begin{equation}\label{6.1}
\sum_{i=1}^r\sig_iF_j(\bfz_i-\bfeta)\equiv m_j\mmod{p^{k_jb}}\quad (1\le j\le r),
\end{equation}
with
\begin{equation}\label{6.2}
1\le \vbfz\le p^{kb},\quad \bfz_i\equiv \bfxi\mmod{p^a}\quad (1\le i\le r)
\end{equation}
and
\begin{equation}\label{6.3}
\Del(\bfz_1,\ldots ,\bfz_r)\not\equiv 0\mmod{p^{(K-r)a+1}}.
\end{equation}
The non-singularity condition (\ref{6.3}) is awkward to handle directly, and so we simplify it using the condition (\ref{6.2}) by means of the following lemma.

\begin{lemma}\label{lemma6.1} One has the polynomial identity
$$\Del(t\bfz_1+\bfxi,\ldots ,t\bfz_r+\bfxi)=t^{K-r}\Del(\bfz_1,\ldots ,\bfz_r).$$
\end{lemma}

\begin{proof} It suffices to establish the claimed identity in the special case $t=1$, since the homogeneity of 
the polynomials $F_j(\bfx)$ ensures that
$$\Del(t\bfz_1,\ldots ,t\bfz_r)=\Bigl( \prod_{j=1}^rt^{k_j-1}\Bigr) \Del(\bfz_1,\ldots ,\bfz_r)=t^{K-r}\Del(\bfz_1,\ldots ,\bfz_r).$$
Consider then the situation with $t=1$, and recall the relation (\ref{2.3}). By the chain rule, we find that for $1\le l\le d$ one has
$$\frac{\partial \bfF}{\partial x_l}(\bfx+\bfxi)=\frac{\partial}{\partial x_l}(\bfF(\bfx+\bfxi))=\frac{\partial}{\partial x_l}(C(\bfxi)\bfF(\bfx))=C(\bfxi)\partial_l\bfF(\bfx).$$
The definition (\ref{4.1}) of $\Del_r(\vbfx)$ therefore delivers the relation
\begin{align*}
\Del(\bfx_1+\bfxi,\ldots ,\bfx_r+\bfxi)&=\text{det}(\partial_{\sig(i)}F_j(\bfx_i+\bfxi))_{1\le i,j\le r}\\
&=\text{det}\left(C(\bfxi)(\partial_{\sig(i)}F_j(\bfx_i))_{1\le i,j\le r}\right)\\
&=(\text{det}\, C(\bfxi))\Del(\bfx_1,\ldots ,\bfx_r).
\end{align*}
But $C(\bfxi)$ is lower unitriangular, so that $\text{det}\, C(\bfxi)=1$. We therefore conclude that
$$\Del(\bfx_1+\bfxi,\ldots ,\bfx_r+\bfxi)=\Del(\bfx_1,\ldots ,\bfx_r),$$
and in view of our earlier discussion, the proof of the lemma is complete.
\end{proof}

We also require an analogue of Hensel's lemma in order to lift solutions of congruences to progressively higher moduli. A suitable version of this lifting process is implicit in the next lemma.

\begin{lemma}\label{lemma6.2} Let $f_1,\ldots ,f_t$ be polynomials in $\dbZ[x_1,\ldots ,x_t]$ with respective degrees $\kap_1,\ldots ,\kap_t$, and write
$$J(\bff;\bfx)={\rm det}\left( \frac{\partial f_j}{\partial x_i}(\bfx)\right)_{1\le i,j\le t}.$$
When $\varpi$ is a prime number, and $l$ is a natural number, let $\calN(\bff;\varpi^l)$ denote the number of solutions of the simultaneous congruences
$$f_j(x_1,\ldots ,x_t)\equiv 0\mmod{\varpi^l}\quad (1\le j\le t),$$
with $1\le x_i\le \varpi^l$ $(1\le i\le t)$ and $(J(\bff;\bfx),\varpi)=1$. Then $\calN(\bff;\varpi^l)\le \kap_1\ldots \kap_t$.
\end{lemma}

\begin{proof} This is \cite[Theorem 1]{Woo1996a}.\end{proof}

We are now equipped to establish the basic estimate for the number of solutions of the congruences (\ref{6.1}) comprising $\calB_{a,b}^\bfsig(\bfm;\bfxi,\bfeta)$.

\begin{lemma}\label{lemma6.3}
Suppose that $a$ and $b$ are non-negative integers with $b>a$. Then
$$\max_{1\le \bfxi\le p^a}\max_{1\le \bfeta\le p^b}\max_{\bfsig\in \Sig_r}{\rm card}(\calB_{a,b}^\bfsig(\bfm;\bfxi,\bfeta))\le k_1\ldots k_rp^{(kb-a)rd-K(b-a)}.$$
\end{lemma}

\begin{proof} Consider fixed integers $a$ and $b$ with $0\le a<b$, a fixed $r$-tuple $\bfsig\in \Sig_r$, and fixed integral $d$-tuples $\bfxi$ and $\bfeta$ with $1\le \bfxi\le p^a$ and $1\le \bfeta\le p^b$. Denote by $\calD_1(\bfn)$ the set of solutions of the system of congruences
\begin{equation}\label{6.4}
\sum_{i=1}^r\sig_i\bfF(\bfz_i-\bfeta)\equiv \bfn\mmod{p^{kb}},
\end{equation}
with $1\le \vbfz\le p^{kb}$ and
\begin{equation}\label{6.5}
\bfz_i\equiv \bfxi_i\mmod{p^{a+1}}\quad \text{for some}\quad \vbfxi \in \Xi_a(\bfxi)\quad (1\le i\le r).
\end{equation}
Given a fixed integral $r$-tuple $\bfm$, the number of $r$-tuples $\bfn$ with $1\le \bfn\le p^{kb}$, for which $n_j\equiv m_j\mmod{p^{k_jb}}$ $(1\le j\le r)$, is equal to
$$\prod_{j=1}^rp^{(k-k_j)b}=(p^b)^{rk-K}.$$
Then it follows from (\ref{6.1}) that
\begin{align}
\text{card}(\calB_{a,b}^\bfsig (\bfm;\bfxi,\bfeta))&=\sum_{\substack{1\le n_1\le p^{kb}\\ n_1\equiv m_1\mmod{p^{k_1b}}}}\ldots \sum_{\substack{1\le n_r\le p^{kb}\\ n_r\equiv m_r\mmod{p^{k_rb}}}}\text{card}(\calD_1(\bfn))\notag \\
&\le (p^b)^{rk-K}\max_{1\le \bfn\le p^{kb}}\text{card}(\calD_1(\bfn)).\label{6.6}
\end{align}

\par We next rewrite each variable $\bfz_i$ in the shape $\bfz_i=p^a\bfy_i+\bfxi$. Notice that the hypothesis that $\bfz_i\equiv \bfxi_i\mmod{p^{a+1}}$ for some $\vbfxi \in \Xi_a(\bfxi)$, recorded in (\ref{6.5}), implies that $\bfxi_i=p^a\bfv_i+\bfxi$ for some integral $d$-tuple $\bfv_i$, and further that
$$\Del(\bfxi_1,\ldots ,\bfxi_r)\not\equiv 0\mmod{p^{(K-r)a+1}}.$$
But as a consequence of Lemma \ref{lemma6.1}, one then has
$$p^{(K-r)a}\Del(\bfv_1,\ldots ,\bfv_r)=\Del(\bfxi_1,\ldots ,\bfxi_r)\not\equiv 0\mmod{p^{(K-r)a+1}},$$
whence
$$\Del(\bfv_1,\ldots ,\bfv_r)\not\equiv 0\mmod{p}.$$
However, for $1\le i\le r$ one has
$$p^a\bfv_i+\bfxi=\bfxi_i\equiv \bfz_i=p^a\bfy_i+\bfxi\mmod{p^{a+1}},$$
so that $\bfy_i\equiv \bfv_i\mmod{p}$, and hence
$$\Del(\bfy_1,\ldots ,\bfy_r)\equiv \Del(\bfv_1,\ldots ,\bfv_r)\not\equiv 0\mmod{p}.$$
With the substitution $\bfz_i=p^a\bfy_i+\bfxi$ in (\ref{6.4}), therefore, we deduce that the set of solutions $\calD_1(\bfn)$ is in bijective correspondence with the set of solutions of the system of congruences
\begin{equation}\label{6.7}
\sum_{i=1}^r\sig_i\bfF(p^a\bfy_i+\bfxi-\bfeta)\equiv \bfn\mmod{p^{kb}},
\end{equation}
with $1\le \vbfy\le p^{kb-a}$ and $\Del_r(\vbfy)\not\equiv 0\mmod{p}$.\par

Let $\vbfy=\vbfw$ be any solution of the system (\ref{6.7}), if indeed such a solution exists. Then it follows that all other solutions $\vbfy$ satisfy the system of congruences
\begin{equation}\label{6.8}
\sum_{i=1}^r\sig_i\bfF(p^a\bfy_i+\bfxi-\bfeta)\equiv \sum_{i=1}^r\sig_i\bfF(p^a\bfw_i+\bfxi-\bfeta)\mmod{p^{kb}}.
\end{equation}
The translation invariance formula (\ref{2.3}) implies that
$$\bfF(\bfx)=C(\bfxi)^{-1}(\bfF(\bfx+\bfxi)-\bfc_0(\bfxi)),$$
in which the matrix $C(\bfxi)$ has determinant $1$. It follows that the system of congruences (\ref{6.8}) is equivalent to the new system
$$\sum_{i=1}^r\sig_i\bfF(p^a\bfy_i)\equiv \sum_{i=1}^r\sig_i\bfF(p^a\bfw_i)\mmod{p^{kb}}.$$
By homogeneity, moreover, this new system is in turn equivalent to
$$\sum_{i=1}^r\sig_iF_j(\bfy_i)\equiv \sum_{i=1}^r\sig_iF_j(\bfw_i)\mmod{p^{kb-k_ja}}\quad (1\le j\le r).$$

\par Next, we write $\calD_2(\bfu)$ for the set of solutions of the system of congruences
$$\sum_{i=1}^r\sig_iF_j(\bfy_i)\equiv u_j\mmod{p^{kb-k_ja}}\quad (1\le j\le r),$$
with $1\le \vbfy\le p^{kb-a}$ and $\Del_r(\vbfy)\not\equiv 0\mmod{p}$. Then it follows from our discussion thus far that
\begin{equation}\label{6.9}
\text{card}(\calD_1(\bfn))\le \max_{1\le \bfu\le p^{kb-a}}\text{card}(\calD_2(\bfu)).
\end{equation}
Denote by $\calD_3(\bfv)$ the set of solutions of the system of congruences
\begin{equation}\label{6.10}
\sum_{i=1}^r\sig_i\bfF(\bfy_i)\equiv \bfv\mmod{p^{kb-a}},
\end{equation}
with $1\le \vbfy\le p^{kb-a}$ and $\Del_r(\vbfy)\not\equiv 0\mmod{p}$. Then we have
\begin{align*}
\text{card}(\calD_2(\bfu))&\le \sum_{\substack{v_1\equiv u_1\mmod{p^{kb-k_1a}}\\ 1\le v_1\le p^{kb-a}}}\ldots \sum_{\substack{v_r\equiv u_r\mmod{p^{kb-k_ra}}\\ 1\le v_r\le p^{kb-a}}}\text{card}(\calD_3(\bfv))\\
&\le (p^a)^{K-r}\max_{1\le \bfv\le p^{kb-a}}\text{card}(\calD_3(\bfv)).
\end{align*}
Combining this estimate with (\ref{6.9}) and (\ref{6.6}), we derive the upper bound
\begin{equation}\label{6.11}
\text{card}(\calB_{a,b}^\bfsig(\bfm;\bfxi,\bfeta))\le (p^b)^{rk-K}(p^a)^{K-r}\max_{1\le \bfv\le p^{kb-a}}\text{card}(\calD_3(\bfv)).
\end{equation}

\par It is at this point that we prepare to apply Lemma \ref{lemma6.2}. For $1\le i\le r$, we consider a fixed choice for the $d-1$ coordinates $y_{ij}$ with $j\ne \sig(i)$. We then define the polynomials
$$f_j(y_{1,\sig(1)},\ldots ,y_{r,\sig(r)})=\sum_{i=1}^r\sig_iF_j(\bfy_i)-v_j\quad (1\le j\le r).$$
Consider the solutions $\vbfy$ of the system (\ref{6.10}) lying in $\calD_3(\bfv)$. For $1\le i\le r$, there are at most $p^{kb-a}$ possible choices for each coordinate $y_{ij}$ with $j\ne \sig(i)$. Write ${\widetilde \bfy}=(y_{1,\sig(1)},\ldots ,y_{r,\sig(r)})$. Then in the notation of the statement of Lemma \ref{lemma6.2}, one has
$$J(\bff;{\widetilde \bfy})=\text{det}(\partial_{\sig(i)}F_j(\bfy_i))_{1\le i,j\le r}=\Del_r(\vbfy)\not\equiv 0\mmod{p}.$$
Thus we may apply Lemma \ref{lemma6.2} to show that the number of solutions
$$y_{1,\sig(1)},\ldots ,y_{r,\sig(r)}$$
of the system of congruences
$$f_j(y_{1,\sig(1)},\ldots ,y_{r,\sig(r)})\equiv 0\mmod{p^{kb-a}}\quad (1\le j\le r),$$
with $1\le y_{i,\sig(i)}\le p^{kb-a}$ $(1\le i\le r)$, is at most $k_1\ldots k_r$. In this way, we conclude that
\begin{align*}
\text{card}(\calD_3(\bfv))&\le \sum_{\substack{1\le y_{1j}\le p^{kb-a}\\ 1\le j\le d\\ j\ne \sig(1)}}\ldots \sum_{\substack{1\le y_{rj}\le p^{kb-a}\\ 1\le j\le d\\ j\ne \sig (r)}}k_1\ldots k_r\\
&=k_1\ldots k_r(p^{kb-a})^{r(d-1)}.
\end{align*}
Substituting this estimate into (\ref{6.11}), we arrive at the upper bound
$$\text{card}(\calB_{a,b}^\bfsig (\bfm;\bfxi,\bfeta))\le k_1\ldots k_r (p^b)^{rk-K}(p^a)^{K-r}(p^{kb-a})^{r(d-1)},$$
and the conclusion of the lemma follows at once.
\end{proof}

\section{The conditioning process}\label{sec:7} Our next step involves extracting from the mean value $I_{a,b}^\bfsig (X;\bfxi,\bfeta)$ associated mean values conditioned so as to avoid singular solutions of an underlying system of congruences. Although motivated by the corresponding treatment for the classical Vinogradov system in \cite[\S5]{Woo2012a}, the less digestible singularity condition of the present work demands some modification.

\begin{lemma}\label{lemma7.1} Let $a$ and $b$ be integers with $b>a\ge 0$. Then one has
$$I_{a,b}(X)\ll K_{a,b}(X)+M^{2s(d-1)+r-1}I_{a,b+1}(X).$$
\end{lemma}

\begin{proof} Consider fixed integral $d$-tuples $\bfxi$ and $\bfeta$ with $1\le \bfxi\le p^a$ and $1\le \bfeta\le p^b$, and an $r$-tuple $\bfsig\in \Sig_r$. Then on considering the underlying Diophantine system, one finds from (\ref{4.14}) that $I_{a,b}^\bfsig (X;\bfxi,\bfeta)$ counts the number of integral solutions of the system
\begin{equation}\label{7.1}
\sum_{i=1}^r\sig_i(\bfF(\bfx_i)-\bfF(\bfy_i))=\sum_{l=1}^s(\bfF(\bfv_l)-\bfF(\bfv_{s+l})),
\end{equation}
with
$$1\le \vbfx,\vbfy,\vbfv\le X,\quad \bfv_l\equiv \bfeta\mmod{p^b}\quad (1\le l\le 2s),$$
and satisfying the property that there exist
$$(\bfxi_1,\ldots ,\bfxi_r)\in \Xi_a(\bfxi)\quad \text{and}\quad (\bfzet_1,\ldots ,\bfzet_r)\in \Xi_a(\bfxi)$$
for which
$$\bfx_i\equiv \bfxi_i\mmod{p^{a+1}}\quad \text{and}\quad \bfy_i\equiv \bfzet_i\mmod{p^{a+1}}\quad (1\le i\le r).$$
Let $T_1$ denote the number of integral solutions $\vbfx,\vbfy,\vbfv$ of the system (\ref{7.1}), counted by $I_{a,b}^\bfsig(X;\bfxi,\bfeta)$, satisfying the condition that for all $r$-tuples $(l_1,\ldots ,l_r)$ with
\begin{equation}\label{7.2}
1\le \bfl\le 2s,
\end{equation}
one has
$$\Del(\bfv_{l_1},\ldots ,\bfv_{l_r})\equiv 0\mmod{p^{(K-r)b+1}}.$$
Also, let $T_2$ denote the corresponding number of solutions satisfying the condition that for some $r$-tuple $\bfl$ satisfying (\ref{7.2}), one has
\begin{equation}\label{7.3}
\Del(\bfv_{l_1},\ldots ,\bfv_{l_r})\not\equiv 0\mmod{p^{(K-r)b+1}}.
\end{equation}
Then we have
\begin{equation}\label{7.4}
I_{a,b}^\bfsig (X;\bfxi,\bfeta)\le T_1+T_2.
\end{equation}

\par We consider first the solutions counted by $T_1$. Suppose that $\vbfx,\vbfy,\vbfv$ is a solution counted by $T_1$. For each index $l$ with $1\le l\le 2s$, we may rewrite $\bfv_l$ in the shape $\bfv_l=p^b\bfu_l+\bfeta$, for some integral $d$-tuple $\bfu_l$. Given an $r$-tuple $(l_1,\ldots ,l_r)$ satisfying (\ref{7.2}), it follows from Lemma \ref{lemma6.1} that
\begin{align*}
(p^b)^{K-r}\Del(\bfu_{l_1},\ldots ,\bfu_{l_r})&=\Del(p^b\bfu_{l_1}+\bfeta,\ldots ,p^b\bfu_{l_r}+\bfeta)\\
&=\Del(\bfv_{l_1},\ldots ,\bfv_{l_r})\equiv 0\mmod{p^{(K-r)b+1}},
\end{align*}
whence
$$\Del(\bfu_{l_1},\ldots ,\bfu_{l_r})\equiv 0\mmod{p}.$$
Write $\calA=\{1,2,\ldots ,p\}$ and $\dbF=\dbZ/p\dbZ$. Then it follows that $\vbfu\equiv \vbfnu\mmod{p}$ for some $\vbfnu\in \calS_{2s}(\calA;\dbF)$. Define $\calH_b$ to be the set of $2s$-tuples $(\bfeta_1,\ldots ,\bfeta_{2s})$ with $1\le \bfeta_l\le p^{b+1}$ $(1\le l\le 2s)$ satisfying the property that $(\bfeta_1,\ldots ,\bfeta_{2s})=(\bfeta+p^b\bfnu_1,\ldots ,\bfeta+p^b\bfnu_{2s})$ for some $(\bfnu_1,\ldots ,\bfnu_{2s})\in \calS_{2s}(\calA;\dbF)$. Then we have $(\bfv_1,\ldots ,\bfv_{2s})\equiv (\bfeta_1,\ldots ,\bfeta_{2s})\mmod{p^{b+1}}$ for some $(\bfeta_1,\ldots ,\bfeta_{2s})\in \calH_b$.\par

On considering the underlying Diophantine system, we deduce that
$$T_1\ll \sum_{(\bfeta_1,\ldots ,\bfeta_{2s})\in \calH_b}\oint |\grF_a^\bfsig (\bfalp;\bfxi)|^2|\grf_{b+1}(\bfalp;\bfeta_1)\ldots \grf_{b+1}(\bfalp;\bfeta_{2s})|\d\bfalp .$$
In view of the elementary inequality
$$|z_1\ldots z_n|\le |z_1|^n+\ldots +|z_n|^n,$$
we find from (\ref{4.14}) that
\begin{align*}
T_1&\ll \sum_{(\bfeta_1,\ldots ,\bfeta_{2s})\in \calH_b}\sum_{i=1}^{2s}\oint |\grF_a^\bfsig (\bfalp;\bfxi)^2\grf_{b+1}(\bfalp;\bfeta_i)^{2s}|\d\bfalp \\
&\ll \text{card}(\calH_b)\max_{1\le \bfeta_0\le p^{b+1}}I_{a,b+1}^\bfsig (X;\bfxi,\bfeta_0).
\end{align*}
But as a consequence of Lemma \ref{lemma5.3}, one has
$$\text{card}(\calH_b)=\text{card}(\calS_{2s}(\calA;\dbF))\ll p^{2s(d-1)+r-1}.$$
Thus we conclude from (\ref{4.16}) that
\begin{equation}\label{7.5}
T_1\ll M^{2s(d-1)+r-1}I_{a,b+1}(X).
\end{equation}

\par Next we turn our attention to the solutions $\vbfx,\vbfy,\vbfv$ counted by $T_2$. We may suppose that for some $r$-tuple $\bfl$ satisfying (\ref{7.2}), one has the congruence (\ref{7.3}). We define $\tau_i$ for $1\le i\le r$ by taking $\tau_i=1$ when $1\le l_i\le s$, and $\tau_i=-1$ when $s+1\le l_i\le 2s$. Notice that since $\bfv_l\equiv \bfeta\mmod{p^b}$ $(1\le l\le 2s)$, the condition (\ref{7.3}) implies that $(\bfv_{l_1},\ldots ,\bfv_{l_r})\equiv (\bfnu_{l_1},\ldots ,\bfnu_{l_r})\mmod{p^{b+1}}$ for some $(\bfnu_{l_1},\ldots ,\bfnu_{l_r})\in \Xi_b(\bfeta)$. Thus, on considering the underlying Diophantine system, we obtain the upper bound
$$T_2\ll \sum_{\bftau\in \Sig_r}\oint |\grF_a^\bfsig (\bfalp;\bfxi)^2\grF_b^\bftau(\bfalp;\bfeta)\grf_b(\bfalp;\bfeta)^{2s-r}|\d\bfalp .$$
On recalling that $s=ur$, an application of H\"older's inequality reveals from (\ref{4.14}) and (\ref{4.15}) that for some $\bftau\in \Sig_r$, one has
\begin{align*}
T_2\ll &\,\Bigl( \oint |\grF_a^\bfsig (\bfalp;\bfxi)^2\grF_b^\bftau (\bfalp;\bfeta)^{2u}|\d\bfalp \Bigr)^{1/(2u)}\\
&\, \times \Bigl( \oint |\grF_a^\bfsig(\bfalp;\bfxi)^2\grf_b(\bfalp;\bfeta)^{2s}|\d\bfalp \Bigr)^{1-1/(2u)}\\
=&\, (K_{a,b}^{\bfsig,\bftau}(X;\bfxi,\bfeta))^{1/(2u)}(I_{a,b}^\bfsig (X;\bfxi,\bfeta))^{1-1/(2u)}.
\end{align*}
We therefore conclude via (\ref{4.16}) and (\ref{4.17}) that
\begin{equation}\label{7.6}
T_2\ll (K_{a,b}(X))^{1/(2u)}(I_{a,b}(X))^{1-1/(2u)}.
\end{equation}

\par On combining (\ref{7.5}) and (\ref{7.6}) with (\ref{7.4}), we deduce that
$$I_{a,b}(X)\ll M^{2s(d-1)+r-1}I_{a,b+1}(X)+(K_{a,b}(X))^{1/(2u)}(I_{a,b}(X))^{1-1/(2u)},$$
whence
$$I_{a,b}(X)\ll K_{a,b}(X)+M^{2s(d-1)+r-1}I_{a,b+1}(X).$$
This completes the proof of the lemma.
\end{proof}

Repeated application of Lemma \ref{lemma7.1} shows that whenever $a$, $b$ and $H$ are non-negative integers with $b>a\ge 0$, then
\begin{equation}\label{7.7}
I_{a,b}(X)\ll \sum_{h=0}^{H-1}(M^h)^{2s(d-1)+r-1}K_{a,b+h}(X)+(M^H)^{2s(d-1)+r-1}I_{a,b+H}(X).
\end{equation}
We next show that $I_{a,b+H}(X)$ is negligible for $H$ large enough.

\begin{lemma}\label{lemma7.2} Let $a$, $b$ and $H$ be non-negative integers with
$$0<b-a\le H\le \tet^{-1}-b.$$
Then one has
$$(M^H)^{2s(d-1)+r-1}I_{a,b+H}(X)\ll M^{-H/2}(X/M^b)^{2sd}(X/M^a)^{2rd-K+\eta_{s+r}}.$$
\end{lemma}

\begin{proof} On considering the underlying system of Diophantine equations, we find from (\ref{4.14}) that when $1\le \bfxi\le p^a$ and $1\le \bfeta\le p^{b+H}$, and $\bfsig \in \Sig_r$, one has
$$I_{a,b+H}^\bfsig (X;\bfxi,\bfeta)\le \oint |\grf_a(\bfalp;\bfxi)^{2r}\grf_{b+H}(\bfalp;\bfeta)^{2s}|\d\bfalp .$$
An application of H\"older's inequality in combination with Lemma \ref{lemma5.1} therefore yields the estimate
\begin{align*}
I_{a,b+H}^\bfsig (X;\bfxi,\bfeta)\le &\, \Bigl( \oint |\grf_a(\bfalp;\bfxi)|^{2s+2r}\d\bfalp \Bigr)^{r/(s+r)}\\
&\, \times \Bigl( \oint |\grf_{b+H}(\bfalp;\bfeta)|^{2s+2r}\d\bfalp \Bigr)^{s/(s+r)}\\
\ll &\, (J_{s+r}(2X/M^a))^{r/(s+r)}(J_{s+r}(2X/M^{b+H}))^{s/(s+r)}.
\end{align*}
Consequently, on recalling (\ref{4.18}) and (\ref{4.21}), we obtain the upper bound
\begin{align}
I_{a,b+H}(X)&\ll \left( (X/M^a)^{r/(s+r)}(X/M^{b+H})^{s/(s+r)}\right)^{2(s+r)d-K+\eta_{s+r}+\del}\notag \\
&\ll X^\del (X/M^a)^{2rd-K+\eta_{s+r}}(X/M^b)^{2sd}\Ups,\label{7.8}
\end{align}
where
$$\Ups=(M^{b-a+H})^{Ks/(s+r)}M^{-2sdH}.$$
But when $H\ge b-a$, one has
\begin{align*}
H(2s(d-1)+r-1)+&(b-a+H)Ks/(s+r)-2sdH\\
&\le H(r-1-2s)+2HKs/(s+r).
\end{align*}
On observing that $s\ge rk\ge K$, and hence $(s+r)^2>K(rk+r)>Kr$, one finds that the expression $2Hs-2HKs/(s+r)$ achieves its minimum value for $s\ge rk$ when $s=rk$. Hence we deduce that
\begin{align*}
H(2s(d-1)+r-1)+&(b-a+H)Ks/(s+r)-2sdH\\
&\le -H+H(r-2rk)+2HKrk/(rk+r)\\
&\le -H+H(r-2rk/(k+1))\le -H.
\end{align*}
In this way, we see that for $k\ge 2$, one has
$$(M^H)^{2s(d-1)+r-1}\Ups\le M^{-H},$$
whence
$$X^\del (M^H)^{2s(d-1)+r-1}\Ups \le M^{-H/2}.$$
The conclusion of the lemma follows on substituting this estimate into (\ref{7.8}).
\end{proof}

Combining Lemma \ref{lemma7.2} with the upper bound (\ref{7.7}), we conclude as follows.

\begin{lemma}\label{lemma7.3} Let $a$ and $b$ be integers with $0\le a<b$, and put $H=b-a$. Suppose that $b+H\le \tet^{-1}$. Then there exists an integer $h$ with $0\le h<H$ having the property that
\begin{align*}
I_{a,b}(X)\ll &\, (M^h)^{2s(d-1)+r-1}K_{a,b+h}(X)\\
&\, +M^{-H/2}(X/M^b)^{2sd}(X/M^a)^{2rd-K+\eta_{s+r}}.
\end{align*}
\end{lemma}

The special case of Lemma \ref{lemma7.3} with $a=0$ and $b=1$ yields a refinement of Lemma \ref{lemma5.4} more easily utilised in what is to come.

\begin{lemma}\label{lemma7.4} One has $J_{s+r}(X)\ll M^{2sd}K_{0,1}(X)$.
\end{lemma}

\begin{proof} When $a=0$ and $b=1$, one has $b-a=1$. Thus we deduce from Lemma \ref{lemma7.3} that
$$I_{0,1}(X)\ll K_{0,1}(X)+M^{-1/2}(X/M)^{2sd}X^{2rd-K+\eta_{s+r}}.$$
Since we may suppose that $M>X^{4\del}$, it follows from Lemma \ref{lemma5.4} that
$$J_{s+r}(X)\ll M^{2sd}I_{0,1}(X)\ll M^{2sd}K_{0,1}(X)+X^{2(s+r)d-K+\eta_{s+r}-2\del}.$$
But in view of (\ref{4.18}) and (\ref{4.20}), one has
$$J_{s+r}(X)\gg X^{2(s+r)d-K+\eta_{s+r}-\del},$$
and thus we reach the upper bound
$$J_{s+r}(X)\ll M^{2sd}K_{0,1}(X)+X^{-\del}J_{s+r}(X).$$
The conclusion of the lemma follows on disentangling this inequality. 
\end{proof}

\section{The efficient congruencing step}\label{sec:8} The mean value $K_{a,b}(X)$ contains the powerful congruence conditions which drive the iterative process. In this section we convert these conditions into a form suited for further iteration.

\begin{lemma}\label{lemma8.1} Suppose that $a$ and $b$ are integers with $0\le a<b\le \tet^{-1}$. Then one has
$$K_{a,b}(X)\ll M^{2(kb-a)rd-K(b-a)}(J_{s+r}(2X/M^b))^{1-r/s}(I_{b,kb}(X))^{r/s}.$$
\end{lemma}

\begin{proof} Consider fixed $r$-tuples $\bfxi$ and $\bfeta$ with $1\le \bfxi\le p^a$ and $1\le \bfeta\le p^b$, and $r$-tuples $\bfsig,\bftau\in \Sig_r$. Then on considering the underlying Diophantine system, one finds that $K_{a,b}^{\bfsig,\bftau}(X;\bfxi,\bfeta)$ counts the number of integral solutions of the system
\begin{equation}\label{8.1}
\sum_{i=1}^r\sig_i(\bfF(\bfx_i)-\bfF(\bfy_i))=\sum_{l=1}^u\sum_{m=1}^r\tau_m(\bfF(\bfv_{lm})-\bfF(\bfw_{lm})),
\end{equation}
in which, for some $r$-tuples $\vbfzet,\vbfnu\in \Xi_a(\bfxi)$, one has
$$1\le \vbfx,\vbfy\le X,\quad \vbfx\equiv \vbfzet\mmod{p^{a+1}}\quad \text{and}\quad \vbfy\equiv \vbfnu\mmod{p^{a+1}},$$
and for $1\le l\le u$, for some $\vbfmu_l,\vbftet_l\in \Xi_b(\bfeta)$, one has
$$1\le \vbfv_l,\vbfw_l\le X,\quad \vbfv_l\equiv \vbfmu_l\mmod{p^{b+1}}\quad \text{and}\quad \vbfw_l\equiv \vbftet_l\mmod{p^{b+1}}.$$
The translation invariance formula (\ref{2.3}) implies that the system (\ref{8.1}) is equivalent to the new system of equations
\begin{equation}\label{8.2}
\sum_{i=1}^r\sig_i(\bfF(\bfx_i-\bfeta)-\bfF(\bfy_i-\bfeta))
=\sum_{l=1}^u\sum_{m=1}^r\tau_m(\bfF(\bfv_{lm}-\bfeta)-\bfF(\bfw_{lm}-\bfeta)).
\end{equation}
In any solution $\vbfx,\vbfy,{\overline \vbfv},{\overline \vbfw}$ counted by $K_{a,b}^{\bfsig,\bftau}(X;\bfxi,\bfeta)$, one has $\vbfv_l\equiv\vbfw_l\equiv \bfeta\mmod{p^b}$ $(1\le l\le u)$. We therefore deduce from (\ref{8.2}) that
\begin{equation}\label{8.3}
\sum_{i=1}^r\sig_iF_j(\bfx_i-\bfeta)\equiv \sum_{i=1}^r\sig_iF_j(\bfy_i-\bfeta)\mmod{p^{k_jb}}\quad (1\le j\le r).
\end{equation}
We also have $\vbfx\equiv \vbfy\equiv \bfxi\mmod{p^a}$,
$$\Del_r(\vbfx)\not\equiv 0\mmod{p^{(K-r)a+1}}\quad \text{and}\quad \Del_r(\vbfy)\not\equiv 0\mmod{p^{(K-r)a+1}}.$$

\par Recall the notation introduced prior to the statement of Lemma \ref{lemma6.1}, and write
$$\grG_{a,b}^\bfsig (\bfalp;\bfxi,\bfeta;\bfm)=\sum_{\vbfzet\in \calB_{a,b}^\bfsig (\bfm;\bfxi,\bfeta)}\prod_{i=1}^r\grf_{kb}(\sig_i\bfalp;\bfzet_i).$$
On considering the underlying Diophantine system, we deduce from (\ref{8.1}) and (\ref{8.3}) that
\begin{equation}\label{8.4}
K_{a,b}^{\bfsig,\bftau}(X;\bfxi,\bfeta)=\sum_{m_1=1}^{p^{k_1b}}\ldots \sum_{m_r=1}^{p^{k_rb}}\oint |\grG_{a,b}^\bfsig (\bfalp;\bfxi,\bfeta;\bfm)^2\grF_b^\bftau (\bfalp;\bfeta)^{2u}|\,\d\bfalp .
\end{equation}
By applying Cauchy's inequality in combination with the estimate supplied by Lemma \ref{lemma6.3}, we have
\begin{align*}
|\grG_{a,b}^\bfsig (\bfalp;\bfxi,\bfeta;\bfm)|^2&\le \text{card}(\calB_{a,b}^\bfsig(\bfm;\bfxi,\bfeta))\sum_{\vbfzet\in \calB_{a,b}^\bfsig(\bfm;\bfxi,\bfeta)}\prod_{i=1}^r|\grf_{kb}(\bfalp;\bfzet_i)|^2\\
&\ll M^{(kb-a)rd-K(b-a)}\sum_{\vbfzet\in \calB_{a,b}^\bfsig (\bfm;\bfxi,\bfeta)}\prod_{i=1}^r|\grf_{kb}(\bfalp;\bfzet_i)|^2.
\end{align*}
Substituting this relation back into (\ref{8.4}) and considering the underlying Diophantine system, we find that
\begin{align}
K_{a,b}^{\bfsig,\bftau}&(X;\bfxi,\bfeta)\notag \\
\ll &\, M^{(kb-a)rd-K(b-a)}\sum_{\substack{1\le \vbfzet\le p^{kb}\\ \vbfzet\equiv \bfxi\mmod{p^a}}}\oint \Bigl( \prod_{i=1}^r|\grf_{kb}(\bfalp;\bfzet_i)|^2\Bigr) |\grF_b^\bftau (\bfalp;\bfeta)|^{2u}\d\bfalp .\label{8.5}
\end{align}

\par Observe next that by H\"older's inequality,
\begin{align*}
\sum_{\substack{1\le \vbfzet\le p^{kb}\\ \vbfzet \equiv \bfxi\mmod{p^a}}}\prod_{i=1}^r|\grf_{kb}(\bfalp;\bfzet_i)|^2&=\Biggl( \sum_{\substack{1\le \bfzet\le p^{kb}\\\bfzet \equiv \bfxi\mmod{p^a}}}|\grf_{kb}(\bfalp;\bfzet)|^2\Biggr)^r\\
&\le (p^{kb-a})^{(r-1)d}\sum_{\substack{1\le \bfzet\le p^{kb}\\ \bfzet\equiv \bfxi\mmod{p^a}}}|\grf_{kb}(\bfalp;\bfzet)|^{2r}.
\end{align*}
Then (\ref{8.5}) delivers the upper bound
\begin{equation}\label{8.6}
K_{a,b}^{\bfsig,\bftau}(X;\bfxi,\bfeta)\ll M^{2(kb-a)rd-K(b-a)}\max_{1\le \bfzet\le p^{kb}}\oint|\grf_{kb}(\bfalp;\bfzet)^{2r}\grF_b^\bftau(\bfalp;\bfeta)^{2u}|\d\bfalp .
\end{equation}
Another application of H\"older's inequality yields the bound
$$\oint |\grf_{kb}(\bfalp;\bfzet)^{2r}\grF_b^\bftau (\bfalp;\bfeta)^{2u}|\d\bfalp \le U_1^{1-r/s}U_2^{r/s},$$
where
$$U_1=\oint |\grF_b^\bftau(\bfalp;\bfeta)|^{2u+2}\d\bfalp$$
and
$$U_2=\oint |\grF_b^\bftau (\bfalp;\bfeta)^2\grf_{kb}(\bfalp;\bfzet)^{2s}|\d\bfalp .$$
On considering the underlying Diophantine system, it follows from Lemma \ref{lemma5.1} that
$$U_1\le \oint |\grf_b(\bfalp;\bfeta)|^{2s+2r}\d\bfalp \ll J_{s+r}(2X/M^b),$$
whilst from (\ref{4.14}) we have $U_2=I_{b,kb}^\bftau (X;\bfeta,\bfzet)$. Thus we conclude that
$$\oint |\grf_{kb}(\bfalp;\bfzet)^{2r}\grF_b^\bftau(\bfalp;\bfeta)^{2u}|\d\bfalp \ll (J_{s+r}(2X/M^b))^{1-r/s}(I_{b,kb}(X))^{r/s}.$$
On substituting this estimate into (\ref{8.6}), the conclusion of the lemma follows.
\end{proof}

We conclude this section by extracting two simplified bounds that may be conveniently deployed in our iteration.

\begin{lemma}\label{lemma8.2}
Suppose that $a$ and $b$ are integers with $0\le a<b\le \tet^{-1}$. Then
$$\llbracket K_{a,b}(X)\rrbracket \ll X^{\eta_{s+r}+\del}(M^{b-a})^K.$$
\end{lemma}

\begin{proof} Consider fixed $d$-tuples $\bfxi$ and $\bfeta$ with $1\le \bfxi\le p^a$ and $1\le \bfeta\le p^b$, and $r$-tuples $\bfsig,\bftau\in \Sig_r$. Considering the underlying Diophantine system and then applying H\"older's inequality, we obtain
\begin{align*}
K_{a,b}^{\bfsig,\bftau}(X;\bfxi,\bfeta)&\le \oint |\grf_a(\bfalp;\bfxi)^{2r}\grf_b(\bfalp;\bfeta)^{2s}|\d\bfalp \\
&\le \Bigl( \oint |\grf_a(\bfalp;\bfxi)|^{2s+2r}\d\bfalp \Bigr)^{r/(s+r)}\Bigl( \oint |\grf_b(\bfalp;\bfeta)|^{2s+2r}\d\bfalp \Bigr)^{s/(s+r)}.
\end{align*}
Next applying Lemma \ref{lemma5.1}, we deduce that
$$K_{a,b}(X)\ll (J_{s+r}(2X/M^a))^{r/(s+r)}(J_{s+r}(2X/M^b))^{s/(s+r)},$$
whence by (\ref{4.19}) we arrive at the upper bound
\begin{align*}
\llbracket K_{a,b}(X)\rrbracket &\ll \frac{X^\del ((X/M^a)^{r/(s+r)}(X/M^b)^{s/(s+r)})^{2(s+r)d-K+\eta_{s+r}}}{(X/M^b)^{2sd}(X/M^a)^{2rd-K}}\\
&\ll X^{\eta_{s+r}+\del}(M^{b-a})^{Ks/(s+r)}.
\end{align*}
The conclusion of the lemma follows.
\end{proof}

The basic iterative relation follows by applying Lemma \ref{lemma7.3} in combination with Lemma \ref{lemma8.1}.

\begin{lemma}\label{lemma8.3} Suppose that $a$ and $b$ are integers with $0\le a<b\le \tfrac{1}{2}(k\tet)^{-1}$, and put $H=(k-1)b$. Then there exists an integer $h$, with $0\le h<H$, having the property that
\begin{align*}
\llbracket K_{a,b}(X)\rrbracket \ll&\, X^\del M^{-(2s-r+1)hr/s}(X/M^b)^{\eta_{s+r}(1-r/s)}\llbracket K_{b,kb+h}(X)\rrbracket^{r/s}\\
&\, +M^{-rH/(3s)}(X/M^b)^{\eta_{s+r}}.
\end{align*}
\end{lemma}

\begin{proof} We deduce from Lemma \ref{lemma8.1} via (\ref{4.19}) that
\begin{equation}\label{8.7}
\llbracket K_{a,b}(X)\rrbracket \ll (M^b)^{2sd}(M^a)^{2rd-K}M^{2(kb-a)rd-K(b-a)}T_1^{1-r/s}T_2^{r/s},
\end{equation}
where
$$T_1=\frac{J_{s+r}(2X/M^b)}{X^{2(s+r)d-K}}\quad \text{and}\quad T_2=\frac{I_{b,kb}(X)}{X^{2(s+r)d-K}}.$$
But
\begin{equation}\label{8.8}
T_1\ll (M^{-b})^{2(s+r)d-K}(X/M^b)^{\eta_{s+r}+\del}.
\end{equation}
Also, on putting $H=(k-1)b$, we see that
$$kb+H=(2k-1)b<\tet^{-1}.$$
Then it follows from Lemma \ref{lemma7.3} that there exists an integer $h$ with $0\le h<H$ having the property that
$$T_2\ll \frac{(M^h)^{2s(d-1)+r-1}K_{b,kb+h}(X)}{X^{2(s+r)d-K}}+\frac{M^{-H/2}(X/M^b)^{\eta_{s+r}}}{(M^{kb})^{2sd}(M^b)^{2rd-K}}.$$
Thus we see that
\begin{equation}\label{8.9}
T_2\ll (M^{-kb})^{2sd}(M^{-b})^{2rd-K}\Ome,
\end{equation}
where
$$\Ome =M^{-(2s-r+1)h}\llbracket K_{b,kb+h}(X)\rrbracket +M^{-H/2}(X/M^b)^{\eta_{s+r}}.$$

\par On substituting (\ref{8.8}) and (\ref{8.9}) into (\ref{8.7}), we deduce that
$$\llbracket K_{a,b}(X)\rrbracket \ll M^{\ome(a,b)}(X/M^b)^{(1-r/s)(\eta_{s+r}+\del)}\Ome^{r/s},$$
where
\begin{align*}
\ome(a,b)=&\, 2sdb+(2rd-K)a+2(kb-a)rd-K(b-a)\\
&\,-(1-r/s)(2(s+r)d-K)b-(2sdkb+(2rd-K)b)r/s.
\end{align*}
A little effort reveals that $\ome(a,b)=0$, and thus we obtain
\begin{align*}
\llbracket K_{a,b}(X)\rrbracket \ll &\, (M^{-H/2})^{r/s}(X/M^b)^{\eta_{s+r}+\del(1-r/s)}\\
&\, +X^\del M^{-(2s-r+1)hr/s}(X/M^b)^{\eta_{s+r}(1-r/s)}\llbracket K_{b,kb+h}(X)\rrbracket^{r/s}.
\end{align*}
The conclusion of the lemma follows on noting that $\del$ is assumed small enough that $(X/M^b)^{\del (1-r/s)}\ll M^{rH/(6s)}$.
\end{proof}

\section{The iterative process}\label{sec:9} Beginning with an application of Lemma \ref{lemma7.4}, which bounds $J_{s+r}(X)$ in terms of $K_{0,1}(X)$, we may apply Lemma \ref{lemma8.3} to bound $J_{s+r}(X)$ in terms of $K_{a,b}(X)$ for an increasing sequence of parameters $a$ and $b$. Our goal in this section is to manage this process, deriving useful information from the iterations. Our first step is to extract from Lemma \ref{lemma8.3} a conclusion transparent enough to be applied as the basic tool in each iterative step.

\begin{lemma}\label{lemma9.1} Suppose that $a$ and $b$ are integers with $0\le a<b\le \tfrac{1}{2}(k\tet)^{-1}$. Suppose in addition that there exist non-negative numbers $\psi$, $c$ and $\gam$, with $c\le 3(s/r)^N$, for which
\begin{equation}\label{9.1}
X^{\eta_{s+r}(1+\psi \tet)}\ll X^{c\del}M^{-\gam}\llbracket K_{a,b}(X)\rrbracket .
\end{equation}
Then, for some non-negative integer $h$ with $h\le (k-1)b$, one has
$$X^{\eta_{s+r}(1+\psi^\prime\tet)}\ll X^{c^\prime \del}M^{-\gam^\prime}\llbracket K_{a^\prime,b^\prime}(X)\rrbracket ,$$
where
$$\psi^\prime =(s/r)\psi+(s/r-1)b,\quad \gam^\prime =(s/r)\gam+(2s-r+1)h,$$
$$c^\prime =(s/r)(c+1),\quad a^\prime=b\quad \text{and}\quad b^\prime=kb+h.$$
\end{lemma}

\begin{proof} We are at liberty to assume that $c\le 3(s/r)^N$ and $\del<(Ns)^{-3N}$, so we have
$$c\del<\tfrac{1}{3}s^{-2N}<\tet/(3s),$$
and hence $X^{c\del}<M^{1/(3s)}$. In addition, one has $M^{1/(3s)}>X^\del$. We therefore deduce from Lemma \ref{lemma8.3} that there exists an integer $h$ with $0\le h<(k-1)b$ having the property that
\begin{align*}
\llbracket K_{a,b}(X)\rrbracket \ll &\, M^{-r/(3s)}X^{\eta_{s+r}}\\
&\, +X^\del (X/M^b)^{(1-r/s)\eta_{s+r}}\left( M^{-(2s-r+1)h}\llbracket K_{b,kb+h}(X)\rrbracket \right)^{r/s}.
\end{align*}
On making use of the hypothesised bound (\ref{9.1}), and employing in addition the lower bound $r\ge 2$ to confirm that $X^{c\del}M^{-r/(3s)}\ll X^{-\del}$, we therefore obtain the estimate
\begin{align*}
X^{\eta_{s+r}(1+\psi \tet)}\ll &\, X^{(c+1)\del}M^{-\gam-(2s-r+1)rh/s}(X/M^b)^{(1-r/s)\eta_{s+r}}\llbracket K_{b,kb+h}(X)\rrbracket^{r/s}\\
&+\,X^{\eta_{s+r}-\del},
\end{align*}
whence
$$X^{\eta_{s+r}(r/s+(\psi+(1-r/s)b)\tet)}\ll X^{(c+1)\del}M^{-\gam -(2s-r+1)rh/s}\llbracket K_{b,kb+h}(X)\rrbracket^{r/s}.$$
The desired conclusion now follows on raising left and right hand sides here to the power $s/r$.
\end{proof}

Our final task is to analyse the growth of the parameters as the iteration proceeds, since from this we are able to extract the conclusion of Theorem \ref{theorem2.1}.

\begin{lemma}\label{lemma9.2}
Put $s=rk$. Then one has $\eta_{s+r}=0$.
\end{lemma}

\begin{proof} We may suppose that $\eta_{s+r}>0$, for otherwise there is nothing to prove. We begin by defining three sequences $(a_n)$, $(b_n)$, $(h_n)$ of non-negative integers for $0\le n\le N$. We put $a_0=0$ and $b_0=1$. Then, when $0\le n<N$, we fix any integer $h_n$ with $0\le h_n\le (k-1)b_n$, and then define
\begin{equation}\label{9.2}
a_{n+1}=b_n\quad \text{and}\quad b_{n+1}=kb_n+h_n.
\end{equation}
We next define the auxiliary sequences $(\psi_n)$, $(c_n)$, $(\gam_n)$ of non-negative real numbers for $0\le n\le N$ by putting $\psi_0=0$, $c_0=1$, $\gam_0=0$. Then, for $0\le n<N$, we define
\begin{align}
\psi_{n+1}&=(s/r)\psi_n+(s/r-1)b_n,\label{9.3}\\
c_{n+1}&=(s/r)(c_n+1),\label{9.4}\\
\gam_{n+1}&=(s/r)\gam_n+(2s-r+1)h_n.\label{9.5}
\end{align}
It is apparent that $\gam_n$ is non-negative for $n\ge 0$, and an inductive argument shows that for $0\le n\le N$, one has
$$c_n=\frac{2s-r}{s-r}(s/r)^n-\frac{s}{s-r}\le \Bigl( 2+\frac{1}{k-1}\Bigr)(s/r)^n\le 3(s/r)^n.$$

\par We claim that a choice may be made for the sequence $(h_n)$ in such a manner that for $0\le n\le N$, one has
\begin{equation}\label{9.6}
b_n<\sqrt{N}(s/r)^n
\end{equation}
and
\begin{equation}\label{9.7}
X^{\eta_{s+r}(1+\psi_n\tet)}\ll X^{c_n\del}M^{-\gam_n}\llbracket K_{a_n,b_n}(X)\rrbracket .
\end{equation}
When $n=0$, the relation (\ref{9.6}) holds by virtue of the definition of $b_0$, and the relation (\ref{9.7}) holds as a consequence of (\ref{4.19}), (\ref{4.20}) and Lemma \ref{lemma7.4}, since the latter implies that
$$X^{\eta_{s+r}-\del}<\llbracket J_{s+r}(X)\rrbracket \ll \llbracket K_{0,1}(X)\rrbracket .$$

\par Before considering larger indices $n$, we conduct a preliminary analysis of the recurrence relations (\ref{9.2})-(\ref{9.5}). Observe that when $n\ge 0$, one has
$$\gam_{n+1}=(s/r)\gam_n+(2s-r+1)(b_{n+1}-kb_n),$$
whence
$$\gam_{n+1}-(2s-r+1)b_{n+1}=(s/r)(\gam_n-(2s-r+1)b_n).$$
Then we deduce by induction that
\begin{align}
\gam_n&=(2s-r+1)b_n+(s/r)^n(\gam_0-(2s-r+1)b_0)\notag \\
&=(2s-r+1)(b_n-(s/r)^n).\label{9.8}
\end{align}

\par Suppose next that the desired conclusions (\ref{9.6}) and (\ref{9.7}) have been established for the index $n<N$. Then from (\ref{9.6}) and (\ref{4.11}) one has $kb_n\tet<k(s/r)^{n-N-2}<\tfrac{1}{2}$, whence $b_n<\tfrac{1}{2}(k\tet)^{-1}$. By appealing to Lemma \ref{lemma9.1} we deduce from (\ref{9.7}) that there exists a non-negative integer $h$, with $h\le (k-1)b_n$, for which one has the upper bound
\begin{equation}\label{9.9}
X^{\eta_{s+r}(1+\psi^\prime\tet)}\ll X^{c^\prime \del}M^{-\gam^\prime}\llbracket K_{a^\prime,b^\prime}(X)\rrbracket ,
\end{equation}
where
\begin{align}
a^\prime&=b_n=a_{n+1},\quad b^\prime=kb_n+h,\label{9.10}\\
\psi^\prime&=(s/r)\psi_n+(s/r-1)b_n=\psi_{n+1},\notag\\
c^\prime&=(s/r)(c_n+1)=c_{n+1},\notag\\
\gam^\prime&=(s/r)\gam_n+(2s-r+1)h.\label{9.11}
\end{align}

\par Suppose, if possible, that $b'\ge \sqrt{N}(s/r)^{n+1}$. The relations (\ref{9.8}), (\ref{9.10}) and (\ref{9.11}) then show that
\begin{align}
\gam^\prime&=(s/r)(2s-r+1)(b_n-(s/r)^n)+(2s-r+1)(b^\prime -kb_n)\notag \\
&=(2s-r+1)(b^\prime -(s/r)^{n+1})\notag \\
&\ge (1-1/\sqrt{N})(2s-r+1)b'.\label{9.12}
\end{align}
But $b^\prime =kb_n+h\le (2k-1)b_n<\tet^{-1}$, and so it follows from Lemma \ref{lemma8.2} that
$$\llbracket K_{a^\prime,b^\prime}(X)\rrbracket \ll X^{\eta_{s+r}+\del}(M^{b^\prime})^K.$$
On substituting this estimate together with (\ref{9.12}) into (\ref{9.9}), we obtain the upper bound
$$X^{\eta_{s+r}(1+\psi_{n+1}\tet)}\ll X^{\eta_{s+r}+(c_{n+1}+1)\del}(M^{b^\prime})^{K-(2s-r+1)(1-1/\sqrt{N})}.$$
We recall that $c_{n+1}\le 3(s/r)^{n+1}$, so that $X^{(c_{n+1}+1)\del}<M^{1/2}$. Also,
$$K-(1-1/\sqrt{N})(2s-r+1)\le rk-(2rk-r+1)+2s/\sqrt{N}<-1.$$
Thus we obtain
$$X^{\eta_{s+r}(1+\psi_{n+1}\tet)}\ll X^{\eta_{s+r}}M^{1-b^\prime}\ll X^{\eta_{s+r}}M^{-1}.$$
Since $\psi_{n+1}$ and $\tet$ are both positive, we are forced to conclude that $\eta_{s+r}<0$, contradicting our opening hypothesis. The assumption that $b^\prime\ge \sqrt{N}(s/r)^{n+1}$ is therefore untenable, and so we must in fact have 
$b^\prime <\sqrt{N}(s/r)^{n+1}$. We now take $h_n$ to be the integer $h$ at hand, so that $b^\prime=b_{n+1}$ and $\gam^\prime=\gam_{n+1}$, and thus we confirm the upper bounds (\ref{9.6}) and (\ref{9.7}) with $n$ replaced by $n+1$.\par

\par We now collect together our various bounds on the parameters in question. We have (\ref{9.6}) and (\ref{9.7}) for $0\le n\le N$, and also the bounds $c_n\le 3(s/r)^n$ and $\gam_n\ge 0$. Also, by induction one finds that $b_n\ge k^n$ and
$$\psi_{n+1}=k\psi_n+(k-1)b_n\ge k\psi_n+(k-1)k^n,$$
whence $\psi_n\ge n(k-1)k^{n-1}$. Finally, one has $b_N\tet<(r/s)^2<\frac{1}{2}$, so that $b_N<\tet^{-1}$. By applying Lemma \ref{lemma8.2} in combination with (\ref{9.7}), we therefore obtain the estimate
$$X^{\eta_{s+r}(1+\psi_N\tet)}\ll X^{\eta_{s+r}+(c_N+1)\del}(M^{b_N})^K\ll X^{\eta_{s+r}+rk}.$$

\par Making use again of the relation $\tet=N^{-1/2}(r/s)^{N+2}$ from (\ref{4.11}), we conclude that
$$\eta_{s+r}\le \frac{rk}{\psi_N\tet}\le \frac{\sqrt{N}rk(s/r)^{N+2}}{N(k-1)k^{N-1}}<\frac{rk^4}{\sqrt{N}}.$$
We may take $N$ to be as large as necessary in terms of $k$ and $r$, and thus $\eta_{s+r}$ can be made arbitrarily small. We are therefore forced to conclude that $\eta_{s+r}=0$, and this completes the proof of the lemma.
\end{proof}

The conclusion of Theorem \ref{theorem2.1} is an immediate consequence of Lemma \ref{lemma9.2}, for in view of (\ref{4.18}) and (\ref{4.21}) the latter shows that when $s\ge r(k+1)$, then one has
$$J_s(X;\bfF)\ll X^{2sd-K+\eps}.$$

\section{Estimates of Weyl type}\label{sec:10} The extraction of estimates analogous to that of Weyl is in general difficult, owing to the complexity of the multidimensional situation. Although it would be feasible, with extra space and care, to analyse directly the exponential sum $f(\bfalp;X;\bfF)$ encoding quite general translation-dilation invariant systems $\bfF$, we have chosen here to instead restrict attention to the Weyl sums associated with the system described in example (b) of \S\ref{sec:2}. Such Weyl sums may be applied so as to bound the apparently more general sums $f(\bfalp;X;\bfF)$ mentioned above, at the cost of somewhat weaker estimates. Since our upper bounds for $J_s(X;\bfF)$ are essentially optimal with $s=r(k+1)$, it transpires that the slight weakening of the Weyl exponent does not impede the bulk of applications.\par

We stress then that throughout this section, until indicated otherwise, the system $\bfF$ is understood to be defined by
$$\bfF=(z_1^{i_1}z_2^{i_2}\ldots z_d^{i_d}:1\le |\bfi|\le k),$$
and the corresponding Weyl sum $f(\bfalp;X;\bfF)$ we abbreviate to $f(\bfalp)$. We define $r$ and $K$ as in (\ref{2.6}) and (\ref{2.7}). We must also consider the system $\bfF^\prime$ defined by
$$\bfF^\prime=(z_1^{i_1}z_2^{i_2}\ldots z_d^{i_d}:1\le |\bfi|\le k-1).$$
This system has rank
$$r^\prime=\binom{k+d-1}{d}-1,$$
and weight
$$K^\prime=\frac{d}{d+1}(r^\prime +1)(k-1).$$
The hard work involved in estimating $f(\bfalp)$ has been presented by the first author in \cite[\S5]{Par2005}, though here we take the opportunity to clarify one or two issues.

\begin{theorem}\label{theorem10.1}
Fix an index $\bfj$ with $2\le |\bfj|\le k$, and put $\sig=(2r^\prime k)^{-1}$. Let $\bfalp\in \dbR^r$, and suppose that $a$ and $q$ are integers with $q\ge 1$, $(a,q)=1$ and $|q\alp_\bfj-a|\le q^{-1}$. Then one has
$$|f(\bfalp)|\ll X^{d+\eps}(q^{-1}+X^{-1}+qX^{-|\bfj|})^\sig .$$
\end{theorem}
 
\begin{proof} The conclusion of the theorem is essentially immediate from \cite[Theorem 5.1]{Par2005}. We apply Theorem \ref{theorem2.1} to show that when $s=r^\prime k$, then one has
$$J_s(X;\bfF^\prime )\ll X^{2sd-K^\prime +\Del},$$
with $\Del\le \eps$. Since $\sig=1/(2s)$, the upper bound for $|f(\bfalp)|$ follows from the aforementioned estimate \cite[Theorem 5.1]{Par2005}.
\end{proof}

We add a few words of clarification to the proof of the latter conclusion in order to serve our purposes in the proof of Theorem \ref{theorem10.2}. The index $\bfj=(j_1,\ldots ,j_d)$ contains at least one coordinate $j_l$ satisfying $j_l\ge 1$. By relabelling variables, if necessary, one may suppose that $l=1$. At the top of page 24 of \cite{Par2005}, it is asserted that there is no loss of generality in supposing that $j_1\ge 1$, an assertion that is made otiose given our relabelling of variables. However, it is apparent that the argument of this proof may be reorganised so that relabelling is unnecessary, with the central elements of the argument focused on the index $j_l$ instead of $j_1$. The only issue to stress is that the set $\calM\subseteq [1,N]\cap \dbZ$ prepared at the end of the proof of \cite[Theorem 5.1]{Par2005} now depends on $l$.\par

Our next estimate refines Theorem \ref{theorem10.1} so that many coefficients are approximated simultaneously with control over a common denominator. In this context, when $\tet\in \dbR$, we define $\|\tet\|={\underset{y\in \dbZ}{\min}}|\tet-y|$. 

\begin{theorem}\label{theorem10.2} Let $\nu$ be a positive number, let $A$ be a real number satisfying $1\le A\le X^d$, and write $Q=X^\nu(X^dA^{-1})^{2r^\prime k}$. Suppose that $|f(\bfalp)|\ge A$, and that $Q\ll X^{1-2\del}$ for some $\del>0$. Let $l$ be an integer with $1\le l\le d$. Then for each index $\bfj=(j_1,\ldots ,j_d)$ with $2\le |\bfj|\le k$ and $j_l\ge 1$, there exist $a_\bfj\in \dbZ$ and $q_\bfj\in \dbN$ with $(a_\bfj,q_\bfj)=1$ and $|q_\bfj\alp_\bfj-a_\bfj|\le QX^{\del-|\bfj|}$. Moreover, the least common multiple $q_0^{(l)}$ of the numbers $q_\bfj$ with $2\le |\bfj|\le k$ and $j_l\ge 1$ satisfies $q_0^{(l)}\ll Q(\log X)^{2r^\prime k}$ and
$$\|q_0^{(l)}\alp_\bfj\|\ll Q^2(\log X)^{2r^\prime k}X^{\del-|\bfj|}\quad (2\le |\bfj|\le k,\, j_l\ge 1).$$
\end{theorem}

\begin{proof} We apply the argument of the proof of \cite[Theorem 5.2]{Par2005}. Consider an index $\bfj=(j_1,\ldots ,j_d)$ with $2\le |\bfj|\le k$, and let $l$ be any integer with $j_l\ge 1$. It follows from Dirichlet's Theorem that there exist coprime integers $q_\bfj$ and $a_\bfj$ with
$$1\le q_\bfj\le Q^{-1}X^{|\bfj|-\del}\quad \text{and}\quad |q_\bfj\alp_\bfj-a_\bfj|\le QX^{\del -|\bfj|}.$$
Keeping in mind the discussion following the proof of Theorem \ref{theorem10.1}, we may apply the argument of the proof of \cite[Theorem 5.2]{Par2005} to deduce from Theorem \ref{theorem10.1} that $q_\bfj\ll Q(\log X)^{2r^\prime k}\ll X^{1-\del}$. Now fix an integer $x\in [1,X]$, and suppose that there is an integer $y\in [1,X]$ such that
$$\|(k!)^k\alp_\bfj(x-y)\|\le X^{1-|\bfj|}$$
for each index $\bfj$ with $2\le |\bfj|\le k$ satisfying $j_l\ge 1$. Then the argument of the proof of \cite[Theorem 5.2]{Par2005} following \cite[equation (5.8)]{Par2005} shows that $q_\bfj$ divides $(k!)^ka_\bfj(x-y)$ for each of the latter indices $\bfj$. Write $q_0^{(l)}$ for the least common multiple of the integers $q_\bfj$ with $2\le |\bfj|\le k$ and $j_l\ge 1$. Then the argument of the proof of \cite[Theorem 5.2]{Par2005} shows that
\begin{equation}\label{10.1}
q_0^{(l)}\ll Q(\log X)^{2r^\prime k}.
\end{equation}
Note here that our hypothesis $j_l\ge 1$ permits the relevant part of the argument of \cite[Theorem 5.1]{Par2005} to be applied successfully, reflecting the discussion above following the proof of Theorem \ref{theorem10.1}. In addition, for $2\le |\bfj|\le k$ and $j_l\ge 1$, one has
$$\|q_0^{(l)}\alp_\bfj\|\le q_0^{(l)}\|q_\bfj\alp_\bfj\|\ll Q(\log X)^{2r^\prime k}(QX^{\del-|\bfj|})=Q^2(\log X)^{2r^\prime k}X^{\del-|\bfj|}.$$
This completes the proof of the theorem.
\end{proof}

We remark that the statement of \cite[Theorem 5.2]{Par2005} should be modified to reflect the argument concluding the above proof, so that the conclusion of \cite[Theorem 5.2]{Par2005} asserts only that $q_0\ll Q^d(\log P)^{2sd}$. This issue follows through the work of \cite{Par2005} discussing Weyl sums. In particular, the conclusion of \cite[Theorem 1.2]{Par2005} should be modified to impose a condition of the shape $\sig^{-1}\ge \frac{4}{3}(d+1)rk\log (rk)$.\par

The next result is established via a Baker-style ``final coefficient lemma'' argument (see \cite[Lemma 4.6]{Bak1986}).

\begin{theorem}\label{theorem10.3} Let $k$ be an integer with $k\ge 2$, and let $\tau$ be a real number with $\tau^{-1}>(2r^\prime k+1)(d+1)$. Suppose that $|f(\bfalp)|\ge A\ge X^{d-\tau+\nu}$ for some $\nu>0$, and write $Y=(X^dA^{-1})^{k+\nu}$. Then there are integers $a_\bfj$ and $q$ satisfying
$$(q,\bfa)=1,\quad 1\le q\ll Y\quad \text{and}\quad |q\alp_\bfj-a_\bfj|\ll YX^{-|\bfj|}\quad (1\le |\bfj|\le k).$$
\end{theorem}

\begin{proof} This is immediate from the argument of the proof of \cite[Theorem 5.5]{Par2005}. Write $W=X^{1-(d+1)\tau}$. We put $s=r^\prime k$ and note that $\tau(2s+1)(d+1)<1$, whence
$$(X^dA^{-1})^{2s(d+1)}\ll (X^{\tau-\nu})^{2s(d+1)}\ll X^{1-(d+1)\tau-2s(d+1)\nu}=WX^{-2s(d+1)\nu}.$$
As in the proof of \cite[Theorem 5.5]{Par2005}, we may apply Theorem \ref{theorem10.2} to show that for $1\le l\le d$, there exist integers $q_0^{(l)}$ with the property that
$$1\le q_0^{(l)}\ll X^\nu (X^dA^{-1})^{2r^\prime k}(\log X)^{2r^\prime k}\ll W^{1/(d+1)}X^{-s\nu},$$
and satisfying the condition that whenever $2\le |\bfj|\le k$ and $j_l\ge 1$, then
$$\|q_0^{(l)}\alp_\bfj\|\ll X^{2\nu}(X^dA^{-1})^{4r^\prime k}(\log X)^{2r^\prime k}X^{\del -|\bfj|}\ll W^{2/(d+1)}X^{\del -|\bfj|}.$$
We take $q_0$ to be the least common multiple of $q_0^{(1)},\ldots ,q_0^{(d)}$. In this way, we deduce first that
$$1\le q_0\le q_0^{(1)}\ldots q_0^{(d)}\ll (W^{1/(d+1)}X^{-s\nu})^d\ll W^{d/(d+1)}.$$
Suppose next that $2\le |\bfj|\le k$. Then there is some index $l$ for which $j_l\ge 1$, and we have
$$\|q_0\alp_\bfj\|\le \Bigl( \prod_{\substack{1\le m\le d\\ m\ne l}}q_0^{(m)}\Bigr) \|q_0^{(l)}\alp_\bfj\| \ll (W^{1/(d+1)}X^{-s\nu})^{d-1}W^{2/(d+1)}X^{\del-|\bfj|}.$$
In this way, we find that whenever $2\le |\bfj|\le k$, then without any condition on $j_1,\ldots ,j_d$, one has
$$\|q_0\alp_\bfj\|\ll WX^{\del-|\bfj|}.$$
An application of \cite[Lemma 5.4]{Par2005} completes the proof of the theorem, just as in the proof of \cite[Theorem 5.5]{Par2005}.
\end{proof}

The conclusion of Theorem \ref{theorem1.3} follows at once from Theorem \ref{theorem10.3} as a special case. It seems worthwhile at this point to extract from Theorem \ref{theorem10.3} a conclusion that serves to estimate $f(\bfalp;X;\bfF)$ for general translation-dilation invariant systems $\bfF$.

\begin{theorem}\label{theorem10.4} Let $\bfF$ be a reduced translation-dilation invariant system of polynomials having dimension $d$, rank $r$ and degree $k$. Define the exponent $\mu$ by means of the relation
$$\mu^{-1}=\left(2k\binom{k+d-1}{d}-2k+1\right) (d+1).$$
Suppose that $|f(\bfalp;X;\bfF)|\ge A\ge X^{d-\mu+\nu}$ for some $\nu>0$. Write $Y=(X^dA^{-1})^{k+\nu}$. Then there are integers $a_j$ and $q$, satisfying
$$(q,\bfa)=1,\quad 1\le q\ll Y\quad \text{and}\quad |q\alp_j-a_j|\ll YX^{-k_j}\quad (1\le j\le r).$$
\end{theorem}

\begin{proof} By assumption, the polynomials $F_j(\bfx)$ $(1\le j\le r)$ are homogeneous, and satisfy the translation-dilation invariance relation (\ref{2.3}) for a suitable lower unitriangular matrix $C(\bfxi)$. Write $\bfF$ for the column vector $(F_j(\bfx))_{1\le j\le r}$, and $\bfX$ for the column vector $(\bfx^\bfi)_{1\le |\bfi|\le k}$, in which the entries are arranged in ascending colex order. Finally, put
$$\rho=\binom{k+d}{d}-1.$$ 
The monomials $\bfx^\bfi$ span the space generated by $\{F_1,\ldots ,F_r\}$, and so one may write $\bfF=\grA \bfX$, with $\grA$ an $r\times \rho$ matrix having integer entries depending only on the coefficients of $\bfF$. The linear independence of the system $\bfF$ ensures that there exists an invertible $\rho\times \rho$ matrix $\grB$, having rational entries depending only on the coefficients of the system $\bfF$, having the property that $\grA\grB$ is an $r\times \rho$ block matrix of the shape $\left( O\ I_r\right)$, with $O$ the $r\times (\rho-r)$ zero matrix, and $I_r$ the $r\times r$ identity matrix. Observe that $\grB$ will have a block structure which associates to $F_j$ monomials $\bfx^\bfi$ with $|\bfi|=k_j$.\par

Suppose that $|f(\bfalp;X;\bfF)|\ge A\ge X^{d-\mu+\nu}$, for some $\nu>0$. One has
$$f(\bfalp;X;\bfF)=\sum_{1\le \bfx\le X}e\Bigl( \sum_{i=1}^r\alp_iF_i(\bfx)\Bigr)=\sum_{1\le \bfx\le X}e\Bigl( \sum_{1\le |\bfj|\le k}\bet_\bfj \bfx^\bfj\Bigr) ,$$
wherein we have written $\bfbet=\grA^T\bfalp$. Thus we have
$$\Bigl| \sum_{1\le \bfx\le X}e\Bigl( \sum_{1\le |\bfj|\le k}\bet_\bfj \bfx^\bfj\Bigr) \Bigr|\ge A\ge X^{d-\mu+\nu}.$$
It follows from Theorem \ref{theorem10.3} that there exist integers $c_\bfj$ and $v$, with $(v,\bfc)=1$, satisfying
\begin{equation}\label{10.2}
1\le v\ll Y\quad \text{and}\quad |v\bet_\bfj-c_\bfj|\ll YX^{-|\bfj|}\quad (1\le |\bfj|\le k).
\end{equation}
Now $\grB^T\bfbet=(\grA\grB)^T\bfalp$. Since $\grA\grB=\left( O\ I_r\right)$, we see that
\begin{equation}\label{10.3}
(0,\ldots,0,\alp_1,\ldots ,\alp_r)^T=\grB^T\bfbet ,
\end{equation}
so that $\alp_1,\ldots ,\alp_r$ are given by linear combinations of the real numbers $\bet_\bfj$ $(1\le |\bfj|\le k)$, with rational coefficients depending at most on the coefficients of $\bfF$.\par

Let $\Ome_0$ be the least natural number having the property that $\Ome_0\grB$ has integral entries, and let $\Ome$ be the largest of the absolute values of the entries of $\Ome_0\grB$. Then it follows from (\ref{10.2}) and (\ref{10.3}) that for $1\le i\le r$, one has
$$\|\Ome_0v\alp_i\|\le \sum_{|\bfj|=k_i}\Ome \|v\bet_\bfj\|\ll \rho \Ome YX^{-k_i}.$$
Write $q_0=v\Ome_0$ and $Z=\rho \Ome Y$. Then we find that $1\le q_0\ll Z$ and there exist integers $b_i$ $(1\le i\le r)$ such that $|q_0\alp_i-b_i|\ll ZX^{-k_i}$ $(1\le i\le r)$. The conclusion of the theorem follows on putting $g=(q_0,\bfb)$, writing $q=q_0/g$ and $a_i=b_i/g$ $(1\le i\le r)$, and noting that $Z\ll Y$.
\end{proof}

\section{Asymptotic formulae associated with Diophantine equations}\label{sec:11} The mean value estimate supplied by Theorem \ref{theorem2.1} may be routinely combined with estimates of Weyl type, explored in \S\ref{sec:10}, so as to establish asymptotic formulae for the number of solutions of associated Diophantine systems. Such consequences have been examined already in the literature, and so our goal in this section is to sketch some conclusions, and outline the arguments necessary for their proofs. In particular, we deliberately avoid going into detail concerning proofs of these new results so as to avoid adding bulk to an already lengthy memoir.\par

We begin by establishing an asymptotic formula for a general counting problem of which Theorem \ref{theorem1.4} is essentially a special case. Consider then a reduced translation-dilation invariant system of polynomials $\bfF$ having dimension $d$, rank $r$, degree $k$ and weight $K$. Let $s$ be a natural number, and consider fixed non-zero integers $c_{ij}$ for $1\le i\le r$ and $1\le j\le s$. Finally, let $N_s(X;\bfF;\bfc)$ denote the number of integral solutions of the Diophantine system
\begin{equation}\label{11.1}
\sum_{j=1}^sc_{ij}F_i(\bfx_j)=0\quad (1\le i\le r),
\end{equation}
with $1\le \vbfx\le X$. For the time being, it is convenient to abbreviate $f(\bfalp;X;\bfF)$ to $f(\bfalp)$.

\begin{theorem}\label{theorem11.1} Suppose that $s\ge 2r(k+1)+1$. Suppose further that $c_{ij}$ $(1\le i\le r,\, 1\le j\le s)$ are non-zero integers, and that the system (\ref{11.1}) has both a non-singular real solution, and a non-singular $p$-adic solution for every prime $p$. Then there exist positive constants $\calD=\calD(s,\bfF,\bfc)$ and $\nu=\nu(s,\bfF,\bfc)$ such that
$$N_s(X;\bfF;\bfc)=\calD X^{sd-K}+O(X^{sd-K-\nu}).$$
\end{theorem}

\begin{proof} We begin by defining a Hardy-Littlewood dissection. When $0<\tet\le 1$, let $\grM_\tet$ denote the union of the boxes
$$\grM_\tet(q,\bfa)=\{ \bfalp\in [0,1)^r:|q\alp_i-a_i|\le X^{\tet-k_i}\ (1\le i\le r)\},$$
with $1\le q\le X^\tet$, $0\le \bfa\le q$ and $(q,\bfa)=1$. Complementing the {\it major arcs} $\grM_\tet$, we define the {\it minor arcs} $\grm_\tet=[0,1)^r\setminus \grM_\tet$. We claim that whenever $s\ge 2r(k+1)+1$, then for a positive number $\del=\del(s,\bfF)$, one has
\begin{equation}\label{11.2a}
\int_{\grm_{1/3}}|f(\bfbet)|^s\d\bfbet \ll X^{sd-K-\del}.
\end{equation}
Define $f_j(\bfalp)$ to be $f(c_{1j}\alp_1,\ldots ,c_{rj}\alp_r)$. Then from the upper bound (\ref{11.2a}) it follows by means of H\"older's inequality that
\begin{align}
\int_{\grm_{1/2}}\prod_{j=1}^sf_j(\bfalp)\d\bfalp &\le \prod_{j=1}^s\Bigl( \int_{\grm_{1/2}}|f_j(\bfalp)|^s\d\bfalp \Bigr)^{1/s}\notag \\
&\ll \max_{1\le j\le s}\int_{\grm_{1/3}}|f(\bfbet)|^s\d\bfbet \ll X^{sd-K-\del}.\label{11.2}
\end{align}

\par In order to confirm the estimate (\ref{11.2a}), observe first that
\begin{equation}\label{11.3}
\int_{\grm_{1/3}}|f(\bfalp)|^s\d\bfalp \le \left( \sup_{\bfalp\in {\grm_{1/3}}}|f(\bfalp)|\right)^{s-2r(k+1)}\oint |f(\bfalp)|^{2r(k+1)}\d\bfalp .
\end{equation}
By applying Theorem \ref{theorem2.1}, one sees that
\begin{equation}\label{11.4}
\oint |f(\bfalp)|^{2r(k+1)}\d\bfalp \ll X^{2r(k+1)-K+\eps}.
\end{equation}
Next, define $\sig$ by means of the relation
$$\sig^{-1}=6(d+1)k\binom{k+d-1}{d}.$$
If we hypothesise that $|f(\bfalp)|\ge X^{d-\sig}$, then it follows from Theorem \ref{theorem10.4} that there exist integers $a_i$ and $q$, with
$$1\le q\le X^{1/4},\quad (q,\bfa)=1\quad \text{and}\quad |q\alp_i-a_i|\ll X^{-k_i+1/4}\quad (1\le i\le r).$$
Then one must have $\bfalp\in \grM_{1/3}$, and thus we infer that
$$\sup_{\bfalp\in \grm_{1/3}}|f(\bfalp)|\le X^{d-\sig}.$$
The desired estimate (\ref{11.2a}) follows on combining (\ref{11.3}) and (\ref{11.4}), provided that we take $\del<\sig$. In view of our earlier discussion, this confirms the estimate (\ref{11.2}).\par

In order to estimate the contribution of the major arcs $\grM_{1/2}$, we may follow the argument of \cite[\S6]{Par2005}. We avoid providing many details here. We write
$$S(q,\bfa)=\sum_{1\le \bfx\le q}e\Bigl( q^{-1}\sum_{1\le i\le r}a_iF_i(\bfx)\Bigr)$$
and
$$v(\bfbet)=\int_{[0,X]^d}e\Bigl( \sum_{1\le i\le r}\bet_iF_i(\bfgam)\Bigr)\d\bfgam ,$$
and then put
$$V(\bfalp;q,\bfa)=q^{-d}S(q,\bfa)v(\bfalp-\bfa/q).$$
Define $V_j(\bfalp)$ to be
$$V(c_{1j}\alp_1,\ldots ,c_{rj}\alp_r;q,c_{1j}a_1,\ldots ,c_{rj}a_r)$$
when $\bfalp\in \grM_{1/2}(q,\bfa)\subseteq \grM_{1/2}$, and otherwise put $V_j(\bfalp)=0$. Then by adapting the argument of \cite[Lemma 5.3]{Par2005}, one finds that
\begin{equation}\label{11.5}
\sup_{\bfalp\in \grM_{1/2}}|f_j(\bfalp)-V_j(\bfalp)|\ll X^{d-1/2}.
\end{equation}
Since $\text{mes}(\grM_{1/2})\ll X^{(r+1)/2-K}$, it follows via H\"older's inequality that for some positive number $\nu$, one has
\begin{align*}
\int_{\grM_{1/2}}\prod_{j=1}^s f_j(\bfalp)\d\bfalp &-\int_{\grM_{1/2}}\prod_{j=1}^sV_j(\bfalp)\d\bfalp \\
& \ll X^{d-1/2}\max_{1\le j\le s}\int_{\grM_{1/2}}|V_j(\bfalp)|^{s-1}\d\bfalp +X^{sd-K-\nu}.
\end{align*}
Reversing course, one finds from (\ref{11.5}) via Theorem \ref{theorem2.1} that whenever $s\ge 2r(k+1)+1$, then
\begin{align*}
\int_{\grM_{1/2}}|V_j(\bfalp)|^{s-1}\d\bfalp &\ll \int_{\grM_{1/2}}|f_j(\bfalp)|^{s-1}\d\bfalp+X^{(s-1)d-K-\nu}\\
&\le \oint |f_j(\bfalp)|^{s-1}\d\bfalp +X^{(s-1)d-K-\nu}\\
&\ll X^{(s-1)d-K+\eps}.\end{align*}
Thus we deduce that
\begin{align}
N_s(X;\bfF;\bfc)&=\int_{\grM_{1/2}}\prod_{j=1}^sf_j(\bfalp)\d\bfalp +\int_{\grm_{1/2}}\prod_{j=1}^sf_j(\bfalp)\d\bfalp \notag \\
&=\int_{\grM_{1/2}}\prod_{j=1}^sV_j(\bfalp)\d\bfalp +O(X^{sd-K-\nu}).\label{11.6}
\end{align}

\par The argument of the proof of \cite[Theorem 6.2]{Par2005} is readily adapted to show that
\begin{equation}\label{11.7}
\int_{\grM_{1/2}}\prod_{j=1}^sV_j(\bfalp)\d\bfalp =\grJ\grS X^{sd-K}+O(X^{sd-K-\nu}),
\end{equation}
for some $\nu>0$, where
$$\grJ=\int_{\dbR^r}\int_{[0,1]^{sd}}e\Bigl( \sum_{i=1}^r\bet_i\sum_{j=1}^sc_{ij}F_i(\bfgam_j)\Bigr)\d\vbfgam\d\bfbet ,$$
and
$$\grS=\sum_{q=1}^\infty \sum_{\substack{1\le \bfa\le q\\ (q,\bfa)=1}}q^{-sd}\prod_{j=1}^sS(q,c_{1j}a_1,\ldots ,c_{rj}a_r).$$
The absolute convergence of $\grJ$ and $\grS$ follows via the methods of \cite[\S\S5 and 6]{Par2005}. Here, the existence of non-singular real and $p$-adic solutions suffices to guarantee that $\grJ>0$ and $\grS>0$ (see \cite{Par2001} and \cite{Par2005} for the necessary ideas). On substituting into (\ref{11.6}) and (\ref{11.7}), the desired conclusion follows.
\end{proof}

Note that, in order to count solutions of the system (\ref{11.1}) with $|\vbfx|\le X$, one may merely add together the contributions from the $2^d$ sectors accommodating the various constellations of signs amongst the $d$ coordinates, and thus Theorem \ref{theorem1.4} is an immediate consequence of Theorem \ref{theorem11.1} corresponding to the special translation-dilation invariant system
$$\bfF=(\bfx^\bfi:1\le |\bfi|\le k).$$
The special case of Theorem \ref{theorem11.1} in which $s=2r(k+1)+2$ and $c_{ij}=(-1)^j$ $(1\le j\le s)$ delivers the following corollary.

\begin{corollary}\label{corollary11.2} Suppose that $t\ge r(k+1)+1$. Then there are positive numbers $\grC=\grC(t,\bfF)$ and $\del=\del(t,\bfF)$ such that
$$J_t(X;\bfF)=\grC X^{2td-K}+O(X^{2td-K-\del}).$$
\end{corollary}

We remark that the positivity of the product of singular integral and singular series, giving $\grC>0$, is in this case a consequence of the lower bound provided by Theorem \ref{theorem3.1}. Note that Theorem \ref{theorem1.5} is a special case of Corollary \ref{corollary11.2}, corresponding to the same special choice of system as above.\par

We now move on to consider Theorem \ref{theorem1.6}. Here we may follow the sketch provided at the end of \cite[\S6]{Par2005}. We write
$$g_j(\bfalp)=\sum_{|\bfx|\le X}e\Bigl( \sum_{|\bfi|=k}c_j\alp_\bfi \bfx^\bfi\Bigr) .$$
Then it follows from Theorem \ref{theorem1.1} that whenever
$$s\ge \left( \binom{k+d}{d}-1\right)(k+1),$$
then
$$\oint |g_j(\bfalp)|^{2s}\d\bfalp \ll X^{2sd-L+\eps}.$$
The argument concluding \cite[\S6]{Par2005} then suffices to prove Theorem \ref{theorem1.6}.\par

Finally, we consider the translation-dilation invariant system $\bfF$, as in the preamble to Theorem \ref{theorem11.1}, with applications in additive combinatorics in mind. We therefore consider non-zero integers $c_1,\ldots ,c_s$ satisfying the condition $c_1+\ldots +c_s=0$, and we investigate the solubility of the Diophantine system
\begin{equation}\label{11.8}
c_1\bfF(\bfx_1)+\ldots +c_s\bfF(\bfx_s)={\mathbf 0}.
\end{equation}
We define {\it projected} and {\it subset-sum} solutions in a manner transparently analogous to that in the preamble to Theorem \ref{theorem1.7}.

\begin{theorem}\label{theorem11.3}
Suppose that $s\ge 2r(k+1)+1$ and $s>K+d^2$. Let $c_i$ $(1\le i\le s)$ be non-zero integers satisfying $c_1+\ldots +c_s=0$. Suppose further that the system of equations (\ref{11.8}) possesses non-singular real and $p$-adic solutions for each prime number $p$. Let $\calA \subseteq \dbZ^d\cap [1,N]^d$, and suppose that the only solutions of the system (\ref{11.8}) from $\calA$ are either projected or subset-sum solutions. Then one has
$${\rm card}(\calA)\ll N^d(\log \log N)^{-1/(s-1)}.$$
\end{theorem}

The conclusion of Theorem \ref{theorem1.7} is a special case of Theorem \ref{theorem11.3}, as we now confirm. For in the special circumstances relevant to the statement of Theorem \ref{theorem1.7} one has $rk\ge K$, and when $k\ge 2$ one has in addition
$$2r+1\ge 2\binom{d+k}{k}-1\ge (d+2)(d+1)-1>d^2.$$
Thus we find that the hypothesis $s\ge 2r(k+1)+1$ already ensures that $s>K+d^2$, and Theorem \ref{theorem1.7} consequently follows as a direct corollary of Theorem \ref{theorem11.3}.\par

The proof of Theorem \ref{theorem11.3} follows by adapting the methods of \cite{Pre2011} to this more general situation wherein $d$ may exceed $2$. The conclusion of Theorem \ref{theorem11.1} may be employed to show that when $s\ge 2r(k+1)+1$ and $\calA$ is of suitable linear uniformity with relative density $\del$, then the number $\calN(\calA)$ of solutions of (\ref{11.8}) with $\vbfx\in \calA^s$ satisfies
$$\calN(\calA)\gg_\bfc \del^sN^{sd-K}.$$
We claim that the number $\calN_0(\calA)$ of projected solutions is $O(N^{s(d-1)+d^2})$, and that the number $\calN_1(\calA)$ of subset-sum solutions is $O(N^{sd-K-1/s})$. Granted this claim, one finds that $\calN(\calA)>\calN_0(\calA)+\calN_1(\calA)$ provided that
$$N^{s(d-1)+d^2}+N^{sd-K-1/s}=o(\del^sN^{sd-K}),$$
and this relation is satisfied whenever $s>K+d^2$ and $\del\gg (\log N)^{-1}$, for example. Under such circumstances, we conclude that the system (\ref{11.8}) contains solutions from $\calA$ that are neither projected nor subset-sum solutions. In particular, if the only solutions of (\ref{11.8}) from $\calA$ are either projected or subset-sum solutions, then $\text{card}(\calA)\ll N^d(\log N)^{-1}$. On the other hand, if $\calA$ is not of suitable linear uniformity, then a concentration argument motivated by that of Roth \cite{Rot1953} and described in the proof of \cite[Lemma 5.3]{Pre2011} may be adapted to this potentially higher dimensional setting to show that in the absence of non-trivial solutions, a suitable sub-progression can be obtained with higher relative density than that of $\calA$. By iterating this argument, one deduces that in circumstances wherein the only solutions of the system (\ref{11.8}) from $\calA$ are either projected or subset-sum solutions, then $\text{card}(\calA)\ll N^d(\log \log N)^{-1/(s-1)}$. We refer the reader to \cite[\S5]{Pre2011} for details.\par

We have still to justify our earlier claim. We begin by counting the number of projected solutions $\calN_0(\calA)$. Let $\vbfx=(\bfx_1,\ldots ,\bfx_s)$ be any projected solution counted by $\calN_0(\calA)$. Then there exists a translate $\bfa\in \dbR^d$ with the property that $\text{span}_\dbR\{\bfx_1-\bfa,\ldots ,\bfx_s-\bfa\}$ is a vector space of dimension $m<d$. There is no loss of generality in taking $\bfa=\bfx_1$, and then $\text{span}_\dbR\{\bfx_1-\bfa,\ldots ,\bfx_s-\bfa\}$ has a basis of the shape $\calB=\{\bfx_{i_1}-\bfx_1,\ldots ,\bfx_{i_m}-\bfx_1\}$ for some indices $1<i_1<\ldots <i_m\le s$. Since $-N<\bfx_{i_l}-\bfx_1<N$ for $1\le l\le m$, one sees that the number of possible choices for $\calB$ is $O((N^d)^m)$. But there are trivially $O(N^d)$ possible choices for $\bfx_1$, and so we find that the number $T$ of possible translated vector spaces $\bfx_1+\text{span}_\dbR \calB$ satisfies
$$T\ll N^d(N^d)^m\le N^d(N^d)^{d-1}=N^{d^2}.$$
The integral vectors lying in $V=\text{span}_\dbR \calB$ form an integral lattice of dimension $m$, and hence volume considerations confirm that the number of vectors $\bfy\in \dbZ^d$ such that $-2N\le \bfy\le 2N$ and $\bfy\in V$ is at most $O(N^m)$. Thus we deduce that for each index $i$ with $1\le i\le s$, and each $m<d$, the number of possible choices for 
$\bfx_i-\bfx_1$ is at most $O(N^m)$. It follows that for fixed $\bfx_1$ and $\calB$, the number of possible choices for $\bfx_1,\ldots ,\bfx_s$ is at most $O(N^{ms})$. Then the total number of possible projected solutions $\bfx_1,\ldots ,\bfx_s$ is
$$\calN_0(\calA)\le N^{ms}T\ll N^{ms+d^2}\le N^{(d-1)s+d^2}.$$

We turn next to the task of bounding $\calN_1(\calA)$. The number of partitions $\{1,\ldots ,s\}=\calJ_1\cup\ldots \cup \calJ_l$, with $l\ge 2$ and the sets $\calJ_v$ disjoint and non-empty, is plainly $O_s(1)$. Let $\calN_2(N;\calJ_1,\ldots ,\calJ_l)$ denote the number of solutions of the system
$$\sum_{u\in \calJ_v}c_u\bfF(\bfx_u)={\mathbf 0}\quad (1\le v\le l),$$
with $1\le \vbfx\le N$. Then it follows that there exists a partition $\{1,\ldots ,s\}=\calJ_1\cup\ldots \cup \calJ_l$, of the aforementioned type, for which
$$\calN_1(\calA)\ll \calN_2(N;\calJ_1,\ldots ,\calJ_l).$$
Write $m_v=\text{card}(\calJ_v)$ $(1\le v\le l)$, and note that
\begin{equation}\label{11.ww}
m_1+\ldots +m_l=s.
\end{equation}
Finally, let $n$ denote the number of the sets $\calJ_v$ with $1\le v\le l$ satisfying the property that $\text{card}(\calJ_v)=1$.\par

Our next step is to consider the contribution of the subset-sum equations defined by sets $\calJ_v$, distinguishing two cases. Suppose first that $\text{card}(\calJ_v)=1$ and $\calJ_v=\{u\}$. Since we may suppose $c_u$ to be non-zero, it follows from Lemma \ref{lemma5.2} that the number of solutions of the system of equations $c_u\bfF(\bfx_u)={\mathbf 0}$, with $1\le \bfx_u\le N$, is $O(N^{d-1})$.\par

Suppose next that $\text{card}(\calJ_v)>1$. Abbreviating $f(\bfalp;N;\bfF)$ to $f(\bfalp)$, a trivial estimate yields
$$\oint |f(\bfalp)|^2\d\bfalp \ll N^{2d-1}.$$
On the other hand, a trivial estimate in combination with Theorem \ref{theorem2.1} delivers the bound
$$\oint |f(\bfalp)|^{s}\d\bfalp \ll N^{sd-K+\eps}.$$
Then it follows from H\"older's inequality and a change of variable that
\begin{align*}
\oint |f(c_u\bfalp)|^{m_v}\d\bfalp &\le \Bigl( \oint |f(\bfalp)|^s\d\bfalp \Bigr)^{(m_v-2)/(s-2)}\Bigl( \oint |f(\bfalp)|^2\d\bfalp \Bigr)^{(s-m_v)/(s-2)}\\
&\ll (N^{sd-K+\eps})^{(m_v-2)/(s-2)}(N^{2d-1})^{(s-m_v)/(s-2)}\\
&\ll N^{m_vd-m_vK/(s-n)-\nu_v+\eps},
\end{align*}
where
$$\nu_v=\left( \frac{m_v-2}{s-2}\right) K-\left( \frac{m_v}{s-n}\right) K+\frac{s-m_v}{s-2}.$$

\par From here, a consideration of the underlying Diophantine system, followed by an application of H\"older's inequality, reveals that
\begin{align*}
\calN_2(N;\calJ_1,\ldots ,\calJ_l)&\ll (N^{d-1})^n\prod_{\substack{1\le v\le l\\ \text{card}(\calJ_v)>1}}\oint \prod_{u\in \calJ_v}|f(c_u\bfalp)|\d\bfalp \\
&\ll N^{(d-1)n}\prod_{\substack{1\le v\le l\\ \text{card}(\calJ_v)>1}}\prod_{u\in \calJ_v}\Bigl( \oint |f(c_u\bfalp)|^{m_v}\d\bfalp \Bigr)^{1/m_v}\\
&\ll N^{(d-1)n}\prod_{\substack{1\le v\le l\\ \text{card}(\calJ_v)>1}}N^{m_vd-m_vK/(s-n)-\nu_v+\eps}.
\end{align*}
We therefore deduce from (\ref{11.ww}) that
$$\calN_1(\calA)\ll N^{sd-K-\nu+\eps},$$
where
$$\nu=n+\sum_{\substack{1\le v\le l\\ \text{card}(\calJ_v)>1}}\left(\left( \frac{m_v-2}{s-2}\right) K-\left( \frac{m_v}{s-n}\right) K+\frac{s-m_v}{s-2}\right).$$
On recalling (\ref{11.ww}), we see that
$$\nu=n+\left(\frac{(s-n)-2(l-n)}{s-2}\right) K-K+\frac{s(l-n)-(s-n)}{s-2}.$$
A modicum of computation reveals that
\begin{align*}
(s-2)\nu &=n(s-2)+(2-2l+n)K+s(l-n)-(s-n)\\
&=(s-2K)(l-1)+n(K-1).
\end{align*}
Our hypothesis on $s$ ensures that $s\ge 2rk+1\ge 2K+1$, and we may suppose moreover that $l\ge 2$. We therefore deduce that $\nu \ge 1/(s-2)$, and thus
$$\calN_1(\calA)\ll N^{sd-K-1/s}.$$
This completes the proof of our earlier claim, and hence the proof of Theorem \ref{theorem11.3} is now complete.

\bibliographystyle{amsbracket}
\providecommand{\bysame}{\leavevmode\hbox to3em{\hrulefill}\thinspace}

\end{document}